\documentclass[12pt,a4paper]{article}
\usepackage{amsmath}
\usepackage{amssymb}
\usepackage{amsthm}
\usepackage{amsfonts}
\usepackage{latexsym}

\theoremstyle{plain}
\newtheorem{thm}{Theorem}[section]
\newtheorem{cor}[thm]{Corollary}
\newtheorem{lem}[thm]{Lemma}
\newtheorem{prop}[thm]{Proposition}

\theoremstyle{definition}
\newtheorem{defn}[thm]{Definition}
\newtheorem{notn}[thm]{Notational Caveat}

\newtheorem{rem}[thm]{Remark}

\title{Eigenfunction asymptotics in the complex domain for 
a compact Lie group}
\author{Simone Gallivanone and Roberto Paoletti\footnote{\noindent{\bf Address:}
Dipartimento di Matematica e Applicazioni, Universit\`a degli Studi
di Milano Bicocca, Via R. Cozzi 55, 20125 Milano,
Italy; {\bf e-mail}: roberto.paoletti@unimib.it }}
\date{}

\begin{document}
\maketitle

\begin{abstract}
Let $(G,\kappa)$ be a compact connected Lie group
endowed with a biinvariant Riemannian metric, and let $\tilde{G}$ be
the complexification of $G$.
We apply Grauert tube techniques to the near-diagonal scaling 
asymptotics of 
certain operator kernels, which are defined in terms of the matrix elements of an irreducible representation drifting
to infinity along a ray in weight space. These kernels are the equivariant components of Poisson and Szeg\H{o} kernels on a fixed sphere bundle
in $\tilde{G}$, when the latter is identified with the tangent bundle
of $G$ in an appropriate way.
\end{abstract}

\section{Introduction}

Let $G$ be a connected compact Lie group with Lie algebra
of left-invariant vector fields
$\mathfrak{g}$ and coalgebra $\mathfrak{g}^\vee$; 
we shall canonically identify $\mathfrak{g}$  
with the tangent space
of $G$ at the identity $e\in G$, $T_eG$.
The dimension of $G$ and its rank (i.e., the dimension of a maximal torus
$T\leqslant G$)
will be denoted by, respectively, $d$ and $r_G$.

Let $\kappa_e$ be an $\mathrm{Ad}$-invariant Euclidean
product on $\mathfrak{g}$,
where $\mathrm{Ad}$ denotes the adjoint representation;
we shall denote by $\kappa_e^\vee$ the corresponding
Euclidean product on $\mathfrak{g}^\vee$.
Let $\kappa$ be the
induced biinvariant Riemannian metric
on $G$, and let $L^2(G)$ be the 
Hilbert space of $L^2$-summable functions on
$G$ with respect to the associated Riemannian
density $\mathrm{d} V_G$.
Then $G$ is unitarily represented 
on $L^2(G)$ by left translations:
\begin{equation}
    \label{eqn:left translations}
    \theta_l(g)(f)(h):=f\left(g^{-1}\,h\right)
\quad (f\in L^2(G),\,g,h\in G).
\end{equation}

Let $T\leqslant G$ be a maximal torus, with Lie algebra
$\mathfrak{t}$ and corresponding set of roots $R\subset \mathfrak{t}^\vee$. As is well-known,
choosing an ordering of 
$R$, i.e., a set of positive roots $R^+\subset R$,
determines a notion of dominant weight for $(\mathfrak{g},
\mathfrak{t})$; the set
$\hat{G}$ of all irreducible representations of $G$
is then in bijective correspondence with a certain subset
$\mathcal{D}^G$ of the set $\mathcal{D}$
of all dominant weights.
For instance, $\mathcal{D}^G=\mathcal{D}$ when
$G$ is simply connected.
For every $\boldsymbol{\lambda}\in \mathcal{D}^G$,
we shall denote by $V_{\boldsymbol{\lambda}}$ the
corresponding representation space, by
$d_{\boldsymbol{\lambda}}$ its dimension, and by 
$\chi_{\boldsymbol{\lambda}}$ its character.
We shall also denote by
$\boldsymbol{\lambda}^\vee$
the weight associated to the dual representation.

By the Theorem of Peter and Weyl, there is
a unitary and equivariant decomposition of 
$L^2(G)$ as the Hilbert space direct sum
of its isotypical components (see, e.g., \cite{btd},
\cite{s95}). 
More precisely, let $L^2(G)_{\boldsymbol{\lambda}}\subset L^2(G)$
denote the span of the matrix elements of the 
representation corresponding to $\boldsymbol{\lambda}^\vee$,
with respect to any given basis of
$V_{\boldsymbol{\lambda}^\vee}$. Then there is 
an equivariant isomorphism
$$
L^2(G)_{\boldsymbol{\lambda}}\cong 
V_{\boldsymbol{\lambda}}^{\oplus d_{\boldsymbol{\lambda}}};
$$
hence $L^2(G)_{\boldsymbol{\lambda}}$
is the $\boldsymbol{\lambda}$-th isotypical component
of $L^2(G)$, and
there is an equivariant isomorphism
of Hilbert spaces
\begin{equation}
\label{eqn:direct sum decomp}
L^2(G)\cong 
\bigoplus_{\boldsymbol{\lambda}\in \mathcal{D}^G}
L^2(G)_{\boldsymbol{\lambda}}.
\end{equation}

The subspaces $L^2(G)_{\boldsymbol{\lambda}}$ may also be described in terms
of the positive Laplacian operator $\Delta$ on 
$(G,\kappa)$. Namely, $\Delta$ acts on $L^2(G)_{\boldsymbol{\lambda}}$ as 
scalar multiplication by 
$c_{\boldsymbol{\lambda}}^2$, where
\begin{equation}
\label{eqn:eigenvalue of lambda}
c_{\boldsymbol{\lambda}}:=\sqrt{\|\boldsymbol{\lambda}^\vee+\boldsymbol{\delta}\|_\kappa^2
-\|\boldsymbol{\delta}\|_\kappa^2}
=\kappa_e^\vee\left(\boldsymbol{\lambda}^\vee,\boldsymbol{\lambda}^\vee+2\,\boldsymbol{\delta}\right)
^{1/2},
\end{equation}
$\boldsymbol{\delta}$ being the half-sum of positive roots
(see (\ref{eqn:half sum prs})); in other words,
every matrix element of 
$V_{\boldsymbol{\lambda}^\vee}$ is an eigenfunction of $\Delta$ for the eigenvalue 
$c_{\boldsymbol{\lambda}}^2$.

The complexification $(\tilde{G},J)$ of
$G$ is a connected complex $d$-dimensional Lie group, 
with complex structure $J$,
in which $G$ sits as a totally real submanifold and a maximal compact subgroup \cite{btd}; the Lie algebra $\tilde{\mathfrak{g}}$
of $\tilde{G}$ the complexification of $\mathfrak{g}$,
i.e. $\tilde{\mathfrak{g}}=\mathfrak{g}\otimes \mathbb{C}$. 
Any irreducible representation of $G$ extends uniquely to
a holomorphic irreducible representation of $\tilde{G}$
on the same representation space.
Hence, the matrix elements of any irreducible
representation of $G$  admit holomorphic extensions
to $\tilde{G}$.

There is a diffeomorphism  
$\gamma:G\times \mathfrak{g}\cong\tilde{G} $, which can be described as follows.
Consider the
exponential map 
$\exp^{\tilde{G}}:\tilde{\mathfrak{g}}\rightarrow \tilde{G}$.
Then 
\begin{equation}
\label{eqn:trivializationtildeG}
\gamma (g,\boldsymbol{\xi}):=
g\,\exp^{\tilde{G}}(\imath\,\boldsymbol{\xi})\qquad
(g\in G,\,\boldsymbol{\xi}\in \mathfrak{g}).
\end{equation}

In the sequel, 
we shall use the
short-hands
$$
g\cdot \boldsymbol{\xi}=\mathrm{d}_e L_g(\boldsymbol{\xi}),
\quad \boldsymbol{\xi}\cdot g=\mathrm{d}_e R_g(\boldsymbol{\xi})\quad 
(g\in G,\,\boldsymbol{\xi}\in \mathfrak{g}),
$$
where $L_g,\,R_g:G\rightarrow G$ denote left and right translations by
$g\in G$, respectively.
Let us compose $\gamma$ with the trivialization of the tangent bundle
$TG$ given by left translations to the identity. 
If $v\in T_gG$,
we have $v=g\cdot \boldsymbol{\xi}$ for a unique
$\boldsymbol{\xi}\in \mathfrak{g}$. 
We obtain a diffeomorphism
\begin{equation}
\label{eqn:defn of gamma'}
E:(g,v)=(g,g\cdot \boldsymbol{\xi})\in TG\mapsto 
g\,\exp^{\tilde{G}}(\imath\,\boldsymbol{\xi})\in \tilde{G}.
\end{equation}

Besides being a complex Lie group, $\tilde{G}$ also
carries a built-in $G$-invariant K\"{a}hler structure,
which is intrinsically determined by $\kappa$, and
has an invariant global potential $\rho$ (see (\ref{eqn:defn di rho})
below). 
In terms of $\rho$, one can define a family of
pseudoconvex boundaries $X^\tau$ in $\tilde{G}$, and
we shall be concerned with certain operator kernels
on $X^\tau$.
To clarify this, it is in order to make a brief recall
on so-called \textit{Grauert tubes} 
of compact real-analytic Riemannian manifolds
(such as $(G,\kappa)$).

Let $(M,\beta)$ be a compact real-analytic Riemannian manifold,
and for $\tau>0$ let $T^\tau M$ be the tubular neighbourhood of 
radius $\tau$ 
of the zero section in the tangent
bundle. For sufficiently small $\tau$, 
there exists an intrinsic
so-call \textit{adapted} complex structure
$J_{\mathrm{ad}}$ on $T^\tau M$, which is uniquely characterized by the property that the standard parametrization  of the leaves of the Riemann foliation are holomorphic  maps
from suitable strips in $\mathbb{C}$ (\cite{pw1},
\cite{pw2}, \cite{gs1}, \cite{gs2}, \cite{ls}, \cite{szo91}; see also \cite{b}, \cite{bh}). Furthermore, $J_{\mathrm{ad}}$
is compatible with the canonical
symplectic structure $\Omega_{\mathrm{can}}$
on $TM$, and the associated Riemannian metric restricts to $\beta$
along $M$. 

For a general $(M,\beta)$, $J_{ad}$ needn't 
be defined on the whole of $TM$: as proved in 
\cite{ls},
this will certainly fail in the presence of negative scalar curvatures.
In the present setting, however, the following holds.

\begin{thm}
\label{thm:szoke compgr}
(Sz\"{o}ke, \cite{szo98}))
Let $(G,\kappa)$ be a compact Lie group with a biinvariant Riemannian metric.
Then $J_{ad}$ is defined on $TG$, and $E$ in
(\ref{eqn:defn of gamma'}) is a biholomorphism from $(TG,J_{ad})$ to
$(\tilde{G},J)$.

\end{thm}

This result puts the theory of complex reductive groups
in contact with the general theory of Grauert tubes.

On the other hand, the asymptotics of Poisson and Szeg\H{o} kernels
on Grauert tube boundaries have attracted considerable
attention in recent years, 
sparked to a large extent by seminal insight of
Boutet de Monvel on the extension properties of
Laplacian eigenfunctions
(see e.g. \cite{bdm1}, \cite{cr1}, \cite{cr2},
\cite{gls},
\cite{g},
\cite{leb}, 
\cite{s1}, \cite{s2},
\cite{z12}, \cite{z13}, \cite{z20}).
Grauert tube techniques also find applications in the
study of nodal sets (see e.g. 
\cite{ct16}, \cite{ct18}, \cite{tz09}, \cite{tz21}).

In light of these considerations, 
it seems natural to apply Grauert tube methods to 
the study the local and global 
asymptotics related to matrix elements
of irrducible representations.
To bring this theme into focus,
let us remark that Theorem 
\ref{thm:szoke compgr} has the following consequences, which are the basis for the present discussion:
\begin{enumerate}
\item $(TG,J_{\mathrm{ad}},\Omega_{\mathrm{can}})$
is a K\"{a}hler manifold.
\item If
$
\Omega:=(E^{-1})^*(\Omega_{\mathrm{can}})
$,
then $(\tilde{G},J,\Omega)$ is also a K\"{a}hler manifold,
and $E$ is an isomorphism of K\"{a}hler manifolds.

\item The induced Riemannian metric 
$\hat{\kappa}(\cdot,\cdot):=\Omega(\cdot, J(\cdot))$ on $\tilde{G}$ restricts
to $\kappa$ along $G$.

\item Consider the norm function $\|\cdot\|_\kappa$ on $TG$, which 
is the pull-back of the norm $\|\cdot\|_{\kappa_e}$ 
on $\mathfrak{g}$ under the previous trivialization:
$$
\|(g,g\cdot \boldsymbol{\xi})\|_\kappa:=
\|g\cdot \boldsymbol{\xi}\|_{\kappa_g}
=\|\boldsymbol{\xi}\|_{\kappa_e};
$$
then $\|\cdot\|_\kappa^2$ is strictly plurisubharmonic 
with respect to $J_{ad}$, and is in fact a global
K\"{a}hler potential for $\Omega_{\mathrm{can}}$. 

\item Therefore, the composition
\begin{equation}
\label{eqn:defn di rho}
\rho:=\|\cdot\|_\kappa^2\circ E^{-1}:
\tilde{G}\rightarrow [0,+\infty)
\end{equation}
is strictly plurisubharmonic  on $(\tilde{G},J)$, and a global
K\"{a}hler potential for $\Omega$:
$$
\Omega=\imath\,\partial\,\overline{\partial}\rho.
$$

\item $\sqrt{\rho}=\|\cdot\|_\kappa\circ E^{-1}$
(a.k.a. the \textit{tube function} on $\tilde{G}$) restricts on $\tilde{G}\setminus G$ to a solution of the homogeneous Monge-Amp\`{e}re
equation.

\item For $\tau>0$, let 
$S^\tau(\mathfrak{g})\subset \mathfrak{g}$
denote the sphere of radius $\tau$ centered at the origin.
Then,
with $\gamma$ as in (\ref{eqn:trivializationtildeG})),
\begin{equation}
\label{eqn:X tau sphere bundle}
X^\tau:=\rho^{-1}(\tau^2)=\gamma\left(G\times S^\tau(\mathfrak{g})\right)
\subset \tilde{G}
\end{equation}
is the boundary of a strictly pseudoconvex domain, and as such is
equipped with a natural contact structure and an induced 
volume form; we shall let $H(X^\tau)
\subset L^2(X^\tau)$
denote its Hardy space, and $\Pi^\tau:L^2(X^\tau)\rightarrow H(X^\tau)$ the corresponding
Szeg\"{o} kernel (orthogonal projector).

\end{enumerate}

Clearly, $E^{-1}(X^\tau)\subset TG$ 
is the sphere bundle of radius $\tau$ (see (\ref{eqn:defn of gamma'})).

Let
$\tilde{L}:G\times \tilde{G}\rightarrow \tilde{G}$ 
be the holomorphic action 
given by left translations. For any $g\in G$,
$\rho$ is 
$\tilde{L}_g$-invariant, and
$\tilde{L}_g$ is a K\"{a}hler automorphism of $(\tilde{G},J,\Omega)$.
Therefore, 
$\tilde{L}$ restricts for any $\tau>0$ 
to a CR-holomorphic action 
\begin{equation}
\label{eqn:mutau defn}
\mu^\tau:=
\left.\tilde{L}\right|_{G\times X^\tau}:
G\times X^\tau
\rightarrow X^\tau
\end{equation}
(that is, $\mu^\tau_g:X^\tau\rightarrow X^\tau$
is given by restriction of 
$\tilde{L}_g:\tilde{G}\rightarrow \tilde{G}$,
$\forall\,g\in G$).
Then $\mu^\tau$ induces 
in a standard manner
a unitary
representation of $G$ on $H(X^\tau)$. By the Theorem
of Peter and Weyl, 
there is 
a unitary and equivariant decomposition as a Hilbert
space direct sum
\begin{equation}
\label{eqn:hardy space decomp equiv}
H(X^\tau)=\bigoplus_{\boldsymbol{\lambda}\in \mathcal{D}^G}
H(X^\tau)_{\boldsymbol{\lambda}},
\end{equation}
where $H(X^\tau)_{\boldsymbol{\lambda}}\subset
H(X^\tau)$ is the $\boldsymbol{\lambda}$-th isotypical
decomposition. 
The equivariant 
decompositions (\ref{eqn:direct sum decomp}) and  (\ref{eqn:hardy space decomp equiv})
are related by the following property:
for any $\boldsymbol{\lambda}$,
$H(X^\tau)_{\boldsymbol{\lambda}}$ consists of the
holomorphic extensions of elements of $L^2(G)_{\boldsymbol{\lambda}}$,
restricted to $X^\tau$. Hence there is an algebraic
(but non-unitary)
isomorphism $L^2(G)_{\boldsymbol{\lambda}}\rightarrow 
H(X^\tau)_{\boldsymbol{\lambda}}$, given by holomorphic extension and restriction.

There are two natural smoothing 
operator kernels associated to this
setting. One relates to the CR and
metric structure of $X^\tau$, and is the kernel of the
orthogonal projector 
\begin{equation}
\label{eqn:equiv szego proj}
\Pi^\tau_{\boldsymbol{\lambda}}:L^2(X^\tau)\rightarrow
H(X^\tau)_{\boldsymbol{\lambda}},
\end{equation}
i.e. the $\boldsymbol{\lambda}$-th equivariant piece
of the Szeg\"{o} kernel $\Pi^\tau$. If 
$(
\sigma_{\boldsymbol{\lambda},j}
)_{j=1}^{d_{\boldsymbol{\lambda}}^2}$ is an orthonormal basis of $H(X^\tau)_{\boldsymbol{\lambda}}$, then 
the distributional kernel of (\ref{eqn:equiv szego proj})
is
\begin{equation}
\label{eqn:equiv szego ker}
\Pi^\tau_{\boldsymbol{\lambda}}(x,y)=
\sum_{j=1}^{d_{\boldsymbol{\lambda}}^2}
\sigma_{\boldsymbol{\lambda},j}(x)\cdot 
\overline{\sigma_{\boldsymbol{\lambda},j}(y)}
\qquad (x,y\in X^\tau).
\end{equation}
The other operator kernel in point is related to the
holomorphic extension property of the matrix elements
of the irreducible representations of $G$, or equivalently of the
eigenfunctions of the Laplacian of $G$. Let 
$(
\sigma_{\boldsymbol{\lambda},j}
)_{j=1}^{d_{\boldsymbol{\lambda}}^2}$
be an orthonormal basis of $L^2(G)_{\boldsymbol{\lambda}}$;
for any $j$, let $\tilde{\sigma}_{\boldsymbol{\lambda},j}$
be the holomorphic extension of 
$\sigma_{\boldsymbol{\lambda},j}$ to $\tilde{G}$, and
denote by $\tilde{\sigma}_{\boldsymbol{\lambda},j}^\tau$
its restriction to $X^\tau$. Thus 
$(\tilde{\sigma}_{\boldsymbol{\lambda},j}^\tau)_{j=1}^{d_{\boldsymbol{\lambda}}^2}$ is also a basis of 
$H(X^\tau)_{\boldsymbol{\lambda}}$, albeit not an orthonormal
one.
The smoothing operator kernel related 
the complexified eigenfunctions 
for the eigenvalue $c_{\boldsymbol{\lambda}}$
(equivalently, of the matrix elements of $V_{\boldsymbol{\lambda}^\vee}$)
is then
\begin{equation}
\label{eqn:Poisson-wave equiv}
P^\tau_{\boldsymbol{\lambda}}(x,y):=
e^{-2\,\tau\,c_{\boldsymbol{\lambda}}}\,
\sum_j
\tilde{\sigma}_{\boldsymbol{\lambda},j}^\tau(x)
\cdot \overline{\tilde{\sigma}_{\boldsymbol{\lambda},j}^\tau(y)}
\qquad (x,y\in X^\tau),
\end{equation}
where $c_{\boldsymbol{\lambda}}$ is as in (\ref{eqn:eigenvalue of lambda}).

Let us briefly dwell to motivate the tempering factor 
$e^{-2\,\tau\,c_{\boldsymbol{\lambda}}}$.
As discussed, say, in \cite{z12}, 
\cite{z14} and \cite{z20},
the operator $P^\tau:\mathcal{C}^\infty(X^\tau)\rightarrow
\mathcal{C}^\infty(X^\tau)$ with Schwartz kernel
\begin{equation}
\label{eqn:Poisson-wave}
P^\tau(x,y):=\sum_{\boldsymbol{\lambda}\in \mathcal{D}^G}
e^{-2\,\tau\,c_{\boldsymbol{\lambda}}}\,
\sum_j
\tilde{\sigma}_{\boldsymbol{\lambda},j}^\tau(x)
\cdot \overline{\tilde{\sigma}_{\boldsymbol{\lambda},j}^\tau(y)}
\qquad (x,y\in X^\tau)
\end{equation}
is a Fourier integral
operator of degree $-(d-1)/2$
with the same canonical relation as $\Pi^\tau$.
Furthermore, $P^\tau$ is closely related to the so-called Poisson-wave operator, which is
obtained by holomorphically extending the kernel
of the wave operator on $G$ and
governs the holomorphic extension of Laplacian eigenfunctions.
For this reason, (\ref{eqn:Poisson-wave}) plays a
crucial role in the asymptotic study of complexified eigenfunction of compact real-analytic Riemannian manifolds
(see also the surveys \cite{z13} and \cite{z17}).
On the other hand, (\ref{eqn:Poisson-wave equiv}) is the
$\boldsymbol{\lambda}$-th component of $P^\tau$; 
hence it is a natural candidate for the asymptotic study of
complexified isotypical eigenfunctions.

In this paper, we shall be concerned with the near-diagonal
asymptotics of (\ref{eqn:equiv szego ker}) and 
(\ref{eqn:Poisson-wave equiv}), when the weight
$\boldsymbol{\lambda}$ drifts to infinity along
a ray in weight space. 
In other words, for a fixed non-zero
$\boldsymbol{\lambda}\in \mathcal{D}^G$, 
we restrict attention to the
\textit{ladder} of representations 
$V_{k\,\boldsymbol{\lambda}}$, 
$k=1,2,\ldots$ (see \cite{gstb82}),
and consider the asymptotics of 
the sequence of smoothing kernels 
$\Pi^\tau_{k\,\boldsymbol{\lambda}}$ and 
$P^\tau_{k\,\boldsymbol{\lambda}}$ when $k\rightarrow +\infty$.
Analogues of these asymptotics in the line bundle setting were considered in \cite{gp19}, \cite{gp20},
and \cite{p22}. In particular, in spite of the rather different
geometric context, the approach in
\cite{p22} is close to the one in the present work,
and the statements bear a formal similarity to
those in \textit{loc. cit.} (see the closing remark of this
introduction).

In order to explain our first result in this direction, 
we need to premise some further notation. 
The isomorphism 
$\mathcal{L}:\mathfrak{g}\rightarrow \mathfrak{g}^\vee$
induced by $\kappa_e$ intertwines
the adjoint and coadjoint representations
$\mathrm{Ad}$ and $\mathrm{Coad}$ of $G$.
Hence the inverse image under 
$\mathcal{L}$ of a coadjoint orbit
$\mathcal{O} \subset \mathfrak{g}^\vee$ is an
adjoint orbit $\tilde{\mathcal{O}} \subset \mathfrak{g}$.
Furthermore, since 
$\mathrm{Ad}$ and $\mathrm{Coad}$
are unitary, 
if non trivial
$\mathcal{O}$ and $\tilde{\mathcal{O}}$
lie in spheres centered at the origin in $\mathfrak{g}^\vee$
and $\mathfrak{g}$, respectively.
For any $\tau>0$, we shall denote by $\mathcal{O}^{\tau}$
(respectively, $\tilde{\mathcal{O}}^{\tau}$) the
unique rescaling of $\mathcal{O}$ (respectively, $\tilde{\mathcal{O}}$) contained in the sphere
of radius $\tau$ centered at the origin 
in $\mathfrak{g}^\vee$
(respectively, $\mathfrak{g}$); thus 
$\mathcal{L}( \tilde{\mathcal{O}}^{\tau})=\mathcal{O}^{\tau}$.

\begin{defn}
\label{defn:defn di XtauO}
Given a non-zero $\boldsymbol{\lambda}\in \mathcal{D}^G$, 
let $\mathcal{O} \subset \mathfrak{g}^\vee$ be its
coadjoint orbit, and let $\mathcal{C}(\mathcal{O})\subseteq
\mathfrak{g}^\vee$ be the positive cone over $\mathcal{O}$.

\begin{enumerate}
    \item With $\gamma$ is as in (\ref{eqn:trivializationtildeG}), let us set
    $$
    \tilde{G}_{\mathcal{O}}:=
\gamma\left(G\times \mathcal{C}(\tilde{\mathcal{O}})\right).
    $$
    \item For any $\tau>0$, let us define
$$
X^\tau_{\mathcal{O}}:=
\gamma (G\times \tilde{\mathcal{O}}^\tau)
\subseteq X^\tau\subset \tilde{G}.
$$
\end{enumerate}
\end{defn}

Clearly, $X^\tau$ and $X^\tau_{\mathcal{O}}$ are 
$G$-invariant.
If $x\in \tilde{G}$, let $G\cdot x \subset \tilde{G}$
denote its $G$-orbit under left translations.

\begin{defn}
\label{eqn:defn di ZtauO}
In the same setting as in Definition
\ref{defn:defn di XtauO}, let us pose
$$ 
\mathcal{Z}^\tau_{\mathcal{O}}:=
\left\{(x,y)\in X^\tau\times X^\tau\,:\,
x\in X^\tau_{\mathcal{O}}\quad\text{and}\quad
y\in G\cdot x\right\}.
$$
\end{defn}

\begin{thm}
\label{thm:rapid decay compact}
Uniformly on compact subsets of
$X^\tau\times X^\tau\setminus \mathcal{Z}^\tau_{\mathcal{O}}$,
one has
$$
\Pi^\tau_{k\,\boldsymbol{\lambda}}(x,y)=
O\left(k^{-\infty}\right)\qquad
\text{and}\qquad P^\tau_{k\,\boldsymbol{\lambda}}(x,y)=
O\left(k^{-\infty}\right)
$$
for $k\rightarrow +\infty$.

\end{thm}

We next aim to show, by analogy with the results in \cite{p22},
that rapid decay
may be established at pairs 
$(x,y)\in X^\tau\times X^\tau\setminus \mathcal{Z}^\tau_{\mathcal{O}}$
approaching $\mathcal{Z}^\tau_{\mathcal{O}}$ at a controlled pace as $k\rightarrow +\infty$.

A first manifestation of this is the following result,
which deals with pairs $(x,y)$ with $y\rightarrow G\cdot x$
sufficiently slowly in $k$ (with a restriction on
$\mathcal{O}$). 

Let $\mathrm{dist}_{X^\tau}$ denote the Riemannian distance 
on $X^\tau$, and let $\mathfrak{t}^0\subset\mathfrak{g}^\vee$ denote the annhilator
of $\mathfrak{t}$.

\begin{thm}
\label{thm:rapid decay orbit}
Suppose that $\mathcal{O}\cap \mathfrak{t}^0=\emptyset$.
Let $C,\,\epsilon>0$ be fixed. 
Then, uniformly for 
\begin{equation}
\label{eqn:bd distance orbits}
\mathrm{dist}_{X^\tau}(y,G\cdot x)\ge C\,k^{\epsilon-\frac{1}{2}},
\end{equation}
one has
$$
\Pi^\tau_{k\,\boldsymbol{\lambda}}(x,y)=
O\left(k^{-\infty}\right)\qquad
\text{and}\qquad P^\tau_{k\,\boldsymbol{\lambda}}(x,y)=
O\left(k^{-\infty}\right)
$$
for $k\rightarrow +\infty$.
\end{thm}

For instance, the previous assumption on $\mathcal{O}$ is satisfied when
$G=U(n)$ and the matrices in $\tilde{\mathcal{O}}$ have non-zero trace.
On the other hand, for $G=SU(2)$
the conclusion of Theorem \ref{thm:rapid decay orbit} may 
be proved by adapting the \textit{ad hoc}
argument for
Theorem 1.1 in \cite{gp20}.
Thus it seems natural to expect that the result should hold in
greater generality.

The weight $\boldsymbol{\lambda}$ is called 
\textit{regular} if its coadjoint orbit $\mathcal{O}$ has
maximal dimension, or equivalently if its
stabilizer is $T$.

\begin{thm}
\label{thm:rapid decay moment map}
In the situation of Theorem \ref{thm:rapid decay orbit},
let us assume
that $\boldsymbol{\lambda}$ is regular.
Let us a fix constants $C,\,\epsilon>0$. Then, uniformly on
\begin{equation}
\label{eqn:bound distance momentO}
\left\{
(x,y)\in X^\tau\times X^\tau\,:\,\mathrm{dist}_{X^\tau}
\left(x,X^\tau_{\mathcal{O}}\right)
%
\ge \,C\,k^{\epsilon-\frac{1}{2}}\right\}
\end{equation}
one has
$$
\Pi^\tau_{k\,\boldsymbol{\lambda}}(x,y)=
O\left(k^{-\infty}\right)\qquad
\text{and}\qquad P^\tau_{k\,\boldsymbol{\lambda}}(x,y)=
O\left(k^{-\infty}\right)
$$
for $k\rightarrow +\infty$.

\end{thm}

Next we consider near-diagonal scaling asymptotics 
for $\Pi^\tau_{k\,\boldsymbol{\lambda}}$ along $X^\tau_{\mathcal{O}}$.
To formulate this, we need to describe how 
the tangent space of $X^\tau$ at a given $x\in X^\tau_{\mathcal{O}}$
decomposes in terms of the local
geometry. 

Let $\jmath^\tau:X^\tau\hookrightarrow \tilde{G}$ be the inclusion,
and set $\alpha^\tau:={\jmath^\tau}^*(\alpha)$; thus 
$(X^\tau,\alpha^\tau)$ is a contact manifold, and $\mathcal{H}_x^\tau:= \ker(\alpha^\tau_x)$ is the
maximal complex subspace of $T_xX^\tau$
(and has complex dimension $d-1$).
Let $\mathcal{R}^\tau$ be the Reeb vector field of 
$(X^\tau,\alpha^\tau)$.
Then
\begin{equation}
\label{eqn:general decomp}
T_xX^\tau=\mathrm{span}_{\mathbb{R}}\big(\mathcal{R}^\tau(x)\big)
\oplus_{\hat{\kappa}_x} \mathcal{H}_x^\tau,
\end{equation}
where $\oplus_{\hat{\kappa}_x}$ denotes an orthogonal direct sum for
$\hat{\kappa}_x$; (\ref{eqn:general decomp}) holds at any $x\in X^\tau$.

Assuming $x\in X^\tau_{\mathcal{O}}$, let 
$N(\tilde{G}_{\mathcal{O}}/\tilde{G})_x\subseteq T_x\tilde{G}$ be the normal 
space to $\tilde{G}_{\mathcal{O}}$ in $\tilde{G}$ at $x$.
Its (real) dimension is $r_{\mathcal{O}}-1$, where 
$r_{\mathcal{O}}$ is the dimension of the stabilizer of any
element of $\mathcal{O}$ (that is, $r_{\mathcal{O}}$
is the codimension of $\mathcal{O}$ in $\mathfrak{g}^\vee$).

By the transversality of $\tilde{G}_{\mathcal{O}}$ and $X^\tau$, $N(\tilde{G}_{\mathcal{O}}/\tilde{G})_x\subseteq T_xX^\tau$
is also the normal space to 
$X^\tau_{\mathcal{O}}=X^\tau\cap \tilde{G}_{\mathcal{O}}$ in $X^\tau$, that is, 
$$
N(\tilde{G}_{\mathcal{O}}/\tilde{G})_x=
N(X^\tau_{\mathcal{O}}/X^\tau)_x.$$
As discussed in \S \ref{sctn:normal bundle}, 
$N(X^\tau_{\mathcal{O}}/X^\tau)_x\subset \mathcal{H}_x$, and 
furthermore $N(X^\tau_{\mathcal{O}}/X^\tau)_x$ is $\hat{\kappa}_x$-orthogonal
to $J_x\big(N(X^\tau_{\mathcal{O}}/X^\tau)_x \big)$.
Let us consider the complex vector subspace
\begin{equation}
\label{eqn:direct sum Nx}
\mathcal{N}_x:=
N(X^\tau_{\mathcal{O}}/X^\tau)_x\oplus_{\hat{\kappa}_x} J_x\big(N(X^\tau_{\mathcal{O}}/X^\tau)_x\big)\subseteq \mathcal{H}_x,
\end{equation}
and let  $\mathcal{S}_x$
be its orthocomplement in $\mathcal{H}_x$. There 
is a direct sum of complex vector spaces
\begin{equation}
\label{eqn:dsdecomp HSN}
\mathcal{H}_x=\mathcal{S}_x\oplus_{\hat{\kappa}_x}\mathcal{N}_x,
\quad\text{where}
\quad \dim_{\mathbb{C}}(\mathcal{N}_x)=r_{\mathcal{O}}-1,\,
\dim_{\mathbb{C}}(\mathcal{S}_x)=d-r_{\mathcal{O}}.
\end{equation}

Then:
\begin{enumerate}
\item $\tilde{G}_{\mathcal{O}}$ is a coisotropic submanifold 
of $(\tilde{G},\Omega)$;
\item $J_x\big(N(X^\tau_{\mathcal{O}}/X^\tau)_x\big)$ 
is the tangent space to the leaf 
through $x$ of the null foliation of
$\tilde{G}_{\mathcal{O}}$;
\item the $\hat{\kappa}_x$-orthocomplement
of $J_x\big(N(X^\tau_{\mathcal{O}}/X^\tau)_x\big)$
in $T_xX^\tau$ is
\begin{equation}
\label{eqn:normal to leaf}
J_x\big(N(X^\tau_{\mathcal{O}}/X^\tau)_x\big)^{\perp_{\kappa_x}}
=\mathrm{span}_{\mathbb{R}}\big(\mathcal{R}^\tau(x)\big)
\oplus_{\hat{\kappa}_x}
N(X^\tau_{\mathcal{O}}/X^\tau)_x \oplus_{\hat{\kappa}_x}\mathcal{S}_x;
\end{equation}
\item there is an orthogonal direct sum of real vector spaces
\begin{equation}
\label{eqn:direct sum decmp TxX}
T_xX^\tau_{\mathcal{O}}=
\mathrm{span}_{\mathbb{R}}\big(\mathcal{R}^\tau(x)\big)
\oplus _{\hat{\kappa}_x}
J_x\big(N(X^\tau_{\mathcal{O}}/X^\tau)_x\big)\oplus_{\hat{\kappa}_x}
\mathcal{S}_x.
\end{equation}

\end{enumerate}
We shall consider scaling asymptotics along directions normal to
the null foliations, that is in (\ref{eqn:normal to leaf}).

Furthermore, these scaling asymptotics will be formulated in suitable
sets of local coordinates on $X^\tau$ centered at $x$, 
adapted to the pseudoconvex geometry of $X^\tau$,
and called
\textit{normal Heisenberg local coordinates} (NHLC's).
We refer to \cite{p24} for a detailed discussion building on 
\cite{cr1}, \cite{cr2}, \cite{sz}, \cite{fs1} and \cite{fs2}; the introduction of this type
of coordinates in pseudoconvex geometry goes back to Folland and Stein, and (to the best of our knowledge) they have first been to use in the specific
Grauert tube setting by Chang and Rabinowitz.
NHLC's at $x$ on $X^\tau$ are constructed by first introducing 
NHLC's on $\tilde{G}$, and then \lq projecting and restricting\rq. 

In NHLC's on $\tilde{G}$, the local defining equation for $X^\tau$ admits an especially 
simple canonical form,
and this allows for a relatively explicit approximation of the
metric, CR and symplectic structures, and ultimatively of
the phase of the Szeg\"{o} and Poisson kernels (see \S 
\ref{scnt:szego parametrix}).

NHLC's on $X^\tau$ centered at $x$ will be denoted by
a pair $(\theta,\mathbf{v})\in \mathbb{R}\times \mathbb{R}^{2d-2}$
of suitably small norm. 
When convenient, we shall identify
$\mathbb{R}^{2d-2}\cong \mathbb{C}^{d-1}$. 
A point $x'$ close to
$x$ having NHLC's $(\theta,\mathbf{v})$ will be denoted 
by the additive
notation $x'=x+(\theta,\mathbf{v})$.
Then
\begin{equation}
\label{eqn:contact theta}
\left.\frac{\partial}{\partial \theta}\right|_x=
\mathcal{R}^\tau(x),\qquad
\left.\frac{\partial}{\partial \mathbf{v}}\right|_x\in \mathcal{H}_x,
\end{equation}
where $\partial/\partial \mathbf{v}$ denotes the directional
derivative along $\mathbf{v}\in \mathbb{R}^{2d-2}$.

For $\theta=0$, we shall generally write
$x+\mathbf{v}$ for $x+(0,\mathbf{v})$.

It follows that NHLC's at $x\in X^\tau$ determine an isomorphism
$\mathbb{R}\times \mathbb{R}^{2d-2}\cong T_xX^\tau$, with
$\mathbb{R}\times \{\mathbf{0}\}$ (and 
$\{0\}\times \mathbb{R}^{2d-2}$) mapping to, respectively, 
$\mathrm{span}_{\mathbb{R}}\big(\mathcal{R}^\tau(x)\big)$
and $\mathcal{H}_x$.
Furthermore, with the identification
$\{0\}\times \mathbb{R}^{2d-2}\cong \mathbb{C}^{d-1}$
we obtain an isomorphism of complex vector spaces
$\mathbb{C}^{d-1}\cong \mathcal{H}_x$.

Therefore, pulling back to $\mathbb{C}^{d-1}$ the direct sum decomposition 
(\ref{eqn:dsdecomp HSN})
determines a corresponding orthogonal complex direct sum decomposition
$$
\mathbb{C}^{d-1}= \mathbb{C}^{r_{\mathcal{O}}-1}_{\mathcal{N}}\oplus
\mathbb{C}^{d-r_{\mathcal{O}}}_{\mathcal{S}}
=\left[\mathbb{R}^{r_{\mathcal{O}}-1}_{N}\oplus 
\mathbb{R}^{r_{\mathcal{O}}-1}_{JN}\right]
\oplus \mathbb{C}^{d-r_{\mathcal{O}}}_{\mathcal{S}},
$$
where $\mathbb{C}^{r_{\mathcal{O}}-1}_{\mathcal{N}}$
and $\mathbb{C}^{d-r_{\mathcal{O}}}_{\mathcal{S}}$ are, respectively, 
the inverse image
of $\mathcal{N}_x$ and $\mathcal{S}_x$; in the second equality, 
we have further used the splitting (\ref{eqn:direct sum Nx}) to define the
real summands $\mathbb{R}^{r_{\mathcal{O}}-1}_{N}
\subseteq \mathbb{C}^{r_{\mathcal{O}}-1}_{\mathcal{N}}
$ (corresponding to
$N(X^\tau_{\mathcal{O}}/X^\tau)_x\subseteq \mathcal{N}_x$) and $\mathbb{R}^{r_{\mathcal{O}}-1}_{JN}\subseteq \mathbb{C}^{r_{\mathcal{O}}-1}_{\mathcal{N}}$
(corresponding to $J_x\big(N(X^\tau_{\mathcal{O}}/X^\tau)_x\big)
\subseteq \mathcal{N}_x$) in a similar manner.

We shall focus on near-diagonal scaling asymptotics on the scale
of $k^{\epsilon-1/2}$ along directions normal to the null foliation
of $\tilde{G}_{\tilde{O}}$, thus at pairs of points 
$(x_{1k},x_{2k})$
near $(x,x)$ (for a given $x\in X^\tau_{\mathcal{O}}$) of the form
\begin{equation}
\label{eqn:defn di xjk}
x_{j,k}:=x+\left(\frac{\theta_j}{\sqrt{k}},
\frac{\mathbf{n}_j+\mathbf{s}_j}{\sqrt{k}}  \right)
\qquad (j=1,2),
\end{equation}
where 
\begin{equation}
\label{eqn:defn scaling vectors}
(\theta_j,\mathbf{n}_j,\mathbf{s}_j)\in \mathbb{R}\times
\mathbb{R}^{r_{\mathcal{O}}-1}_{N}\times \mathbb{C}^{d-r_{\mathcal{O}}}_{\mathcal{S}}
\cong \mathbb{R}\cdot \mathcal{R}_x\times N(X^\tau_{\mathcal{O}}/X^\tau)_x\times 
\mathcal{S}_x
\end{equation}
has norm $O\left(k^\epsilon\right)$.

If  
$\mathcal{O}=\mathcal{O}_{\boldsymbol{\lambda}}$ is the coadjoint orbit of
a regular weight, then $r_{\mathcal{O}}=r_G$.

In order to state the following Theorem, we need to
define a certain number of local and global invariants
of our geometric setting.

\begin{defn}
Let $(V,h)$ be a complex $d$-dimensional Hermitian vector space,
and set $\varphi:=\Re(h)$ (an Euclidean scalar product on $V$)
and $\gamma=-\Im(h)$ (a symplectic bilinear form on
$V$), so that
$h=\varphi-\imath\,\gamma$. 
Let us define 
$\psi_2^h:V\times \times V\rightarrow \mathbb{C}$ by
setting
$$
\psi_2^h(v,w):=-\imath\,\gamma(v,w)-\frac{1}{2}\,
\|v-w\|_\varphi^2,
$$
where $\|v \|_\varphi:=\varphi(v,v)^{1/2}$; 
we shall equivalently write $\psi_2^h=
\psi_2^\varphi=\psi_2^\gamma$.
If $(V,h)=(\mathbb{C}^d,h_{st})$ (the standard Hermitian
product), we shall write $\psi_2=\psi_2^{h_{st}}$.

\end{defn}

\begin{defn}
Let us adopt the following notation:
\begin{enumerate}
\item $\mathrm{vol}(\mathcal{O})$ is the 
symplectic volume of the coadjoint orbit $\mathcal{O}
=\mathcal{O}_{\boldsymbol{\lambda}}$ through 
$\boldsymbol{\lambda}$ (for the Kirillov-Kostant-Souriau symplectic form);
\item $\mathrm{vol}^k(G)$ and $\mathrm{vol}^\kappa (T)$
are the volumes of $G$ and $T$, respectively, for the 
bi-invariant Riemannian structures associated to $\kappa$;
\item If $\boldsymbol{\lambda}$ is a regular weight
and $\boldsymbol{\lambda}^\kappa\in \mathfrak{t}$ corresponds to
$\boldsymbol{\lambda}$ under $\kappa_e$, then the
endomorphism
\begin{equation}
\label{eqn:inf adoint iso}
S_{\boldsymbol{\lambda}}:=\left.
\mathrm{ad}_{\boldsymbol{\lambda}^\kappa}
\right |_{\mathfrak{t}^{\perp_{\kappa_e}}}:\boldsymbol{\rho}\in 
\mathfrak{t}^{\perp_{\kappa_e}}\mapsto 
\left[\boldsymbol{\lambda}^\kappa , \boldsymbol{\rho}\right]
\in \mathfrak{t}^{\perp_{\kappa_e}}
\end{equation}
is a skew-symmetric automorphism. Thus 
$\mathfrak{d}_{\boldsymbol{\lambda}}:=\det(S_{\boldsymbol{\lambda}})
>0$.
\item Let $\omega:=\frac{1}{2}\,\Omega$,
so that $(\tilde{G},J,\omega)$ is K\"{a}hler manifold
with associated Riemannian metric $\tilde{\kappa}:=
\frac{1}{2}\,\hat{\kappa}$.
\item Given $x\in \tilde{G}$, let $\mathrm{val}_x:
\mathfrak{g}\rightarrow T_x\tilde{G}$ be the map
$\boldsymbol{\xi}\mapsto \boldsymbol{\xi}_{\tilde{G}}(x)$,
where $\boldsymbol{\xi}_{\tilde{G}}\in \mathfrak{X}(\tilde{G})$
is the vector field induced by $\boldsymbol{\xi}$ under 
the action $\tilde{L}$; then $\mathrm{val}_x^t(\tilde{\kappa}_x)$ is an
Euclidean scalar product on $\mathfrak{g}$. 

\item Suppose that $x\in \tilde{G}\setminus G$ 
is such that $\Phi(x)\in \mathfrak{g}^\vee$
is regular. 
Let $T_x$ be the stabilizer of 
$\Phi(x)$ (a maximal torus) and $\mathfrak{t}_x\subseteq \mathfrak{g}$
its Lie algebra. Let
$\mathfrak{t}_x':=\mathfrak{t}_x\cap \Phi(x)^0$
(a hyperplane in $\mathfrak{t}_x$). Let $\mathfrak{B}_x$
be an orthonormal basis of $\mathfrak{t}_x'$ for $\kappa_e$,
and let $D_x$ be the matrix representing the
restriction of $\mathrm{val}_x^t(\tilde{\kappa}_x)$ to $\mathfrak{t}_x'$
with respect to $\mathcal{B}_x$.
Then $\det(D_x)$ only depends on $x$, and we may set
$$
\mathfrak{D}^\kappa(x):=\sqrt{\det(D_x)}.
$$
\end{enumerate}

\end{defn}

\begin{thm}
\label{thm:rescaled asympt}
Let us fix constants $C>0$ and $\epsilon\in (0,1/6)$.
Under the assumptions of Theorem 
\ref{thm:rapid decay moment map}, consider $x\in X^\tau_{\mathcal{O}}$ and choose a system
of NHLC at $x$. Let us define 
$x_{jk}$ as in (\ref{eqn:defn di xjk})
and (\ref{eqn:defn scaling vectors}).
Then, uniformly for 
$\|(\theta_j,\mathbf{n}_j,\mathbf{s}_j)\|
\le C\,k^{\epsilon-1/2}$, the following asymptotic expansions
hold for $k\rightarrow +\infty$:
\begin{eqnarray*}
\Pi^\tau_{k\,\boldsymbol{\lambda}}(x_{1k},x_{2k})
&\sim&
\left( \frac{k\,\| \boldsymbol{\lambda} \|}{2\pi\,\tau}  \right)
^{d-1+\frac{1-r_G}{2}}\,
\left(\frac{\mathrm{vol}(\mathcal{O}_{\boldsymbol{\lambda}})}{\mathrm{vol}^\kappa (G)}\right)^2\cdot 
\frac{\mathrm{vol}^\kappa
(T)}{\mathfrak{D}^\kappa (x)\cdot\mathfrak{d}_{\boldsymbol{\lambda}}}\nonumber
\\
&&\cdot 
\exp\left(\frac{\|\boldsymbol{\lambda}\|}{\tau}\,\left[
\psi_2^{\omega_x}\big( \mathbf{s}_1
,
\mathbf{s}_2  \big)
-\|
\mathbf{n}_1\|_{\tilde{\kappa}_x}^2
-\|\mathbf{n}_2\|_{\tilde{\kappa}_x}  ^2
\right]\right)\nonumber\\
&&\cdot \left[1+\sum_{j\ge 1}k^{-j/2}\,R_j(\theta_1,\theta_2,\mathbf{s}_1,
\mathbf{s}_2,\mathbf{n}_1,
\mathbf{n}_2)    \right],\nonumber
\end{eqnarray*}
\begin{eqnarray*}
P^\tau_{k\,\boldsymbol{\lambda}}(x_{1k},x_{2k})
&\sim&
\left(  \frac{1}{2} \right)^{\frac{d-1}{2}}
\left( \frac{k\,\| \boldsymbol{\lambda} \|}{2\pi\,\tau}  \right)
^{\frac{d-r_G}{2}}\,
\left(\frac{\mathrm{vol}(\mathcal{O}_{\boldsymbol{\lambda}})}{\mathrm{vol}^\kappa (G)}\right)^2\cdot 
\frac{\mathrm{vol}^\kappa
(T)}{\mathfrak{D}^\kappa (x)\cdot\mathfrak{d}_{\boldsymbol{\lambda}}}\nonumber
\\
&&\cdot 
\exp\left(\frac{\|\boldsymbol{\lambda}\|}{\tau}\,\left[
\psi_2^{\omega_x}\big( \mathbf{s}_1
,
\mathbf{s}_2  \big)
-\|
\mathbf{n}_1\|_{\tilde{\kappa}_x}^2
-\|\mathbf{n}_2\|_{\tilde{\kappa}_x}  ^2
\right]\right)\nonumber\\
&&\cdot \left[1+\sum_{j\ge 1}k^{-j/2}\,S_j(\theta_1,\theta_2,\mathbf{s}_1,
\mathbf{s}_2,\mathbf{n}_1,
\mathbf{n}_2)    \right],\nonumber\nonumber
\end{eqnarray*}
where $R_j,\,S_j$ are polynomials of degree $\le 3j$ and
parity $j$.

\end{thm}

\subsection{Applications}

Let us describe a sample of applications of the previous Theorems.
To begin with, Theorem \ref{thm:rapid decay moment map} and the
second expansion of Theorem 
\ref{thm:rescaled asympt}
yield an asymptotic estimate on the $L^\infty$ norms of complexified
matrix elements and their Husimi probability distributions (see
\cite{z12}, \cite{z20}).

If $\varphi\in L^2(G)$ is an eigenfunction of
$\Delta$, let $\tilde{\varphi}$ denote its complexification
and set $\tilde{\varphi}^\tau:=
\left.\tilde{\varphi}\right|_{X^\tau}$.
The \textit{Husimi distribution} 
$U^\tau_{\varphi}:X^\tau\rightarrow \mathbb{R}$ is given by
$$
U^\tau_{\varphi}(x):=\frac{\left|\tilde{\varphi}^\tau (x)\right|^2}{\|\varphi^\tau\|^2_{L^2(X^\tau)}}.
$$
The following Proposition specializes to the present setting
the upper bounds in Theorem 0.1 in \cite{z20}.

\begin{thm}
\label{thm:Husimi equiv}
There exists constants $C(\tau,\boldsymbol{\nu}),\,
C'(\tau,\boldsymbol{\nu})>0$ such that
for any $\gg 0$ the following holds.

\begin{enumerate}
\item Let $\varphi\in L^2(G)_{k\,\boldsymbol{\nu}}$ have
unit $L^2(G)$-norm.
Then
$$
\max_{x\in X^\tau}\left\{\left|\varphi^\tau(x)\right|\right\}\le 
C(\tau,\boldsymbol{\nu})\,
e^{\tau\,c_{k\,\boldsymbol{\lambda}}}\,
(c_{k\,\boldsymbol{\lambda}})
^{\frac{d-r_G}{4}}.
$$

\item Under the same assumptions,
$$
\max_{x\in X^\tau}\left\{ U^\tau_{\varphi}(x)^{\frac{1}{2}}
\right\}\le 
C'(\tau,\boldsymbol{\nu})\,
(c_{k\,\boldsymbol{\lambda}})
^{\frac{d}{2}-\frac{1+r_G}{4}}.
$$

\end{enumerate}

\end{thm}

The second application concerns the norm
of $\Pi^\tau_{k\,\boldsymbol{\lambda}}$ as a bounded operator
$L^p(X^\tau)\rightarrow L^q(X^\tau)$. A similar estimate
can be given for $P^\tau_{k\,\boldsymbol{\lambda}}$ and
will be left to the reader. This result is the analogue of 
analogous operator estimates in \cite{cr2} 
and \cite{gp24} in the Grauert tube setting, and \cite{sz03} in the line bundle context.
On the other hand, these estimates are inspired by previous estimates
in the real domain due to Sogge (\cite{so88} and \cite{so}).

\begin{thm}
\label{thm:operator norm norm estimate}
Suppose $1\le p\le q$, and define $R$ by setting
$$
\frac{1}{R}:=1-\frac{1}{p}+\frac{1}{q}.
$$ 
Then there is a constant
$C>0$ such that for any $\epsilon>0$ and
$k\gg 0$ one has
$$
\|\Pi^{\tau}_{k\,\boldsymbol{\lambda}}\|_{L^p(M)\rightarrow L^q(M)}\le C\,
k^{\frac{1}{2\,R}\left[ (d-1)\,\left( R-1  \right)
+R(d-r_G)\right]+\epsilon}.
$$

\end{thm}

Let us close this introduction with a remark 
on a slight difference in the representation-theoretic
approach
between \cite{p22}, where similar asymptotics
were considered in the line bundle setting, 
and the present paper.
Given 
a dominant weight 
$\boldsymbol{\lambda}$, let us set 
$\boldsymbol{\nu}:=\boldsymbol{\nu}_{\boldsymbol{\lambda}}=\boldsymbol{\lambda}+\boldsymbol{\delta}$,
where $\boldsymbol{\delta}$ is the half-sum of the positive roots
- see (\ref{eqn:half sum prs}) below;
for a fixed dominant weight $\boldsymbol{\lambda}$, the 
equivariant asymptotics in \cite{p22} are those associated to
the sequence of weights $\boldsymbol{\lambda}_k$ such that
$\boldsymbol{\nu}_{\boldsymbol{\lambda}_k}=k\,\boldsymbol{\nu}$. 

In this paper, given a fixed regular dominant weight
$\boldsymbol{\lambda}$, we consider the asymptotics
associated to the representations $k\,\boldsymbol{\lambda}$.
This implies that the role played by 
$\mathcal{O}_{\boldsymbol{\nu}}$ in \cite{p22} is played here by
$\mathcal{O}_{\boldsymbol{\lambda}}$.

This approach is perhaps more natural, and makes the
character computations based on the Kirillov character formula
only slightly more involved. The same arguments and change in
approach could of course be applied to the setting of
\cite{p22}.

\section{Notation}

In this section, we collect some notation and conventions 
used in the paper for the
reader's convenience. 

To begin with, we shall think of $\mathfrak{g}$ as the Lie algebra of left-invariant vector fields on $G$ (that is, those vector fields associated to
the right action of $G$ on itself by right translations).
Since $G$ also acts on $\tilde{G}$ by right translations, any $\boldsymbol{\xi}\in \mathfrak{g}$ extends to
a left-invariant vector field $\boldsymbol{\xi
}^{\tilde{G}}$
on $\tilde{G}$. 

On the other hand, $G$ also acts on itself and $\tilde{G}$ by left translations. This action associates to each $\boldsymbol{\xi}\in 
\mathfrak{g}$ a right-invariant vector field $\boldsymbol{\xi}_{\tilde{G}}$
on $\tilde{G}$. 

Thus $\boldsymbol{\xi
}^{\tilde{G}}(x),\,\boldsymbol{\xi
}_{\tilde{G}}(x)\in T_x\tilde{G}$ are given by
$$
\boldsymbol{\xi
}_{\tilde{G}}(x):=\left.\frac{\mathrm{d}}{\mathrm{d}t}\,
e^{t\,\boldsymbol{\xi}}
\cdot x\right|_{t=0},\quad
\boldsymbol{\xi
}^{\tilde{G}}(x):=\left.\frac{\mathrm{d}}{\mathrm{d}t}\,
x\cdot e^{t\,\boldsymbol{\xi}}\right|_{t=0}\quad
(x\in \tilde{G}).
$$

Since $X^\tau$ is $G$-invariant,
when $x\in X^\tau$ we have $\boldsymbol{\xi
}_{\tilde{G}}(x)\in T_xX^\tau$; we shall occasionally emphasize this by writing $\boldsymbol{\xi
}_{X^\tau}(x)$ for $\boldsymbol{\xi
}_{\tilde{G}}(x)$.

\begin{enumerate}

\item $\mathcal{O}_{\boldsymbol{\beta}}$: the coadjoint orbit through 
$\boldsymbol{\beta}\in \mathfrak{g}^\vee$; if $\boldsymbol{\beta}=
\boldsymbol{\lambda}$ (the fixed regular weight in the positive Weyl chamber) we shall often write $\mathcal{O}$ for $\mathcal{O}_{\boldsymbol{\lambda}}$;

\item $\mathcal{O}^\tau$: the rescaling of $\mathcal{O}$ which is contained in the sphere of radius $\tau$ in $\mathfrak{g}^\vee$ ($\tau>0$);

\item $\tilde{G}$: the complexification of $G$; $\tilde{\mathfrak{g}}=
\mathfrak{g}\oplus\imath\,\mathfrak{g}$: its Lie algebra;

\item $\boldsymbol{\xi}_{\tilde{G}}\in \mathfrak{X}(\tilde{G})$:
the vector field induced by $\boldsymbol{\xi}\in \mathfrak{g}$
where $G$ acts on $\tilde{G}$ by left translations (see above);

\item $\mathrm{val}_x:\mathfrak{g}\rightarrow T_x\tilde{G}$:
the evaluation map $\boldsymbol{\xi}\mapsto \boldsymbol{\xi}_{\tilde{G}}(x)$; its complex linear extension is the 
corresponding evaluation map for the holomorphic action
of $\tilde{G}$ on itself by left translations,
$\widehat{\mathrm{val}}_{x}:\tilde{\mathfrak{g}}\rightarrow T_x\tilde{G}$;

\item since $X^\tau\subset\tilde{G}$ is $G$-invariant for every $\tau>0$,
we have 
$\boldsymbol{\xi}_{\tilde{G}}(x)\in T_xX^\tau\subseteq T_x\tilde{G}$ 
if $x\in X^\tau$;
thus there is no ambiguity in identifying $\boldsymbol{\xi}_{\tilde{G}}(x)$ and
$\boldsymbol{\xi}_{X^\tau}(x)$ (similarly defined) and 
as mentioned we shall occasionally lighten notation (e.g. in the presence of subscripts) by simply writing 
$\boldsymbol{\xi}(x)$;

\item for a vector subspace $\mathfrak{a}\subseteq \mathfrak{g}$,
we shall set $\mathfrak{a}_{\tilde{G}}\subseteq T\tilde{G}$
the (real) vector subbundle given by
$$
\mathfrak{a}_{\tilde{G}}(x):=
\mathrm{val}_x(\mathfrak{a}),
$$
and similarly for its restriction 
$\mathfrak{a}_{X^\tau}\subset TX^\tau$ to $X^\tau$; we shall again 
occasionally lighten notation and simply write $\mathfrak{a}(x)$
for $\mathfrak{a}_{\tilde{G}}(x)$;

\item when the previous vector subspace depends 
on $x\in X^\tau$, we shall emphasize this by writing 
$\mathfrak{a}_x$, and lighten notation by writing
$\mathfrak{a}_x(x)$ instead of $({\mathfrak{a}_x})_{X^\tau}(x)$;

\item $\kappa_e$: the given
$\mathrm{Ad}$-invariant Euclidean product on $\mathfrak{g}$;

\item $\sigma_x:=\mathrm{val}_x^t(\hat{\kappa}_x)$ (the Euclidean product on $\mathfrak{g}$ given by pull-back
of $\hat{\kappa}$ under vector field evaluation at $x$
- see \S \ref{sctn:direct sum dec g});

\item $\rho$: the unique K\"{a}hler potential on $\tilde{G}$, such that
$\sqrt{\rho}$ satisfies the homogeneous Monge-Amp\`{e}re equation on
$\tilde{G}\setminus G$ and the associated K\"{a}hler metric $\hat{\kappa}$ restricts to $\kappa$ on $G$;

\item $\Omega:=\imath\,\overline{\partial}\,\partial\,\rho$: 
the K\"{a}hler form on $\tilde{G}$; 
$\omega:=\frac{1}{2}\,\Omega$, $\tilde{\kappa}:=\frac{1}{2}\,\hat{\kappa}$.

\end{enumerate}

\section{An example}

Consider the compact torus 
$T^d=\mathbb{R}^d/(2\,\pi\,\mathbb{Z}^d)$ with the standard metric (\textit{cfr.} \S 2 of \cite{p24});
in this case $d=r_{T^d}$. 
Let us identify $\mathbb{R}^d\cong (\mathbb{R}^d)^\vee$ by the
standard Euclidean product.
The the unitary dual $\hat{T}^d\cong \mathbb{Z}^d$, 
and for each $\boldsymbol{\lambda}\in \mathbb{Z}^d$ the subspace
$L^2(T^d)_{\boldsymbol{\lambda}}$ is $1$-dimensional.
Furthermore, the cone $\mathcal{C}(\mathcal{O})$ is in this case
just the positive ray $\mathbb{R}_+\cdot\boldsymbol{\lambda}$.

More precisely, let 
$e^{\imath\,\boldsymbol{\vartheta}}=\begin{pmatrix}
e^{\imath\,\vartheta_1}&\ldots&e^{\imath\,\vartheta_d}
\end{pmatrix}$
denote the general element of $T^d$, and for $\boldsymbol{\lambda}\in \mathbb{Z}^d$
let us set $\varphi_{\boldsymbol{\lambda}}
(e^{\imath\,\boldsymbol{\vartheta}})
:=c_d\,
e^{-\imath\,
\langle\boldsymbol{\lambda},\boldsymbol{\vartheta}
\rangle}$, where $c_d=(2\,\pi)^{-d/2}$.
In view of (\ref{eqn:left translations}), $(\varphi_{\boldsymbol{\lambda}})$ is an orthonormal basis  of $L^2(T^d)_{\boldsymbol{\lambda}}$;
furthermore,
$\varphi_{\boldsymbol{\lambda}}$
is an eigenvector of the positive 
Laplacian corresponding to the eigenvalue $\|\boldsymbol{\lambda}\|^2$.

Let us identify $\tilde{T}^d=\mathbb{C}^d/(2\,\pi\,\mathbb{Z}^d)\cong (\mathbb{C}^*)^d$, and write the
general element of
$\tilde{T}^d$ as $e^{\imath\,\boldsymbol{\vartheta}+\mathbf{r}}
=e^{\imath\,\boldsymbol{\vartheta}}\,
e^{\imath \cdot \imath (-\mathbf{r})}$,
where $\mathbf{r}\in \mathbb{R}^d$. The holomorphic extension of
$\varphi_{\boldsymbol{\lambda}}$ is
$\tilde{\varphi}_{\boldsymbol{\lambda}}(e^{\imath\,\boldsymbol{\vartheta}+\mathbf{r}}):=c_d\,
e^{-\imath\,
\langle\boldsymbol{\lambda},\boldsymbol{\vartheta}
\rangle
-\langle
\boldsymbol{\lambda},\mathbf{r}
\rangle
}$.

On $\tilde{T}^d\setminus T$, we may pass to polar coordinates in the real component, and write the general element as
in the form
$e^{\imath\,\boldsymbol{\vartheta}+\tau\,\boldsymbol{\omega}}
=e^{\imath\,\boldsymbol{\vartheta}}\,
e^{\imath\,\imath (-\tau\,\boldsymbol{\omega})}$, where 
$\tau>0$ and $\boldsymbol{\omega}\in S^{d-1}$ (the unit sphere).
Keeping $\tau$ fixed yields a parametrization of $X^\tau$.

Thus on $X^\tau$ 
$$
e^{-2\,k\,\tau\,\|\boldsymbol{\lambda}\|}\,\left|\tilde{\varphi}_{k\,\boldsymbol{\lambda}}(e^{\imath\,\boldsymbol{\vartheta}+\tau\,\boldsymbol{\omega}})
\right|^2
=(2\,\pi)^{-d}\,
e^{2\,\tau\,k\,\left[\langle\boldsymbol{\lambda},
-\boldsymbol{\omega}\rangle-\|\boldsymbol{\lambda}\|\right]}.
$$
We have $\langle\boldsymbol{\lambda},
-\boldsymbol{\omega}\rangle-\|\boldsymbol{\lambda}\|\le 0$ for any
$\boldsymbol{\omega}$ of unit norm, and equality holds if and only if
$-\boldsymbol{\omega}=
\boldsymbol{\lambda}/\|\boldsymbol{\lambda}\|$; 
in terms
of the moment map $\tilde{\Phi}$ discussed in
\S \ref{sctn:prel}, this is 
the condition $e^{\imath\,\boldsymbol{\vartheta}
+\tau\,\boldsymbol{\omega}}\in 
{\tilde{\Phi}}^{-1}(\mathbb{R}_+\cdot\boldsymbol{\lambda})$.

Therefore, if $e^{\imath\,\boldsymbol{\vartheta}+\tau\,\boldsymbol{\omega}}
\not\in 
{\tilde{\Phi}}^{-1}(\mathbb{R}_+\cdot\boldsymbol{\lambda})\cap X^\tau$ then
$$
e^{-2\,k\,\tau\,\|\boldsymbol{\lambda}\|}\,\left|\tilde{\varphi}_{k\,\boldsymbol{\lambda}}(e^{\imath\,\boldsymbol{\vartheta}+\tau\,\boldsymbol{\omega}})
\right|^2=O\left(k^{-\infty}\right)\quad\text{for}\quad
k\rightarrow+\infty;
$$
on the other hand, if
$e^{\imath\,\boldsymbol{\vartheta}+\tau\,\boldsymbol{\omega}}\in 
\Phi^{-1}(\mathbb{R}_+\cdot\boldsymbol{\lambda})\cap X^\tau$ then
$$
\left|\tilde{\varphi}_{k\,\boldsymbol{\lambda}}(e^{\imath\,\boldsymbol{\vartheta}+\tau\,\boldsymbol{\omega}})
\right|
=(2\,\pi)^{-d/2}\,e^{\tau\,(k\,\|\boldsymbol{\lambda}\|)}
,\quad
e^{-2\,k\,\tau\,\|\boldsymbol{\lambda}\|}\,\left|\tilde{\varphi}_{k\,\boldsymbol{\lambda}}(e^{\imath\,\boldsymbol{\vartheta}+\tau\,\boldsymbol{\omega}})
\right|^2
=(2\,\pi)^{-d}.
$$
This conclusion is in agreement with the statement of
Theorem \ref{thm:rescaled asympt} for $P^\tau_{k\,\boldsymbol{\lambda}}$,
since $\mathrm{vol}^\kappa(G)=(2\,\pi)^d$ and
$D_x=2^{-1}\,I_{d-1}$, so that
$$
\mathfrak{D}^\kappa(x)=\sqrt{\det(D_x)}
=\frac{1}{2^{(d-1)/2}}.
$$

Let us determine the orthonormal basis of 
$L^2(X^\tau)_{k\,\boldsymbol{\lambda}}$.
Let $S^{d-1}(\tau)\subseteq \mathbb{R}^d$ denote the
sphere centered of radius $\tau$ at the origin.
Then the $L^2$-norm of $\tilde{\varphi}_{k\,\boldsymbol{\lambda}}^\tau$ is
\begin{eqnarray}
\left\|\tilde{\varphi}_{k\,\boldsymbol{\lambda}}^\tau\right\|^2
_{L^2(X^\tau)}
&=&(2\,\pi)^{-d}\,\int_{T^d}\,\mathrm{d}\boldsymbol{\vartheta}
\,\int_{S^{d-1}(\tau)}\,\mathrm{d}\boldsymbol{\omega}\,
\left[e^{2\,k\,\langle\boldsymbol{\lambda},
\boldsymbol{\omega}\rangle}   \right]
\nonumber\\
&=&\tau^{d-1}\,\int_{S^{d-1}(1)}\,\mathrm{d}\boldsymbol{\omega}\,
\left[e^{2\,\tau\,k\,\langle\boldsymbol{\lambda},
\boldsymbol{\omega}\rangle}   \right]\nonumber\\
&=&\tau^{d-1}\,e^{2\,k\,\tau\,\|\boldsymbol{\lambda}\|}\,
\int_{S^{d-1}(1)}\,\mathrm{d}\boldsymbol{\omega}\,
\left[e^{2\,\tau\,k\,\left[\langle\boldsymbol{\lambda},
\boldsymbol{\omega}\rangle-
\|\boldsymbol{\lambda}\|\right]}   \right]\nonumber\\
&=&\tau^{d-1}\,e^{2\,k\,\tau\,\|\boldsymbol{\lambda}\|}\,
\int_{S^{d-1}(1)}\,\mathrm{d}\boldsymbol{\omega}\,
\left[e^{\imath\,k\,\Psi_{\boldsymbol{\lambda}}
(\boldsymbol{\omega})}   \right],\nonumber 
\end{eqnarray}
where
$$
\Psi_{\boldsymbol{\lambda}}
(\boldsymbol{\omega})=\imath\,2\,\tau\,\left[
\|\boldsymbol{\lambda}\|-
\langle\boldsymbol{\lambda},
\boldsymbol{\omega}\rangle\right].
$$

We have $\Psi_{\boldsymbol{\lambda}}=\Im\Psi_{\boldsymbol{\lambda}}\ge 0$, and
$\Im\Psi_{\boldsymbol{\lambda}}= 0$
if and only if $\boldsymbol{\omega}=
\boldsymbol{\lambda}/\|\boldsymbol{\lambda}\|$.
In the neighbourhood of $\boldsymbol{\omega}=
\boldsymbol{\lambda}/\|\boldsymbol{\lambda}\|$,
we can write
$$
\boldsymbol{\omega}_{\boldsymbol{\eta}}:=\sqrt{1-\|\boldsymbol{\eta}\|^2}\,
\frac{\boldsymbol{\lambda}}{\|\boldsymbol{\lambda}\|}+
\boldsymbol{\eta}=\left(1-\frac{1}{2}\,\|\boldsymbol{\eta}\|^2
+R_3(\boldsymbol{\eta})\right)\,\frac{\boldsymbol{\lambda}}{\|\boldsymbol{\lambda}\|}+
\boldsymbol{\eta},
$$
where $\boldsymbol{\eta}\in \boldsymbol{\lambda}^\perp
\cong \mathbb{R}^{2d-1}$
varies near the origin.
Hence,
\begin{eqnarray*}
\Psi_{\boldsymbol{\lambda}}(\boldsymbol{\omega}_{\boldsymbol{\eta}})&=&
2\,\imath\,\tau\,
\left[ \|\boldsymbol{\lambda}\|- 
\left(1-\frac{1}{2}\,\|\boldsymbol{\eta}\|^2
+R_3(\boldsymbol{\eta})\right)\,\|\boldsymbol{\lambda}\|
     \right]\\
     &=&\imath\,\tau\,\left[\|\boldsymbol{\eta}\|^2\cdot 
     \|\boldsymbol{\lambda}\|+R_3(\boldsymbol{\eta})\right].
\end{eqnarray*}
Hence $\boldsymbol{\eta}=\mathbf{0}$ is a 
non-degenerate critical point, with Hessian
matrix $2\,\imath\,\tau\,\|\boldsymbol{\lambda}\|\,I_{d-1}$.
Thus
$$
\sqrt{\det\left( \frac{1}{2\,\pi\,\imath}	\,2\,\imath\,
\tau\,\|\boldsymbol{\lambda}\|\,I_{d-1}   \right)}
=\sqrt{\left(\frac{\tau\,\|\boldsymbol{\lambda}\|}{\pi}\right)^{d-1}}=
\left(\frac{\tau\,\|\boldsymbol{\lambda}\|}{\pi}\right)
^{\frac{d-1}{2}}
.
$$
Hence, there is an asymptotic expansion for 
$k\rightarrow +\infty$ 
\begin{eqnarray*}
\left\|\tilde{\varphi}_{k\,\boldsymbol{\lambda}}^\tau\right\|^2
_{L^2(X^\tau)}&\sim& \tau^{d-1}\,e^{2\,k\,\tau\,\|\boldsymbol{\lambda}\|}\,
\left(\frac{\pi}{k\,\tau\,\|\boldsymbol{\lambda}\|}\right)
^{\frac{d-1}{2}}\cdot \left(1+\sum_jk^{-j}\,a_j\right)\\
&=&e^{2\,k\,\tau\,\|\boldsymbol{\lambda}\|}\,
\left(\frac{\tau\,\pi}{k\,\|\boldsymbol{\lambda}\|}\right)
^{\frac{d-1}{2}}\cdot \left(1+\sum_jk^{-j}\,a_j\right)
.
\end{eqnarray*}

Thus the $L^2$-normalization of 
$\tilde{\varphi}^\tau_{k\,\boldsymbol{\lambda}}$
is
$$
\rho^\tau_{k\,\boldsymbol{\lambda}}
\left( e^{\imath\,\boldsymbol{\vartheta}+\tau\,\boldsymbol{\omega}}   \right)
\sim \frac{1}{(2\,\pi)^{d/2}}\,
\left(\frac{k\,\|\boldsymbol{\lambda}\|}{\tau\,\pi}\right)
^{\frac{d-1}{4}}\,
e^{\imath\,k\,
\langle\boldsymbol{\lambda},\boldsymbol{\vartheta}
\rangle
+k\,\tau\,\left[\langle
\boldsymbol{\lambda},-\boldsymbol{\omega}
\rangle-\|\boldsymbol{\lambda}\|\right]
}.
$$

Therefore, to leading order,
\begin{eqnarray}
\label{eqn:general near dg expon tor}
\left| \rho^\tau_{k\,\boldsymbol{\lambda}}
\left( e^{\imath\,\boldsymbol{\vartheta}+\tau\,\boldsymbol{\omega}}   \right)   \right|^2   
&\sim&\frac{1}{(2\,\pi)^{d}}\,
\left(\frac{k\,\|\boldsymbol{\lambda}\|}{\tau\,\pi}\right)
^{\frac{d-1}{2}}\,
e^{2\,k\,\tau\,\left[\langle
\boldsymbol{\lambda},-\boldsymbol{\omega}
\rangle-\|\boldsymbol{\lambda}\|\right]
},
\end{eqnarray}
which is also rapidly decreasing unless 
$-\boldsymbol{\omega}=\boldsymbol{\lambda}/\|\boldsymbol{\lambda}\|$,
and in that case it agrees 
with the first asymptotic expansion in
Theorem \ref{thm:rescaled asympt}.

Let us consider a rescaled displacement in a normal direction
to $Z^\tau$, which amounts in this case to replacing
(to first order) $\tau\,\boldsymbol{\omega}=\tau\,(
-\boldsymbol{\lambda}/\|\boldsymbol{\lambda}\|)$ by $\tau\,\boldsymbol{\omega}=\tau(-\boldsymbol{\lambda}/\|\boldsymbol{\lambda}\|)+\mathbf{n}/\sqrt{k}$, where $\mathbf{n}\in \boldsymbol{\lambda}^\perp$
is fixed.
In the previous general notation, 
we consider the asymptotics at a diagonal pair
$(x+\mathbf{n}/\sqrt{k},x+\mathbf{n}/\sqrt{k})$, with
$x\in X^\tau_{\mathcal{O}}$ and $\mathbf{n}$ normal to $X^\tau_{\mathcal{O}}$.
Since $\boldsymbol{\omega}\in S^{d-1}$, 
$$
\boldsymbol{\omega}=\sqrt{1-\frac{1}{\tau^2\,k}\|\mathbf{n}\|^2}\,
\left(-\frac{\boldsymbol{\lambda}}{\|\boldsymbol{\lambda}\|}\right)+
\frac{1}{\tau\,\sqrt{k}}\,\mathbf{n}.
$$
Hence, the exponent in (\ref{eqn:general near dg expon tor})
is
\begin{eqnarray*}
\lefteqn{ 
2\,k\,\tau\,\left[\langle
\boldsymbol{\lambda},-\boldsymbol{\omega}
\rangle-\|\boldsymbol{\lambda}\|\right]  
=2\,k\,\tau\,\|\boldsymbol{\lambda}\|\,\left[
\sqrt{1-\frac{1}{\tau^2\,k}\|\mathbf{n}\|^2}-1\right] }\\\\
&=&2\,k\,\tau\,\|\boldsymbol{\lambda}\|\,
\left[
-\frac{1}{2\,\tau^2\,k}\|\mathbf{n}\|^2+R_2\left(\frac{\|\mathbf{n}\|^2}{k}\right)\right] 
=\frac{\|\boldsymbol{\lambda}\|}{\tau}\,\left[
-\|\mathbf{n}\|^2+O\left(k^{-1}\right)
\right]\\
&=&\frac{\|\boldsymbol{\lambda}\|}{\tau}\,\left[
-2\,\|\mathbf{n}\|_{\tilde{\kappa}}^2+O\left(k^{-1}\right)
\right],
\end{eqnarray*}
where $\|\cdot\|^2_{\tilde{\kappa}}=\|\cdot\|^2/2$ in the notation of the 
introduction. The same argument applies for, say,
$\|\mathbf{n}\|=O\left(k^{1/6}\right)$, with a remainder
term $O\left(k^{-1/3}\right)$.

\section{Preliminaries}
\label{sctn:prel}

In this section, we collect some basic results that will be used in the following proofs. More specific preliminaries
will be given in the subsections devoted to the proofs.

\subsection{The Hamiltonian structure of $\tilde{L}$}
\label{sctn:ham str L}
As above,
let $L:G\times G\rightarrow G$ and $\tilde{L}:G\times \tilde{G}\rightarrow \tilde{G}$ be the actions given by left translations. 
Then $\tilde{L}$ is intertwined by $E$ in (\ref{eqn:defn of gamma'}) with the tangent lift $\mathrm{d}L:G\times TG\rightarrow TG$ of $L$.

In turn, $\tilde{L}$ extends to the holomorphic action 
$\hat{L}$
of $\tilde{G}$
on itself given by left translations. 
For any $x\in \tilde{G}$, the induced evaluation 
$\widehat{\mathrm{val}}_{x}:\tilde{\mathfrak{g}}\rightarrow T_x\tilde{G}$
is an isomorphism of complex vector spaces.
Since $\tilde{\mathfrak{g}}=\mathfrak{g}\oplus \imath\,\mathfrak{g}$, 
we conclude that
$(\mathfrak{g}_{\tilde{G}})_x\cap J_x\big(
(\mathfrak{g}_{\tilde{G}})_x\big)=(0)$ (heuristically, the action of $G$
on $\tilde{G}$ and on $TG$ is \lq totally real\rq).
Since $\mathrm{d}L$ is Hamiltonian for $\Omega_{\mathrm{can}}$, so is $\tilde{L}$ for $\Omega$, with moment
map
$$
\tilde{\Phi}:=\Phi\circ E^{-1}:\tilde{G}\rightarrow \mathfrak{g}^\vee,
$$
where $\Phi:TG\rightarrow \mathfrak{g}^\vee$ is the moment map for 
$\mathrm{d}L$.

For any $\boldsymbol{\xi}\in \mathfrak{g}$, 
$\boldsymbol{\beta}
\in \mathfrak{g}^\vee$ let us set
\begin{equation}
\label{eqn:defn di xik e betak}
\boldsymbol{\xi}_\kappa:=\mathcal{L}(\boldsymbol{\xi})\in \mathfrak{g}^\vee,\qquad\boldsymbol{\beta}^\kappa:=\mathcal{L}^{-1}(\boldsymbol{\beta}).
\end{equation}
Thus the annhilator subspace of $\boldsymbol{\beta}$ is the
orthocomplement of $\boldsymbol{\beta}^\kappa$ for $\kappa_e$:
$$
\boldsymbol{\beta}^0=(\boldsymbol{\beta}^\kappa)^{\perp_{\kappa_e}}
\subseteq \mathfrak{g}.
$$

Then $\Phi$ can be written equivalently as
$$
\Phi (g,\boldsymbol{\xi}\cdot g)=\boldsymbol{\xi}_\kappa,
\qquad \Phi (g,g\cdot \boldsymbol{\xi})
=\mathrm{Ad}_g(\boldsymbol{\xi})_\kappa.
$$
Hence
\begin{equation}
\label{eqn:moment map tilde G}
\tilde{\Phi}\big(g\,\exp^{\tilde{G}}(\imath\,\boldsymbol{\xi})\big)=
\mathrm{Ad}_g(\boldsymbol{\xi})_\kappa=\mathrm{Coad}_g(\boldsymbol{\xi}_\kappa).
\end{equation}
In other words, if we identify $\mathfrak{g}$ and $\mathfrak{g}^\vee$ 
by $\mathcal{L}$ and $TG$ with $G\times \mathfrak{g}\cong G\times \mathfrak{g}^\vee$ by right
translations, then $\Phi$ is the projection on the second factor.

Let now 
$$
\mathcal{C}(\mathcal{O}):=
\left\{r\,\boldsymbol{\beta}\,:\,r>0,\,\boldsymbol{\beta}
\in \mathcal{O}\right\}\subset \mathfrak{g}^\vee,
\quad
\mathcal{C}(\tilde{\mathcal{O}}):=
\left\{r\,\boldsymbol{\xi}\,:\,r>0,\,\boldsymbol{\xi}
\in \tilde{\mathcal{O}}\right\}\subseteq \mathfrak{g}
$$
be the cone over $\tilde{\mathcal{O}}=\mathcal{L}^{-1}(\mathcal{O})$. 
Then, recalling Definition \ref{defn:defn di XtauO},
\begin{equation}
\label{eqn:defn di tildeGO}
\tilde{\Phi}^{-1}\left(\mathcal{C}(\tilde{\mathcal{O}})\right)=
\gamma\left(G\times \mathcal{C}(\tilde{\mathcal{O}})\right)
=
\tilde{G}_{\mathcal{O}}
.
\end{equation}
Therefore, if $\Phi^\tau:=\left.\tilde{\Phi}\right|_{X^\tau}:X^\tau\rightarrow\mathfrak{g}^\vee$ then
\begin{equation}
\label{eqn:XtauO GtauO}
{\Phi^\tau}^{-1}\big(\mathcal{C}(\mathcal{O})\big)=X^\tau_{\mathcal{O}}
=\tilde{G}_{\mathcal{O}}\cap X^\tau.
\end{equation}
(see Definition \ref{defn:defn di XtauO}), 
and the intersection is transverse.

\subsection{The splitting of $T\tilde{G}$}
\label{sctn:splitting alpha}

The previous setting determines a natural splitting of the
tangent bundle $T\tilde{G}$ on $\tilde{G}\setminus G$.
Let $\upsilon_{\sqrt{\rho}}$ be the Hamiltonian vector field of
$\sqrt{\rho}$ with respect to $\Omega$.
Then 
\begin{equation}
\label{eqn:VT vector bundles}
\mathcal{T}:=\mathrm{span}_{\mathbb{R}}\left(\upsilon_{\sqrt{\rho}}  \right),\quad
\mathcal{V}:=\mathrm{span}_{\mathbb{C}}\left(\upsilon_{\sqrt{\rho}}  \right)
=\mathrm{span}_{\mathbb{R}}
\left(\upsilon_{\sqrt{\rho}} ,\,J\left(\upsilon_{\sqrt{\rho}}  \right)  \right)
\end{equation}
are a real and a complex line subbundle of $T\tilde{G}$, respectively; the symplectic orthocomplement of $\mathcal{V}$ is a complex subbundle 
$\mathcal{H}$ of $T\tilde{G}$. 
Thus $T\tilde{G}=\mathcal{V}\oplus \mathcal{H}$. 

The previous splitting may be related to the exact symplectic structure of
$(\tilde{G},\,\Omega)$, as follows.
Let $\lambda_{\mathrm{can}}$ denote the canonical 1-form on $TG$ (identified with
$T^\vee G$ by $\kappa$), 
so that $\Omega_{\mathrm{can}}:=-\mathrm{d}\lambda_{\mathrm{can}}$,
and set
$$
\lambda:=(E^{-1})^*(\lambda_{\mathrm{can}}),\quad
\alpha:=-\lambda.
$$
Thus $\Omega=-\mathrm{d}\lambda=\mathrm{d}\alpha$.
Then
$$
\ker(\alpha)=J(\mathcal{T})\oplus \mathcal{H},\quad
\mathcal{T}=\ker(\alpha)^{\perp_{\hat{\kappa}}},
$$
 on $\tilde{G}\setminus G$,
 where the suffix $\perp_{\hat{\kappa}}$ denotes the Riemannian orthocomplement with respect to $\hat{\kappa}$.

Furthermore, $\mathcal{H}$ and $\mathcal{T}$ are tangent to $X^\tau$
for every $\tau>0$; therefore they restrict to subbundles $\mathcal{T}^\tau$
and $\mathcal{H}^\tau$ of $TX^\tau$, and we have the $\hat{k}$-orthogonal direct sum of vector bundles
\begin{equation}
\label{eqn:decomp Xtau}
TX^\tau=\mathcal{T}^\tau\oplus \mathcal{H}^\tau,\quad
\mathcal{H}^\tau=\ker(\alpha^\tau)
\end{equation}
where $\alpha^\tau$ is the pull-back of $\alpha$ to $X^\tau$.
Furthermore, $\mathcal{H}^\tau$ is the maximal complex subbundle of
$TX^\tau$.

For any $\boldsymbol{\xi}\in \mathfrak{g}$, let 
$\boldsymbol{\xi}_{\tilde{G}}\in \mathfrak{X}(\tilde{G})$ denote the
induced vector field on $\tilde{G}$ under $\tilde{L}$.
Then $\boldsymbol{\xi}_{\tilde{G}}$ is Hamiltonian and 
is tangent to $X^\tau$ for every $\tau$.
In view of this and (\ref{eqn:decomp Xtau}), by the discussion in 
\S 3.2 of [Pao 24] on $\tilde{G}\setminus G$ we have
\begin{equation}
\label{eqn:xiG horiz R}
\boldsymbol{\xi}_{\tilde{G}}=\boldsymbol{\xi}_{\tilde{G}}^\sharp
-\tilde{\varphi}^{\boldsymbol{\xi}}\,\mathcal{R},\,
\text{where $\boldsymbol{\xi}_{\tilde{G}}^\sharp$ is tangent to $\mathcal{H}$ and $
\mathcal{R}:=-\frac{1}{\sqrt{\rho}}\,\upsilon_{\sqrt{\rho}}$.}
\end{equation}
Here $\tilde{\varphi}^{\boldsymbol{\xi}}:=
\langle \tilde{\Phi},\boldsymbol{\xi}\rangle$.
We shall call $\mathcal{R}\in \mathfrak{X}(\tilde{G}\setminus G)$ in (\ref{eqn:xiG horiz R}) the \textit{Reeb vector field} of $\tilde{G}$; its restriction 
$\mathcal{R}^\tau\in \mathfrak{X}(X^\tau)$ is the genuine Reeb vector field
of $(X^\tau,\alpha^\tau)$ for every $\tau>0$.
In particular, $\alpha (\mathcal{R})\equiv 1$ on $\tilde{G}\setminus G$.

It follows from (\ref{eqn:moment map tilde G}) and 
(\ref{eqn:xiG horiz R}) that if 
$x=g\,\exp^{\tilde{G}}(\imath\,\boldsymbol{\xi})\in \tilde{G}$
and $\boldsymbol{\eta}\in \mathfrak{g}$ then
$$
\boldsymbol{\eta}_{\tilde{G}}(x)=
\boldsymbol{\eta}_{\tilde{G}}(x)^\sharp
-\kappa_e\left(\boldsymbol{\eta}, \mathrm{Ad}_g(\boldsymbol{\xi}) \right)\,\mathcal{R}_x.
$$
Hence $\boldsymbol{\eta}_{\tilde{G}}(x)\in \mathcal{H}_x$
if and only if 
$\boldsymbol{\eta}\in \mathrm{Ad}_g(\boldsymbol{\xi})^{\perp_{\kappa_e}}$.
We conclude the following.

\begin{lem}
\label{lem:H_xdirect sum eval}
Suppose that $x=g\,\exp^{\tilde{G}}(\imath\,\boldsymbol{\xi})\in \tilde{G}$.
Then
$$
\mathcal{H}_x=\widehat{\mathrm{val}}_{x}\left(
\mathrm{Ad}_g(\boldsymbol{\xi})^{\perp_{\kappa_e}}\oplus
\imath\,\mathrm{Ad}_g(\boldsymbol{\xi})^{\perp_{\kappa_e}}  
\right)  .
$$
\end{lem}

\begin{proof}
We have seen that
$\widehat{\mathrm{val}}_{x}\left(
\mathrm{Ad}_g(\boldsymbol{\xi})^{\perp_{\kappa_e}}\right)
\subseteq \mathcal{H}_x$.
Since $\mathcal{H}_x\subseteq T_x\tilde{G}$ is a complex
vector subspace, 
$$
\widehat{\mathrm{val}}_{x}\left(
\mathrm{Ad}_g(\boldsymbol{\xi})^{\perp_{\kappa_e}}\right)\oplus
J_x\left(\widehat{\mathrm{val}}_{x}\left(
\mathrm{Ad}_g(\boldsymbol{\xi})^{\perp_{\kappa_e}}\right)
\right)
\subseteq \mathcal{H}_x.
$$
However $\widehat{\mathrm{val}}_{x}$ is $\mathbb{C}$-linear,
hence $\widehat{\mathrm{val}}_{x}\circ M_\imath=J_x\circ 
\widehat{\mathrm{val}}_{x}$ ($M_\imath$ denoting multiplication
by $\imath$). 
Thus
$$
\mathcal{H}_x\supseteq\widehat{\mathrm{val}}_{x}\left(
\mathrm{Ad}_g(\boldsymbol{\xi})^{\perp_{\kappa_e}}\oplus
\imath\,\mathrm{Ad}_g(\boldsymbol{\xi})^{\perp_{\kappa_e}}  
\right)  ;
$$
since both spaces have complex dimension $d-1$,
equality holds.
\end{proof}

\subsection{The normal bundle to $\tilde{G}_{\mathcal{O}}$ and 
$X^\tau_{\mathcal{O}}$}

\label{sctn:normal bundle}

We assume here that $\mathcal{O}$ is the coadjoint orbit
(not necessarily of maximal dimension) through a given 
$\boldsymbol{\lambda}\neq 0\in \mathfrak{g}^\vee$.

By (\ref{eqn:XtauO GtauO}), the normal bundle 
$N(X^\tau_{\mathcal{O}}/X^\tau)$ of
$X^\tau_{\mathcal{O}}$ in $X^\tau$
is the restriction to $X^\tau$ of the normal bundle 
$N(\tilde{G}_{\mathcal{O}}/\tilde{G})$ to $\tilde{G}_{\mathcal{O}}$ in $\tilde{G}$.
If the stabilizer $T_{\boldsymbol{\lambda}}\leqslant T$ of 
$\boldsymbol{\lambda}$ has dimension $r_{\boldsymbol{\lambda}}$, then 
$\mathcal{O}\cong G/T_{\boldsymbol{\lambda}}$ has dimension 
$d-r_{\boldsymbol{\lambda}}$; hence 
$\dim\big(\mathcal{C}(\mathcal{O})  \big)=d-r_{\boldsymbol{\lambda}}+1$.
By (\ref{eqn:defn di tildeGO}), 
$\dim(\tilde{G}_{\mathcal{O}})=2d-r_{\boldsymbol{\lambda}}+1$; thus $N(\tilde{G}_{\mathcal{O}}/\tilde{G})$ has rank $r_{\boldsymbol{\lambda}}-1$.
In the sequel we shall also set
$r_{\mathcal{O}}:=r_{\boldsymbol{\lambda}}$ if 
$\boldsymbol{\lambda}\in \mathcal{O}$.

\begin{defn}
\label{defn:tlambda'}
For $\boldsymbol{\lambda}\in \mathfrak{g}^\vee$, let us set set
$$
\mathfrak{t}_{\boldsymbol{\lambda}}':=
\mathfrak{t}_{\boldsymbol{\lambda}}\cap \boldsymbol{\lambda}^0=
\left\{ \boldsymbol{\eta}\in \mathfrak{g}\,:\,
\left[\boldsymbol{\eta},\boldsymbol{\lambda}^\kappa\right]=0,\,
\kappa_e\left(\boldsymbol{\eta},\boldsymbol{\lambda}^\kappa\right)=0  \right\}.
$$
We shall denote by $T'_{\boldsymbol{\lambda}}\leqslant 
T_{\boldsymbol{\lambda}}$ the connected subgroup with Lie algebra
$\mathfrak{t}_{\boldsymbol{\lambda}}'$. 

\end{defn}

\begin{rem}
If $\boldsymbol{\lambda}\in 
\mathcal{C}(\mathcal{O})$, where $\mathcal{O}$ is an integral orbit
\cite{gstb82}, then $T'_{\boldsymbol{\lambda}}$ is closed. In particular, 
this is the
case if $\boldsymbol{\lambda}\in \mathcal{D}^G$.
\end{rem}

\begin{rem}
\label{rem:tlambda'}
For $x\in \tilde{G}$, we shall abridge notation by
setting
$$T_x:=T_{\tilde{\Phi}(x)},\quad \mathfrak{t}_x:=
\mathfrak{t}_{\tilde{\Phi}(x)},
\quad
\mathfrak{t}_x':=\mathfrak{t}_{\tilde{\Phi}(x)}',\quad
T'_x:=T'_{\tilde{\Phi}(x)}.
$$
Hence if $x=g\cdot\exp^{\tilde{G}}(\imath\,\boldsymbol{\xi})$ then
$T_x=T_{\mathrm{Coad}_g(\boldsymbol{\xi}_\kappa)}=
T_{\mathrm{Ad}_g(\boldsymbol{\xi})}$, and so forth.
\end{rem}

Thus $\dim (\mathfrak{t}_x')=r_{\tilde{\Phi}(x)}-1$; since the
$G$ action on $\tilde{G}$ is free, for any $y\in \tilde{G}$ we also have
$\dim \left(
(\mathfrak{t}_x')_{\tilde{G}}(y)\right)=r_{\tilde{\Phi}(x)}-1$.
In the following, we shall abridge notation and simply
write 
\begin{equation}
\label{eqn:defn Np}
\mathfrak{t}_x'(x):=(\mathfrak{t}_x')_{\tilde{G}}(x).
\end{equation}
This does not define a vector bundle on $\tilde{G}$, given that 
$\mathfrak{t}_x'$ has
variable dimension.
On the other hand,
(\ref{eqn:defn Np}) does define a vector bundle on $\tilde{G}_{\mathcal{O}}$, with the following geometric significance: 
$\tilde{G}_{\mathcal{O}}$ is a coisotropic submanifold of $(\tilde{G},\Omega)$, and 
$\mathfrak{t}_x'(x)\subseteq T_x\tilde{G}_{\mathcal{O}}$
is the tangent space to its null fibration.

\begin{lem}
\label{lem:normal bundle GO}
For every $x\in \tilde{G}$, 
the following holds:
\begin{enumerate}
\item $\mathfrak{t}_x'(x)\oplus 
J_x\left( \mathfrak{t}_x'(x) \right)\subseteq \mathcal{H}_x$;
\item if $x\in \tilde{G}_{\mathcal{O}}$, then
$N(\tilde{G}_{\mathcal{O}}/\tilde{G})_x=
J_x\big(\mathfrak{t}_x'(x)\big)$.
\end{enumerate}

\end{lem}

\begin{proof}
To begin with, $\mathfrak{g}_{\tilde{G}}(x)
\cap J_x\left(\mathfrak{g}_{\tilde{G}}(x)\right)=(0)$
at any $x\in \tilde{G}$, because $\tilde{G}$ acts freely on itself
by left translations.

Suppose $\boldsymbol{\xi}\in \mathfrak{t}_x'$, and consider the
induced vector field (see (\ref{eqn:xiG horiz R}))
$\boldsymbol{\xi}_{\tilde{G}}=\boldsymbol{\xi}_{\tilde{G}}^\sharp
-\tilde{\varphi}^{\boldsymbol{\xi}}\,\mathcal{R}$.
As $\boldsymbol{\xi}\in \tilde{\Phi}(x)^0$, 
$\tilde{\varphi}^{\boldsymbol{\xi}}(x)=\langle\tilde{\Phi}(x),\boldsymbol{\xi}\rangle=0$;
hence $\boldsymbol{\xi}_{\tilde{G}}(x)\in \mathcal{H}_x$.
Thus $\mathfrak{t}_x'(x)\subseteq \mathcal{H}_x$ by definition.
Since $\mathcal{H}_x$ is a complex subspace of $T_p\tilde{G}$, the
proof of 1. is complete. 

Since both $N(\tilde{G}_{\mathcal{O}}/\tilde{G})_x$ and $J_p\big(\mathfrak{t}_x'(p)\big)$
have dimension $r_{\mathcal{O}}-1$, to prove 2. it suffices 
to show that $N(\tilde{G}_{\mathcal{O}}/\tilde{G})_x\supseteq J_x\big(\mathfrak{t}_x'(x)\big)$.
If $w\in T_x\tilde{G}_{\mathcal{O}}$, then 
$$
\mathrm{d}_x\tilde{\Phi}(w)\in T_{\tilde{\Phi}(x)}\mathcal{C}(\mathcal{O})=\mathbb{R}\,\tilde{\Phi}(x)\oplus T_{\tilde{\Phi}(x)}\mathcal{O}
\subseteq \mathfrak{g}^\vee.
$$
Hence there exist $a\in \mathbb{R}$ and $\boldsymbol{\eta}\in \mathfrak{g}$
such that 
\begin{equation}
\label{eqn:dPhi w}
\mathrm{d}_x\tilde{\Phi}(w)=
a\,\tilde{\Phi}(x)
+\mathrm{coad}_{\boldsymbol{\eta}}\big(\tilde{\Phi}(x)\big).
\end{equation}
Suppose $\boldsymbol{\xi}\in \mathfrak{t}_x'$.
Then
\begin{eqnarray*}
\hat{\kappa}_x\left(J_x\big(\boldsymbol{\xi}_{\tilde{G}}(x)\big),w\right)
&=&\Omega_x\left(\boldsymbol{\xi}_{\tilde{G}}(x),w\right)
=\mathrm{d}_x\tilde{\varphi}^{\boldsymbol{\xi}}(w)=
\left\langle\mathrm{d}_x\tilde{\Phi}(w),\boldsymbol{\xi}\right\rangle\\
&=&\left\langle a\,\tilde{\Phi}(x)
+\mathrm{coad}_{\boldsymbol{\eta}}\big(\tilde{\Phi}(x)\big),\boldsymbol{\xi}\right\rangle\\
&=&\kappa_e\left(a\,\tilde{\Phi}(x)^\kappa
+\left[\boldsymbol{\eta},\tilde{\Phi}(x)^\kappa\right],\boldsymbol{\xi}\right)\\
&=&a\,\kappa_e\left(\tilde{\Phi}(x)^\kappa,\boldsymbol{\xi}\right)
+\kappa_e\left(\boldsymbol{\eta},
\left[\tilde{\Phi}(x)^\kappa,\boldsymbol{\eta}\right]\right)=0.
\end{eqnarray*}
Hence $J_x\big(\boldsymbol{\xi}_{\tilde{G}}(x)\big)\in 
N(\tilde{G}_{\mathcal{O}}/\tilde{G})_x$.

\end{proof}

\begin{cor}
\label{cor:normal bundle Xtau}
The normal bundle 
of $X^\tau_{\mathcal{O}}$ in $X^\tau$ is given by
$$
N\left(X^\tau_{\mathcal{O}}/X^\tau\right)_x=
J_x\left( \mathfrak{t}_x'(x)  \right)
\qquad (x\in X^\tau_{\mathcal{O}}).
$$
\end{cor}

\subsection{The geodesic flow on $TG$ and $\tilde{G}$}

Let us consider the trivialization of $TG$ given by \textit{right}
translations: 
$$
\Psi:(g,\boldsymbol{\xi})\in 
G\times \mathfrak{g}\mapsto (g,\boldsymbol{\xi}\cdot g)\in TG.
$$
Passing to the differential, we obtain a diffeomorphism
$$
\mathrm{d}\Psi^{-1}:T(TG)\rightarrow TG\times T\mathfrak{g}\cong
TG\times (\mathfrak{g}\times \mathfrak{g}),
$$
Composing with 
$\Psi^{-1}\times \mathrm{id}_{\mathfrak{g}\times \mathfrak{g}}:
TG\times (\mathfrak{g}\times \mathfrak{g})\rightarrow
(G\times \mathfrak{g})\times (\mathfrak{g}\times \mathfrak{g})$
yields another diffeomorphism
$$
(\Psi^{-1}\times \mathrm{id}_{\mathfrak{g}\times \mathfrak{g}})\circ 
\mathrm{d}\Psi^{-1}:T(TG)\rightarrow 
(G\times \mathfrak{g})\times (\mathfrak{g}\times \mathfrak{g}).
$$
Explicitly, the element of $T(TG)$ corresponding to 
$\big((g,\boldsymbol{\eta}'),(\boldsymbol{\xi},\boldsymbol{\eta}'')\big)$
is 
$$
\left.\frac{\mathrm{d}}{\mathrm{d}t}
\left(e^{t\,\boldsymbol{\eta}'}\,g,\left(\boldsymbol{\xi}+t\,
\boldsymbol{\eta}''\right)\cdot e^{t\,\boldsymbol{\eta}'}\,g\right)
\right|_{t=0}.
$$

The geodesic in $G$ 
with initial condition $(g,\boldsymbol{\xi}\cdot g)\in TG$
is 
$$\gamma_{(g,\boldsymbol{\xi})}:t\in \mathbb{R}\mapsto 
e^{t\,\boldsymbol{\xi}}\cdot g\in G.
$$ 
Its velocity lift 
$\dot{\gamma}_{(g,\boldsymbol{\xi})}:\mathbb{R}\rightarrow TG$ 
is 
$$
\dot{\gamma}_{(g,\boldsymbol{\xi})}(t)=
\left(e^{t\,\boldsymbol{\xi}}\cdot g, \boldsymbol{\xi}\cdot 
\left( e^{t\,\boldsymbol{\xi}}\cdot g \right)   \right)\quad 
(t\in \mathbb{R}),
$$
and is the integral curve of the geodesic vector field 
$\Gamma^G\in \mathfrak{X}(TG)$ passing through 
$(g,\boldsymbol{\xi}\cdot g)$ for $t=0$. In particular,
$$
\Gamma^G(g,\boldsymbol{\xi}\cdot g)=
\left.\frac{\mathrm{d}}{\mathrm{d}t}
\dot{\gamma}_{(g,\boldsymbol{\xi})}(t)\right|_{t=0}.
$$
On the other hand, 
$\Upsilon_{(g,\boldsymbol{\xi})}:=\Psi^{-1}\circ \dot{\gamma}_{(g,\boldsymbol{\xi})}:
\mathbb{R}\rightarrow G\times \mathfrak{g}$ is given by
$$
\Upsilon_{(g,\boldsymbol{\xi})}(t)
=\left(e^{t\,\boldsymbol{\xi}}\cdot g, \boldsymbol{\xi}  \right).
$$
Hence, 
$$
\dot{\Upsilon}_{(g,\boldsymbol{\xi})}(0)=
\mathrm{d}\Psi^{-1}\left(\Gamma^G(g,\boldsymbol{\xi}\cdot g)\right)
=
\left(\left( g,
\boldsymbol{\xi}\cdot  g\right), (\boldsymbol{\xi},0)\right)
\in TG\times (\mathfrak{g}\times \mathfrak{g}).
$$
Therefore, 
\begin{equation}
\label{eqn:geod vf}
(\Psi^{-1}\times \mathrm{id})\circ 
\mathrm{d}\Psi^{-1}\left(\Gamma^G(g,\boldsymbol{\xi}\cdot g)\right)
=\left( (g,
\boldsymbol{\xi}), (\boldsymbol{\xi},0)\right)
\in (G\times \mathfrak{g})
\times (\mathfrak{g}\times \mathfrak{g}).
\end{equation}

Since $\Psi$ intertwines the action $\mathrm{d}L$  of $G$ on $TG$ 
with the action on $G\times \mathfrak{g}$
given by
$$
\hat{L}_h(g,\boldsymbol{\xi}):=\big(h\,g,
\mathrm{Ad}_h(\boldsymbol{\xi})\big),
$$
the composed diffeomorphism
$$
(\Psi^{-1}\times \mathrm{id})\circ 
\mathrm{d}\Psi^{-1}:T(TG)\rightarrow (G\times \mathfrak{g})
\times (\mathfrak{g}\times \mathfrak{g})
$$
is also $G$-equivariant.
Thus for any $\boldsymbol{\eta}\in \mathfrak{g}$ the induced vector
fields $\boldsymbol{\eta}_{TG}$ and 
$\boldsymbol{\eta}_{G\times \mathfrak{g}}$ are $\Psi$-correlated,
and abusing notation we may view the latter as a map
$G\times \mathfrak{g}\rightarrow (G\times \mathfrak{g})\times
(\mathfrak{g}\times \mathfrak{g})$. Hence
\begin{eqnarray}
\label{eqn:Psi corr eta}
\lefteqn{(\Psi^{-1}\times \mathrm{id})\circ 
\mathrm{d}\Psi^{-1}\left(\boldsymbol{\eta}_{TG}(g,\boldsymbol{\xi}\cdot g)\right)
=(\Psi^{-1}\times \mathrm{id})
\left(\boldsymbol{\eta}_{G\times \mathfrak{g}}(g,\boldsymbol{\xi}) \right)}\\
&=&(\Psi^{-1}\times \mathrm{id})\big((g,\boldsymbol{\eta}\cdot g),(\boldsymbol{\xi},
[\boldsymbol{\eta},\boldsymbol{\xi}] )   \big)
=\big((g,\boldsymbol{\eta}),(\boldsymbol{\xi},
[\boldsymbol{\eta},\boldsymbol{\xi}] )   \big)
\in (G\times \mathfrak{g})
\times (\mathfrak{g}\times \mathfrak{g}).\nonumber
\end{eqnarray}
In particular, with $\boldsymbol{\xi}=\boldsymbol{\eta}$ 
in view of (\ref{eqn:geod vf}) we obtain
\begin{eqnarray}
\label{eqn:eta xi geod}
(\Psi^{-1}\times \mathrm{id})\circ 
\mathrm{d}\Psi^{-1}\left(
\boldsymbol{\xi}_{TG}(g,\boldsymbol{\xi}\cdot g)\right)&=&\big((g,\boldsymbol{\xi}),(\boldsymbol{\xi},
0 )   \big)\\
&=&(\Psi^{-1}\times \mathrm{id})\circ 
\mathrm{d}\Psi^{-1}\left(\Gamma^G(g,\boldsymbol{\xi}\cdot g)\right).
\nonumber
\end{eqnarray}
We conclude the following.

\begin{lem}
\label{lem:geod vf and act}
For any $(g,\boldsymbol{\xi}\cdot g)\in TG$, we have
$$
\Gamma^G(g,\boldsymbol{\xi}\cdot g)=
\boldsymbol{\xi}_{TG}(g,\boldsymbol{\xi}\cdot g),
$$
where $\Gamma^G,\,\boldsymbol{\xi}_{TG}\in \mathfrak{X}(TG)$
are, respectively, the geodesic vector field and 
the vector field induced by $\boldsymbol{\xi}$ under
the action $\mathrm{d}L$. In particular, $\Gamma^G$ is tangent
to the $G$-orbits and the geodesic flow of any $p\in TG$ is
contained in its $G$-orbit.
\end{lem}

\begin{cor}
The $\mathbb{R}$-orbit of any non-zero $p\in TG$ under the homogeneous geodesic flow
is contained in its $G$-orbit under $\mathrm{d}L$.
\label{cor:geodesic flow TG}
\end{cor}

\begin{rem}
\label{rem:equiv form geod}
Lemma \ref{lem:geod vf and act} may be reformulated as saying that
$$
\Gamma^G(g,g\cdot \boldsymbol{\xi})=
\mathrm{Ad}_g(\boldsymbol{\xi})_{TG}(g,g\cdot \boldsymbol{\xi})
$$
\end{rem}

The $G$-equivariant diffemorphism $E:TG\rightarrow \tilde{G}$ in
(\ref{eqn:defn of gamma'}) intertwines the homogeneous geodesic
flow on $TG$ (which is the Hamiltonian flow of $\|\cdot \|_\kappa$
with respect to $\Omega_{\mathrm{can}}$) with the Hamiltonian
flow of $\sqrt{\rho}$ with respect to $\Omega$. 
Therefore, 
Corollary \ref{cor:geodesic flow TG} may be restated as follows.  

\begin{cor}
The $\mathbb{R}$-orbit of any $x\in \tilde{G}\setminus G$ under the 
flow of $\upsilon_{\sqrt{\rho}}$
is contained in its $G$-orbit under $\tilde{L}$.
\label{cor:geodesic flow TG 1}
\end{cor}

Somewhat more precisely, let $\tilde{\Gamma}^G\in \mathfrak{X}(\tilde{G})$ be the vector field
on $\tilde{G}$ correlated to $\Gamma^G$ by $E$. We shall refer to $\tilde{\Gamma}^G$ as the
geodesic vector field on $\tilde{G}$.
By Remark \ref{rem:equiv form geod},
\begin{equation}
\label{eqn:geod vct fld tildeG}
\tilde{\Gamma}^G\left(g\cdot\exp^{\tilde{G}}(\imath\,\boldsymbol{\xi})  \right)
=\mathrm{Ad}_g(\boldsymbol{\xi})_{\tilde{G}}
\left(g\cdot\exp^{\tilde{G}}(\imath\,\boldsymbol{\xi})  \right),
\end{equation}
for all $g\in G$ and $\boldsymbol{\xi}\in \mathfrak{g}$.
Furthermore, $\tilde{\Gamma}^G$ is a multiple of $\upsilon_{\sqrt{\rho}}$, hence of $\mathcal{R}$ in (\ref{eqn:xiG horiz R}); in particular, it spans
$\mathcal{V}^\tau$.
It follows from (\ref{eqn:moment map tilde G}),
(\ref{eqn:xiG horiz R}),
and (\ref{eqn:geod vct fld tildeG}) that 
\begin{eqnarray}
\label{eqn:Adgxi val}
\lefteqn{ \mathrm{Ad}_g(\boldsymbol{\xi})_{\tilde{G}}\left(g\,\exp^{\tilde{G}}(\imath\,\boldsymbol{\xi})  \right)  }\\
&=&-\left\langle\tilde{\Phi} \left(g\cdot\exp^{\tilde{G}}(\imath\,\boldsymbol{\xi})  \right),
\mathrm{Ad}_g(\boldsymbol{\xi})_{\tilde{G}}  \right\rangle\,
\mathcal{R}\left(g\cdot\exp^{\tilde{G}}(\imath\,\boldsymbol{\xi})  \right)
\nonumber\\
&=&-\|\boldsymbol{\xi}\|^2_\kappa\,\mathcal{R}\left(g\cdot\exp^{\tilde{G}}(\imath\,\boldsymbol{\xi})  \right)\nonumber
\end{eqnarray}
(notice that $\|\boldsymbol{\xi}\|_\kappa=\tau$ if 
$g\cdot\exp^{\tilde{G}}(\imath\,\boldsymbol{\xi}) \in X^\tau$).

\subsection{Direct sum decompositions of $\mathfrak{g}$}

\label{sctn:direct sum dec g}

In \S \ref{sctn:normal bundle}, we have denoted by 
$T_{\boldsymbol{\lambda}}\leqslant G$ the stabilizer subgroup
of a given $\boldsymbol{\lambda}\in \mathfrak{g}^\vee$ under the coadjoint action, and
by $\mathfrak{t}_{\boldsymbol{\lambda}}\leqslant \mathfrak{g}$
its Lie algebra. With slight 
ambiguity, we shall also denote by $T_{\boldsymbol{\xi}}\leqslant G$
the stabilizer subgroup of a given $\boldsymbol{\xi}\in \mathfrak{g}$
under the adjoint action,
by $r_{\boldsymbol{\xi}}$ its dimension, 
and
by $\mathfrak{t}_{\boldsymbol{\xi}}\leqslant \mathfrak{g}$
its Lie algebra.

Similarly, we shall set
\begin{equation}
\label{eqn:t'xi}
\mathfrak{t}'_{\boldsymbol{\xi}}:=
\mathfrak{t}_{\boldsymbol{\xi}}\cap \boldsymbol{\xi}^{\perp_{\kappa_e}}=
\left\{\boldsymbol{\eta}\in \mathfrak{g}:
[\boldsymbol{\eta},\boldsymbol{\xi}]=0,\,\kappa_e(\boldsymbol{\eta},\boldsymbol{\xi})=0\right\}.
\end{equation}

For any $x\in \tilde{G}$, 
there are two natural choices of a Euclidean product on $\mathfrak{g}$.
One is the given $\kappa_e$, and the other is the pull-back
$\sigma_x:=\mathrm{val}_x^t(\hat{\kappa}_x)$ of
$\hat{\kappa}_x$ under the injective linear map given by
vector field evaluation at $x$: 
$$
\mathrm{val}_x:\boldsymbol{\eta}\in \mathfrak{g}\mapsto
\boldsymbol{\eta}_{\tilde{G}}(x)\in T_x\tilde{G}.
$$
If $x=g\cdot\exp^{\tilde{G}}(\imath\,\boldsymbol{\xi})$, where
$g\in G$ and $\boldsymbol{\xi}\in \mathfrak{g}$, then 
$\sigma_x$ depends only on $\boldsymbol{\xi}$, 
since $G$ acts isometrically on $\tilde{G}$.
In particular, $\kappa_e=\sigma_x$ when $x\in G$ (i.e., when
$\boldsymbol{\xi}=0$).

Given $\boldsymbol{\eta}\neq 0\in \mathfrak{g}$, the orthocomplements
$\boldsymbol{\eta}^{\perp_{\kappa_e}},\,\boldsymbol{\eta}^{\perp_{\sigma_x}}
\subset \mathfrak{g}$ are two \textit{a priori} distinct hyperplanes.

\begin{lem}
\label{lem:same orthocomplement}
If $g\in G$, $\boldsymbol{\xi}\in \mathfrak{g}\setminus \{0\}$ and $x=g\cdot\exp^{\tilde{G}}(\imath\,\boldsymbol{\xi})$, then
we have the following equality of Euclidean orthocomplements:
$$
\mathrm{Ad}_g(\boldsymbol{\xi})^{\perp_{\kappa_e}}
=\mathrm{Ad}_g(\boldsymbol{\xi})^{\perp_{\sigma_x}}.
$$
\end{lem}

\begin{proof}
Let us set $\tau:=\|\boldsymbol{\xi}\|_{\kappa_e}$,
so that $x\in X^\tau$. By (\ref{eqn:xiG horiz R}), for any
$\boldsymbol{\eta}\in \mathfrak{g}$
\begin{equation}
\label{eqn:xiG horiz Rx}
\boldsymbol{\eta}_{\tilde{G}}(x)=\boldsymbol{\eta}_{\tilde{G}}^\sharp(x)
-\kappa_e\big(\mathrm{Ad}_g(\boldsymbol{\xi}),\boldsymbol{\eta}  \big)
\,\mathcal{R}(x).
\end{equation}
Thus one the one hand
$\mathrm{Ad}_g(\boldsymbol{\xi})_{\tilde{G}}\left(x  \right)\in \mathcal{T}^\tau_x$ by (\ref{eqn:Adgxi val}), and on the other
$\boldsymbol{\eta}_{\tilde{G}}(x)=\boldsymbol{\eta}_{\tilde{G}}^\sharp(x)
\in \mathcal{H}^\tau_x$ for any $\boldsymbol{\eta}\in 
\mathrm{Ad}_g(\boldsymbol{\xi})^{\perp_{\kappa_e}}$.
Since (\ref{eqn:decomp Xtau}) is an orthogonal direct sum, the claim follows.

\end{proof}

By its definition in  (\ref{eqn:t'xi}), $\mathfrak{t}'_{\mathrm{Ad}_g(\boldsymbol{\xi})}\subset
\mathrm{Ad}_g(\boldsymbol{\xi})^{\perp_{\kappa_e}}$.

\begin{defn}
\label{defn:sfrakx}
If $x=g\cdot\exp^{\tilde{G}}(\imath\,\boldsymbol{\xi})\in \tilde{G}\setminus G$, so that
$\tilde{\Phi}(x)=\mathrm{Coad}_g(\boldsymbol{\xi}_\kappa)$,
\begin{enumerate}
\item we shall set 
$\mathfrak{r}_x:=
(\mathfrak{t}_x)^{\perp_{\kappa_e}}$,
so that (recalling Remark \ref{rem:tlambda'})
$$
\mathrm{Ad}_g(\boldsymbol{\xi})^{\perp_{\kappa_e}}
=\mathfrak{t}'_{\mathrm{Ad}_g(\boldsymbol{\xi})}
\oplus_{\kappa_e}
\mathfrak{r}_x,
$$
that is, the direct sum is $\kappa_e$-orthogonal;

\item we shall denote by
$\mathfrak{s}_x\subseteq \mathrm{Ad}_g(\boldsymbol{\xi})^{\perp_{\kappa_e}}$
the vector subspace such that
$$
\mathrm{Ad}_g(\boldsymbol{\xi})^{\perp_{\kappa_e}}
=\mathfrak{t}'_{\mathrm{Ad}_g(\boldsymbol{\xi})}
\oplus_{\sigma_x}
\mathfrak{s}_x,
$$
that is, the direct sum is $\sigma_x$-orthogonal.

\end{enumerate}

\end{defn}

Clearly,
$$
\dim \mathfrak{s}_x=\dim (\mathrm{Ad}_g(\boldsymbol{\xi})^{\perp_{\kappa_e}})-
\dim(\mathfrak{t}'_{\mathrm{Ad}_g(\boldsymbol{\xi})})=
(d-1)-(r_{\boldsymbol{\xi}}-1)=
d-r_{\boldsymbol{\xi}}.
$$

\begin{lem}
\label{lem:s_x orthocomplement}
Under the previous assumptions, 
$$
\mathfrak{s}_x(x)\oplus J_x\big( \mathfrak{s}_x(x) \big)
=\left[ \mathfrak{t}_x(x)\oplus J_x\big(
\mathfrak{t}_x(x)\big)\right]^{\perp_{\hat{\kappa}_x}}\cap \mathcal{H}_x.
$$
\end{lem}

\begin{proof}
We have 
$\mathfrak{t}_x(x)=\mathrm{span}\big(
\mathrm{Ad}_g(
\boldsymbol{\xi})_{X^\tau}(x)
\big)
\oplus \mathfrak{t}'_x(x)$. 
Therefore, by Definition \ref{defn:sfrakx} 
$\mathfrak{s}_x(x)\subseteq \mathfrak{t}_x(x)^{\perp_{\hat{\kappa}_x}}$.
On the other hand, since $X^\tau_{\mathcal{O}}$
is $G$-invariant $\mathfrak{g}_{X^\tau}(x)\subseteq
T_xX^\tau_{\mathcal{O}}$.
Hence, $\mathfrak{s}_x(x)\subseteq \mathfrak{g}_{X^\tau}(x)
\subseteq 
J_x\big( \mathfrak{t}'_x(x) \big)^{\perp_{\hat{\kappa}_x}}$
by Corollary \ref{cor:normal bundle Xtau}.
Thus
\begin{eqnarray}
\label{eq:sxJxsx}
\lefteqn{\mathfrak{s}_x(x)\subseteq \left[\mathfrak{t}'_x(x)
\oplus J_x\big( \mathfrak{t}'_x(x) \big)\right]
^{\perp_{\hat{\kappa}_x}}}\nonumber\\
&\Rightarrow&
\mathfrak{s}_x(x)\oplus 
J_x\big(\mathfrak{s}_x(x)   \big)\subseteq \left[\mathfrak{t}'_x(x)
\oplus J_x\big( \mathfrak{t}'_x(x) \big)\right]
^{\perp_{\hat{\kappa}_x}}\cap \mathcal{H}_x.
\end{eqnarray}

On the other hand, 
$\mathfrak{s}_x(x)\oplus 
J_x\big(\mathfrak{s}_x(x)   \big)$ has real dimension
$2\,\dim \mathrm{s}_x=2\,(d-r_{\boldsymbol{\xi}})$,
and 
$\left[\mathfrak{t}'_x(x)
\oplus J_x\big( \mathfrak{t}'_x(x) \big)\right]
^{\perp_{\hat{\kappa}_x}}\cap \mathcal{H}_x$ has 
real dimension 
$2d-2-2\,(r_{\boldsymbol{\xi}}-1)=
2\,(d-r_{\boldsymbol{\xi}})$.
Hence equality holds in (\ref{eq:sxJxsx}), and the
claim follows since $\mathfrak{t}'_x(x)=
\mathfrak{t}_x(x)\cap \mathcal{H}_x$.
\end{proof}

Thus we have direct sum decompositions
\begin{eqnarray}
\label{eqn:direct sums T_x}
\mathcal{H}_x&=&\left[\mathfrak{t}'_x(x)
\oplus_{\hat{\kappa}_x} J_x\big( \mathfrak{t}'_x(x) \big)\right]\oplus 
_{\hat{\kappa}_x}
\left[\mathfrak{s}_x(x)\oplus J_x\big(\mathfrak{s}_x(x)  \big)   \right]\\
&=&(\tilde{\Phi}(x)^{\perp_{\kappa_e}}_{\tilde{G}})(x)
\oplus J_x\left((\tilde{\Phi}(x)^{\perp_{\kappa_e}}_{\tilde{G}})(x)\right),\nonumber\\
T_x\tilde{G}_{\mathcal{O}}\cap \mathcal{H}_x&=&
\mathfrak{t}'_x(x)
\oplus_{\hat{\kappa}_x}\left[\mathfrak{s}_x(x)\oplus J_x\big(\mathfrak{s}_x(x)  \big)   \right],\nonumber
\end{eqnarray}
where $\oplus_{\hat{\kappa}_x}$ denotes
$\hat{\kappa}_x$-orthogonality. 

\begin{rem}
\label{rem:NSsigma}
In terms of (\ref{eqn:dsdecomp HSN}), we have
$$
\mathcal{N}_x=\mathfrak{t}'_x(x)
\oplus_{\hat{\kappa}_x} J_x\big( \mathfrak{t}'_x(x) \big)=
\widetilde{\mathfrak{t}'}_x(x),
\qquad
\mathcal{S}_x=\mathfrak{s}_x(x)\oplus J_x\big(\mathfrak{s}_x(x)\big)=
\widetilde{\mathfrak{s}}_x(x),
$$
where $\widetilde{\mathfrak{t}'}_x$ and $\widetilde{\mathfrak{s}}_x$
are the complexifications of $\mathfrak{t}'_x$ and $\mathfrak{s}_x$,
respectively.
\end{rem}

\subsection{$\Pi^\tau$ and $P^\tau$}
\label{scnt:szego parametrix}

Let us dwell on the key properties of the operators $\Pi^\tau$ and $P^\tau$ discussed in the Introduction that will be used in the following arguments.

\subsubsection{The description by Fourier integral operators}

The Szeg\"{o} projector $\Pi^\tau:L^2(X^\tau)\rightarrow H(X^\tau)$
is a Fourier integral operator with complex phase of positive type, whose
microlocal structure has been precisely determined in \cite{bs}
(see also \cite{feff}, \cite{ms}, \cite{bg}, \cite{sz}). 
In particular,
up to a smoothing term its Schartz kernel 
$\Pi^\tau\in \mathcal{D}'(X^\tau\times X^\tau)$
can be written in the form
\begin{equation}
\label{eqn:szego ker}
\Pi^\tau(x,y)\sim \int_0^{+\infty}e^{\imath\,u\,\psi^\tau(x,y)}\,
s^\tau(x,y,u)\,\mathrm{d}u,
\end{equation}
where the phase $\psi^\tau$ has non-negative imaginary part and the
amplitude $s^\tau$ is a semiclassical symbol admitting
an asymptotic expansion of the form
\begin{equation}
\label{eqn:amplitude pi}
s^\tau(x,y,u)\sim \sum_{j\ge 0}u^{d-1-j}\,s_j^\tau(x,y).
\end{equation}

The phase $\psi^\tau$ is only determined up to a function vanishing to infinite order along the diagonal. In particular, for any
$x\in X^\tau$ we have 
\begin{equation}
\label{eqn:diagonal diff psi}
\mathrm{d}_{(x,x)}\psi^\tau=(\alpha^\tau_x,-\alpha^\tau_x).
\end{equation}
Furthermore, 
the imaginary part of $\psi^\tau$ may be assumed to satisfy
\begin{equation}
\label{eqn:bd dist psitau xy}
\Im \left(\psi^\tau(x,y)\right)\ge C^\tau\,\mathrm{dist}_{X^\tau}
(x,y)^2\quad (x,y\in X^\tau)
\end{equation}
for an appropriate constant $C^\tau>0$.

Let $\Sigma^\tau$ be the closed symplectic cone in $TX^\tau\setminus (0)$
($(0)$ denoting the zero section)
sprayed by the contact form $\alpha^\tau$ 
(see \S \ref{sctn:splitting alpha}):
$$
\Sigma^\tau:=\left\{\left(x,r\,\alpha_x^\tau\right)\,:\,x\in X^\tau,
\,r>0\right\}.
$$
The wave front of $\Pi^\tau$ is the anti-diagonal of $\Sigma^\tau$:
\begin{eqnarray}
\mathrm{WF}(\Pi^\tau)=
{\Sigma^\tau}^\sharp:=\left\{\left(x,r\,\alpha_x^\tau,x,-r\,\alpha_x^\tau
\right)\,:\,x\in X^\tau,\,r>0\right\}
\end{eqnarray}
Therefore, the singular support is the diagonal in $X^\tau\times X^\tau$:
$$
\mathrm{S.S.}(\Pi^\tau)=\mathrm{diag}(X^\tau\times X^\tau)
$$
(see \cite{bs}
for discussion and derivation of these properties)

As proved by Zelditch, the operator $P^\tau$
in (\ref{eqn:Poisson-wave})
is also a Fourier integral operator with complex
phase, associated to the same canonical relation as $\Pi^\tau$ - hence with the same wave front and singular support - 
and of
degree $-(d-1)/2$ (see e.g. \cite{z12}, \cite{z14} and \cite{z20}, and the discussions in \cite{cr1} and \cite{cr2}).
In fact $P^\tau=U_\mathbb{C}(2\,\imath\,\tau)$,
where for general $t\in \mathbb{R}$ one denotes
$ U_\mathbb{C}(t+2\,\imath\,\tau):\mathcal{O}(X^\tau)
\rightarrow \mathcal{O}(X^\tau)$ the complexified
Poisson-wave operators obtained (roughly speaking) by
holomorphically extending the operator kernel of the
wave operator. 
Zelditch showed that complexified 
Poisson-wave operators may described
in terms of dynamical Toeplitz operators;
one can then conclude that
$P^\tau$ is a Toeplitz operator of degree
$-(d-1)/2$, given by the compression 
by $\Pi^\tau$ of a pseudodifferential
operator $Q^\tau$ of the same degree. 
To leading order, $Q^\tau$
has the form $(\pi\,\tau)^{\frac{d-1}{2}} \cdot D_{\sqrt{\rho}}^{-(d-1)/2}$
(\textit{cfr.} the discussion in \S 5 of \cite{gp24}).
It follows that
the operator kernel of $P^\tau$ has the form
\begin{equation}
\label{eqn:Ptau ker}
P^\tau(x,y)\sim \int_0^{+\infty}e^{\imath\,u\,\psi^\tau(x,y)}\,
q^\tau(x,y,u)\,\mathrm{d}u,
\end{equation}
where 
\begin{equation}
\label{eqn:asy exp Ptau}
q^\tau(x,y,u)\sim \sum_{j\ge 0}u^{\frac{d-1}{2}-j}\,q_j(x,y),
\quad\text{where}\quad
q^\tau_0(x,y)=\pi^{\frac{d-1}{2}}\,s^\tau_0(x,y).
\end{equation}

In the present real-analytic
setting, 
there is a natural choice for $\psi^\tau$.
Namely, $\phi^\tau:=\rho-\tau^2:\tilde{G}\rightarrow \mathbb{R}$ is 
real-analytic defining funtion for $X^\tau$; let $\tilde{\phi}^\tau$
denote the holomorphic extension of $\phi^\tau$ to
$\tilde{G}\times \overline{\tilde{G}}$ (defined at least on 
a neighbourhood of the diagonal). Then we can set
\begin{equation}
\label{eqn:psi in terms of phi}
\psi^\tau:=\frac{1}{\imath}\left.\tilde{\phi}^\tau
\right|_{X^\tau\times X^\tau}.
\end{equation}

\subsubsection{Normal Heisenberg local coordinates on $X^\tau$}
\label{scnt:heis coord}

To perform computations involving $\Pi^\tau$ and $P^\tau$,
it is convenient to work in specific systems 
of local coordinates on $X^\tau$.
For any $\tau>0$ and $x\in X^\tau$, there exist systems of holomorphic
local coordinates for $\tilde{G}$ centered at $x$ in which the defining
equation of $X^\tau$ has a certain canonical form;
in turn, these induce local coordinates on $X^\tau$ centered at
$x$, in which (given (\ref{eqn:psi in terms of phi}))
$\psi^\tau$ admits a relatively explicit form.

These coordinates are called \textit{normal Heisenberg local coordinates}
and will be referred to as NHLC's in the following
(here we shall follow the notation and conventions in [P 2024],
and refer the
reader to \cite{fs1}, \cite{fs2}, \cite{cr1}, \cite{cr2}, \cite{sz}
for general background and discussion). 

We shall refer as needed to \S 3 of [Pao 2024], in particular Propositions 34 and 48. Furthermore, in NHLC's centered at $x$, $s_0^\tau$ in
(\ref{eqn:amplitude pi}) satisfies
\begin{equation}
\label{eqn:valore s0 xx}
s_0^\tau(x,x)=\frac{\tau}{(2\,\pi)^d}
\end{equation}
(Theorem 51 of [Pao 2024]).
NHLC's centered at $x\in X^\tau$ will be 
denoted in additive notation, in the form
$x+(\theta,\mathbf{v})$, where $(\theta,\mathbf{v})\in \mathbb{R}\times \mathbb{R}^{2d-2}$ belong to a small ball centered at the origin.

\section{Proofs of the Theorems}

Let us premise the following general remark.
Let $\mathcal{Q}^\tau_{\boldsymbol{\lambda}}:L^2(X^\tau)
\rightarrow L^2(X^\tau)_{\boldsymbol{\lambda}}$ be the orthogonal
projector, so that $\Pi^\tau_{\boldsymbol{\lambda}}=\mathcal{Q}^\tau_{\boldsymbol{\lambda}}
\circ \Pi^\tau$, and let $\mu^\tau$ 
be as in (\ref{eqn:mutau defn}).
Then the relation between the Schwartz kernels of
$\Pi^\tau$ and 
$\Pi^\tau_{\boldsymbol{\lambda}}$ is given
by 
\begin{equation}
\label{eqn:equivariant szego}
\Pi^\tau_{\boldsymbol{\lambda}}(x,y):=
d_{\boldsymbol{\lambda}}\cdot 
\int_G \overline{\chi_{\boldsymbol{\lambda}}(g)}\,
\Pi^\tau
\left(\mu^\tau_{g^{-1}}(x),y   \right)\,
\mathrm{d}^HV_G(g),
\end{equation}
where $\mathrm{d}^H V_G$ denotes the Haar measure on
$G$. A similar relation holds between $P^\tau$
and $P^\tau_{\boldsymbol{\lambda}}$.

\begin{notn}
We shall generally denote the coupling of elements in $\boldsymbol{\beta}\in\mathfrak{g}^\vee$ and $\boldsymbol{\xi}\in \mathfrak{g}$ by
$\boldsymbol{\beta}(\boldsymbol{\xi})$, and of elements
$\boldsymbol{\lambda}\in \mathfrak{t}^\vee$ and 
$\boldsymbol{\vartheta}\in \mathfrak{t}$ by
by $\langle\boldsymbol{\lambda},\boldsymbol{\vartheta}\rangle$.
\end{notn}

\subsection{Proof of Theorem \ref{thm:rapid decay compact}}
\label{sctn:rapid decay compact}

\begin{proof}
Given $(x,y)\in X^\tau\times X^\tau$, set
$$
\delta(x,y)
:=
\mathrm{dist}_{X^\tau}(x,G\cdot y)+
\mathrm{dist}_{X^\tau}
\left(x,X^\tau_{\mathcal{O}}\right).
$$
If $K\Subset X^\tau\times X^\tau\setminus \mathcal{Z}^\tau_{\mathcal{O}}$, then 
$$
\delta_K:=\min\{\delta (x,y)\,:\,
(x,y)\in K\}>0.
$$
Thus it suffices to show that for any given 
$\delta_0>0$ the conclusion of the Theorem
holds uniformly on the locus where $\delta (x,y)\ge \delta_0$.

Set  
$K_{\delta_0}:=\{(x,y)\in X^\tau\times X^\tau\,:\,\delta
(x,y)\ge \delta_0\}$
and
$$
U_1:=\left\{(x,y)\in X^\tau\times X^\tau\,:\,
\mathrm{dist}_{X^\tau}
(x,G\cdot y)> \frac{1}{2}\,\delta_0
\right\},
$$
$$
U_2:=\left\{(x,y)\in X^\tau\times X^\tau\,:\,
\mathrm{dist}_{X^\tau}\left(x,X^\tau_{\mathcal{O}}\right)> 
\frac{1}{2}\,\delta_0
\right\};
$$
then $\{U_1,U_2\}$ is an open cover of $K_{\delta_0}$,
and we need only prove that the statement holds
uniformly on each $U_j$.

Let us consider $U_1$. Since the singular support of $\Pi^\tau$ is the diagonal in $X^\tau\times X^\tau$, 
$f_{x,y}(g):=\Pi^\tau
\big(\mu^\tau_{g^{-1}}(x),y   \big)$ is a uniformly
smooth function of $g\in G$ for $(x,y)\in U_1$.
By a theorem of Sugiura \cite{sug}, 
its group-theoretic Fourier transform 
$\mathcal{F}(f_{x,y})$
is 
uniformly rapidly decreasing
as a matrix-valued function on $\mathcal{D}^G$; therefore, so is its trace.

Let $\rho_{\boldsymbol{\lambda}}(g)\in 
\mathrm{GL}(V_{\boldsymbol{\lambda}})$ be the automorphism
associated to $g\in G$ in the representation 
$V_{\boldsymbol{\lambda}}$; assuming the choice of
an orthonormal basis, we may view it as a unitary matrix.
In view of (\ref{eqn:equivariant szego})
\begin{eqnarray}
\label{eqn:equivariant szego klambda}
\Pi^\tau_{k\,\boldsymbol{\lambda}}(x,y)&=&
d_{k\boldsymbol{\lambda}}\,
\int_G \overline{\chi_{k\,\boldsymbol{\lambda}}(g)}\,
\Pi^\tau
\left(\mu^\tau_{g^{-1}}(x),y   \right)\,
\mathrm{d}^HV_G(g)\nonumber\\
&=&d_{k\boldsymbol{\lambda}}\,
\mathrm{trace}\left(
\int_G \rho_{k\,\boldsymbol{\lambda}}(g^{-1})\,
\Pi^\tau
\left(\mu^\tau_{g^{-1}}(x),y   \right)\,
\mathrm{d}^HV_G(g)\right)\nonumber\\
&=&
d_{k\boldsymbol{\lambda}}\,
\mathrm{trace}\left(\mathcal{F}(f_{x,y})(k\,
\boldsymbol{\lambda})\right)=O\left(k^{-\infty}\right).
\end{eqnarray}

Let us consider $U_2$. 
Let $L_{\boldsymbol{\lambda}}=(\boldsymbol{\lambda},\,2\,\boldsymbol{\lambda},
\,\ldots)$ denote the \textit{ladder} of irreducible representations
generated by $\boldsymbol{\lambda}$ \cite{gstb82}. 
Let us consider the corresponding subspaces
$$
L^2(X^\tau)_{L_{\boldsymbol{\lambda}}}:=\bigoplus_{k=1}^{+\infty} 
L^2(X^\tau)_{k\boldsymbol{\lambda}},\qquad
H(X^\tau)_{L_{\boldsymbol{\lambda}}}:=\bigoplus_{k=1}^{+\infty} 
H(X^\tau)_{k\boldsymbol{\lambda}}
$$
with orthogonal projectors
$$
\mathcal{Q}^\tau_{L_{\boldsymbol{\lambda}}}:
L^2(X^\tau)\rightarrow L^2(X^\tau)_{L_{\boldsymbol{\lambda}}},
\qquad
\Pi^\tau_{L_{\boldsymbol{\lambda}}}=\mathcal{Q}^\tau_{L_{\boldsymbol{\lambda}}}
\circ \Pi^\tau :
L^2(X^\tau)\rightarrow H^2(X^\tau)_{L_{\boldsymbol{\lambda}}}.
$$
Therefore the wave front set 
$\mathrm{WF}(\Pi^\tau_{L_{\boldsymbol{\lambda}}})\subseteq \left(T^*X^\tau\setminus (0)\right)
\times \left(T^*X^\tau\setminus (0)\right)$
of $\Pi^\tau_{L_{\boldsymbol{\lambda}}}$
is obtained by composing those of 
$\mathcal{Q}^\tau_{L_{\boldsymbol{\lambda}}}$
and $\Pi^\tau$.
More precisely, the arguments in \S 3.1.2 of \cite{gp19} (based on the theory of
\cite{gstb82}) imply that
\begin{eqnarray}
\label{eqn:wave front piL}
\mathrm{WF}(\Pi^\tau_{L_{\boldsymbol{\lambda}}})\subseteq 
\left\{\left((x,r\,\alpha^\tau_x),(y,\,-r\,\alpha^\tau_y)\right) \,:\, 
x\in X^\tau_{\mathcal{O}},\,r>0,\,y\in T'_{\tilde{\Phi}(x)}\cdot x
 \right\}
\end{eqnarray}
(recall Definition \ref{defn:tlambda'}).
Hence, $\Pi^\tau_{L_{\boldsymbol{\lambda}}}$
is $\mathcal{C}^\infty$ on $\overline{U}_2$; given that $U_2$ is $G$-invariant, 
$h_{x,y}(g):=\Pi^\tau_{L_{\boldsymbol{\lambda}}}\left(\mu^\tau_{g^{-1}}(x),y   \right)$
is $\mathcal{C}^\infty$ on $G$ when $(x,y)\in \overline{U}_2$.
On the other hand, 
$\Pi^\tau_{k\,\boldsymbol{\lambda}}=\mathcal{Q}^\tau_{k\,\boldsymbol{\lambda}}
\circ \Pi^\tau_{L_{\boldsymbol{\lambda}}}$, since we can first project 
onto the full ladder of isotypical components and then to 
the $k\,\boldsymbol{\lambda}$-th. Hence, again by Sugiura's Theorem
(\cite{sug}),
\begin{eqnarray}
\label{eqn:equivariant ladder szego klambda}
\Pi^\tau_{k\,\boldsymbol{\lambda}}(x,y)&=&
d_{k\boldsymbol{\lambda}}\,
\int_G \overline{\chi_{k\,\boldsymbol{\lambda}}(g)}\,
\Pi^\tau_{L_{\boldsymbol{\lambda}}}
\left(\mu^\tau_{g^{-1}}(x),y   \right)\,
\mathrm{d}^HV_G(g)\nonumber\\
&=&d_{k\boldsymbol{\lambda}}\,
\mathrm{trace}\left(
\int_G \rho_{k\,\boldsymbol{\lambda}}(g^{-1})\,
\Pi^\tau_{L_{\boldsymbol{\lambda}}}
\left(\mu^\tau_{g^{-1}}(x),y   \right)\,
\mathrm{d}^HV_G(g)\right)\nonumber\\
&=&
d_{k\boldsymbol{\lambda}}\,
\mathrm{trace}\left(\mathcal{F}(h_{x,y})(k\,
\boldsymbol{\lambda})\right)=O\left(k^{-\infty}\right).
\end{eqnarray}

Since $\Pi^\tau$ and $P^\tau$ are Fourier integral
operators associated to the same canonical relation
(\S \ref{scnt:szego parametrix}), 
the previous arguments apply \textit{verbatim}
to the case of $P^\tau$.

\end{proof}

\subsection{Proof of Theorem \ref{thm:rapid decay orbit}}

Before attacking the proof, we need to lay down some 
basic facts from Lie theory
(standard references are 
\cite{btd}, \cite{var har}, \cite{var lie}).

There exists a a finite covering of Lie groups 
$\mathfrak{p}:
\hat{G}\rightarrow G$, where $\hat{G}$ is isomorphic
to the direct product of a compact torus and 
a simply connected compact semisimple Lie group;
in particular, $\hat{G}$ is also compact (\S V.8 of
\cite{btd}). 
The Lie algebras
of $\hat{G}$ and $G$ 
may be identified by means of
the differential 
$\mathrm{d}_{\hat{e}}\mathfrak{p}:
T_{\hat{e}}\hat{G}\rightarrow T_eG$ ($\hat{e}\in \hat{G}$ and
$e\in G$ are the identity elements).

The inverse image $\hat{T}=\mathfrak{p}^{-1}(T)$
is a maximal torus of $\hat{G}$, and $\mathfrak{p}$
restricts to a covering map $\mathfrak{q}:
\hat{T}\rightarrow T$ of the same degree as
$\mathfrak{p}$ (this is clear of the connected component
through the identity $\hat{T}_0\subseteq \hat{T}$;
on the other hand, $\ker(\mathfrak{p})$ is normal and
discrete, hence central, so that $\ker(\mathfrak{p})
\subset \hat{T}_0=\hat{T}$).
We can similarly identify the Lie algebras 
of $\hat{T}$ and $T$ by 
$\mathrm{d}_{\hat{e}}\mathfrak{q}$.

Let $\hat{\Lambda},\,\Lambda\subset \mathfrak{t}$ be the
lattices of $\hat{T}$ and $T$, respectively.
Then $\hat{\Lambda}\subseteq \Lambda$, and 
$\hat{\Lambda}$ has index in $\Lambda$ equal to the degree
of $\mathfrak{p}$. Thus
$\mathcal{D}^{\hat{G}}\supseteq \mathcal{D}^G$.
Any representation of $G$ pulls back to 
a representation of $\hat{G}$, and the isotypical
decomposition is the same over $G$ and $\hat{G}$,
in view of the previous inclusion.

Let $R_+\subset R$ denote the collection of positive roots
of the pair $(\mathfrak{g},\mathfrak{t})
=(\hat{\mathfrak{g}},\hat{\mathfrak{t}})$,
and set
\begin{equation}
\label{eqn:half sum prs}
\boldsymbol{\delta}:=\frac{1}{2}\,
\sum_{\boldsymbol{\alpha}\in R_+}
\boldsymbol{\alpha}.
\end{equation}
Then $\boldsymbol{\delta}\in \mathcal{D}^{\hat{G}}$.

We shall view the elements of 
Weyl group as 
group automorphisms of $\hat{T}$ (or, depending
on the context of $T$, $\mathfrak{t}$,...). 
If $f:\hat{T}\rightarrow \mathbb{C}$, we shall
set $f^w:=f\circ w^{-1}$.
\begin{defn}
\label{defn:alternator}
Given $f:\hat{T}\rightarrow \mathbb{C}$, 
let us set
$$
\mathrm{Alt}_W(f):=\sum_{w\in W}\,(-1)^w\,
f^w,
$$
where $(-1)^w:=\det(w)\in \{\pm 1\}$ 
is the determinant of $w\in W$ 
as a linear automorphism of $\mathfrak{t}$. 
\end{defn}

\begin{rem}
\label{rem:alternanza}
For any $f:\hat{T}\rightarrow \mathbb{C}$ and
$w\in W$, 
$$
\mathrm{Alt}_W(f^w)=(-1)^w\,\mathrm{Alt}_W(f).
$$
\end{rem}

\begin{defn}
Any $\boldsymbol{\eta}\in \mathcal{D}^{\hat{G}}$
corresponds to a group character
$E_{\boldsymbol{\eta}}:\hat{T}\rightarrow S^1$.
Let us define $A_{\boldsymbol{\eta}},\,\Delta:
\hat{T}\rightarrow \mathbb{C}$ by setting
$$
A_{\boldsymbol{\eta}}:=\mathrm{Alt}_W(E_{\boldsymbol{\eta}})
=\sum_{w\in W}\,(-1)^w\,E_{w(\boldsymbol{\eta})},
\quad \Delta:=A_{\boldsymbol{\delta}}
$$
\end{defn}

Then $\Delta\neq 0$ on the dense open subset
$\hat{T}'\subseteq \tilde{T}$ of regular elements; $\hat{T}'$
is the inverse image in $\hat{T}$ of the open dense subset
$T'\subseteq T$ of regular elements of $T$.

\subsubsection{The Weyl character formula}
\label{sctn:weyl char for}

Suppose 
$\boldsymbol{\lambda}\in \mathcal{D}^G\subseteq \mathcal{D}^{\hat{G}}$.
The Weyl character formula describes the restriction 
$\left.\chi_{\boldsymbol{\lambda}}\right|_T:T\rightarrow \mathbb{C}$ 
in terms of alternating sums
of irreducible characters on $T$ or, more precisely, on the
covering $\hat{T}$.
Set $\boldsymbol{\nu}_{\boldsymbol{\lambda}}:=
\boldsymbol{\lambda}+\boldsymbol{\delta}\in \mathcal{D}^{\hat{G}}$.
Thus we can consider the ratio
\begin{equation}
\label{eqn:ratio weyl char}
\frac{A_{\boldsymbol{\nu}_{\boldsymbol{\lambda}}}}{\Delta}:
\hat{T}'\rightarrow \mathbb{C}.
\end{equation}

\begin{thm}
(Weyl character formula)
The function in (\ref{eqn:ratio weyl char}) admits a unique continuous extension to $\hat{T}$; furthermore, this extended function is the pull-back of $\left.\chi_{\boldsymbol{\lambda}}\right|_T$ by the covering map
$\hat{T}\rightarrow T$.
\end{thm}
(See, e.g., \cite{btd}, \S VI.1).

\subsubsection{The Weyl integration formula}
\label{sctn:weyl integration for}

Given any compact Lie group $K$, $\mathrm{d}^H_KV$ denotes its
Haar measure. If $K'\leqslant K$ is an inclusion of compact Lie groups, there
is an induced Haar measure $\mathrm{d}^{H}V_{K/K'}$ on the
homogenous space $K/K'$.
If $f:T\rightarrow \mathbb{C}$ is contnuous, let us define
$f_T:T\rightarrow \mathbb{C}$ by setting
$$
f_T(t):=\int_{G/T}f\left(g\,t\,g^{-1}\right)\,
\mathrm{d}^HV_{G/T}(g\,T).
$$

Let us consider the $\kappa$-orthogonal direct sum
$\mathfrak{g}=\mathfrak{t}\oplus \mathfrak{t}^\perp$. For any $t\in T$,
the orthogonal automorphism $\mathrm{Ad}_{t^{-1}}:\mathfrak{g}\rightarrow
\mathfrak{g}$ leaves both $\mathfrak{t}$ and $\mathfrak{t}^\perp$
invariant; hence, it restricts to an orthogonal endomorphism
$(\mathrm{Ad}_{\mathfrak{t}^\perp})_{t^{-1}}:\mathfrak{t}^\perp\rightarrow \mathfrak{t}^\perp$, which doesn't have the eigenvalue $1$ whenever
$t$ is regular.
Furthermore,
$\det\big((\mathrm{Ad}_{\mathfrak{t}^\perp})_{t^{-1}}
-\mathrm{Id}_{\mathfrak{t}^\perp})>0$ for any $t\in T'$.

\begin{thm}
(Weyl integration formula)
Under the previous assumptions,
$$
\int_G\,f(g)\,\mathrm{d}^HV_G(g)=\frac{1}{|W|}\,\int_T\,
\det\big((\mathrm{Ad}_{\mathfrak{t}^\perp})_{t^{-1}}
-\mathrm{Id}_{\mathfrak{t}^\perp})\,f_T(t)\,
\mathrm{d}^HV_T(t).
$$
In addition, when $(G,T)=(\tilde{G},\tilde{T})$
$$
\det\big((\mathrm{Ad}_{\mathfrak{t}^\perp})_{t^{-1}}
-\mathrm{Id}_{\mathfrak{t}^\perp})=
|\Delta(t)|^2.
$$
\end{thm}

\subsubsection{The proof}

\begin{proof}
[{Proof of Theorem \ref{thm:rapid decay orbit}}]
We shall give the argument for $\Pi^\tau$, the one 
for $P^\tau$ being essentially the same.
Our starting point is again the relation
\begin{equation}
\label{eqn:Piklambda}
\Pi^\tau_{k\,\boldsymbol{\lambda}}(x,y)=
d_{k\boldsymbol{\lambda}}\,
\int_G \overline{\chi_{k\,\boldsymbol{\lambda}}(g)}\,
\Pi^\tau
\left(\mu^\tau_{g^{-1}}(x),y   \right)\,
\mathrm{d}^HV_G(g).
\end{equation}
By composing with $\mathfrak{p}$, 
we can pull-back $\mu^\tau$ to an action $\hat{\mu}^\tau$ of $\hat{G}$ on $X^\tau$, and
$\chi_{k\,\boldsymbol{\lambda}}$ to the character $\hat{\chi}_{k\,\boldsymbol{\lambda}}$ on 
$\hat{G}$. We
can rewrite (\ref{eqn:Piklambda})
as
\begin{equation}
\label{eqn:Piklambdatilde}
\Pi^\tau_{k\,\boldsymbol{\lambda}}(x,y)=
d_{k\boldsymbol{\lambda}}\,
\int_{\hat{G}} \overline{\hat{\chi}_{k\,\boldsymbol{\lambda}}(g)}\,
\Pi^\tau
\left(\hat{\mu}^\tau_{g^{-1}}(x),y   \right)\,
\mathrm{d}^HV_{\hat{G}}(g).
\end{equation}
By Theorem \ref{thm:rapid decay compact}, we may restrict to the case where
$(x,y)$ belongs to a small neighbourhood of $\mathcal{Z}^\tau_{\mathcal{O}}
\subset X^\tau\times X^\tau$; hence we may assume that $x$ and $y$ belong to a small neighbourhood of
$X^\tau_{\mathcal{O}}$, and therefore that 
$\mu^\tau$ is locally free at $x$ and $y$.

Perhaps after replacing $(x,y)$ by $\left(\mu^\tau_h(x),\mu^\tau_h(y)\right)$
for some $h\in G$,
we may assume that $\Phi^\tau(y)$ belongs to a small conic 
neighourhood in $\mathfrak{g}^\vee$ of the ray 
$\mathbb{R}_+\,\boldsymbol{\lambda}$. Hence
\begin{equation}
\label{eqn:Phitauy}
\Phi^\tau(y)=a\,\boldsymbol{\lambda}+
\boldsymbol{\beta},
\end{equation}
where $\boldsymbol{\beta}\in \boldsymbol{\lambda}^\perp$ and
$\|\boldsymbol{\beta}\|\ll \|\Phi^\tau(y)\|$.

Furthermore, let us fix a suitably small $\delta>0$,
and consider a cut-off function
$\rho_1(g,x,y)$ on $\hat{G}\times X^\tau\times X^\tau$ which is identically equal to one where 
$\mathrm{dist}_{X^\tau}\big(\hat{\mu}^\tau_{g^{-1}}(x),y   \big)
<\delta$, and vanishes identically where $\mathrm{dist}_{X^\tau}\big(\hat{\mu}^\tau_{g^{-1}}(x),y   \big)
>2\,\delta$. We can rewrite the left hand side of
(\ref{eqn:Piklambdatilde})  as
$$
\Pi^\tau_{k\,\boldsymbol{\lambda}}(x,y)=
\Pi^\tau_{k\,\boldsymbol{\lambda}}(x,y)'+
\Pi^\tau_{k\,\boldsymbol{\lambda}}(x,y)'',
$$
where $\Pi^\tau_{k\,\boldsymbol{\lambda}}(x,y)'$
(respectively,  $\Pi^\tau_{k\,\boldsymbol{\lambda}}(x,y)''$) is as in 
(\ref{eqn:Piklambdatilde}), but with the integrand multiplied by 
$\rho_1(g,x,y)$ (respectively, by $1-\rho_1(g,x,y)$).
Arguing as in \S \ref{sctn:rapid decay compact}, 
one can check that $\Pi^\tau_{k\,\boldsymbol{\lambda}}(x,y)''
=O\left(k^{-\infty}\right)$, so that 
$$
\Pi^\tau_{k\,\boldsymbol{\lambda}}(x,y)\sim
\Pi^\tau_{k\,\boldsymbol{\lambda}}(x,y)'.
$$
On the domain of integration of the integrand of $\Pi^\tau_{k\,\boldsymbol{\lambda}}(x,y)'$, we can represent $\Pi^\tau$ as a Fourier
integral operator with complex phase, as discussed in \S \ref{scnt:szego parametrix}.

In view of the Weyl integration formula in \S \ref{sctn:weyl integration for}
\begin{eqnarray}
\label{eqn:first reduction Pik}
\Pi^\tau_{k\,\boldsymbol{\lambda}}(x,y)&\sim&
\frac{d_{k\,\boldsymbol{\lambda}}}{|W|}\,
\int_{\hat{T}}\,\mathrm{d}^HV_{\hat{T}}(t)\,\int_{\hat{G}/\hat{T}}
\,\mathrm{d}^HV_{\hat{G}/\hat{T}}(g\,\hat{T})\\
&&\left[\overline{\hat{\chi}_{k\,\boldsymbol{\lambda}}(t)}
\,\rho_1\left(g\,t\,g^{-1},x,y\right)\,
\Pi^\tau
\left(\hat{\mu}^\tau_{g\,t^{-1}\,g^{-1}}(x),y   \right)
|\Delta(t)|^2\right].\nonumber
\end{eqnarray}
Combining (\ref{eqn:first reduction Pik}) with the Weyl character formula
in \S \ref{sctn:weyl char for} we obtain
\begin{eqnarray}
\label{eqn:second reduction Pik}
\Pi^\tau_{k\,\boldsymbol{\lambda}}(x,y)&\sim&
\frac{d_{k\,\boldsymbol{\lambda}}}{|W|}\,\sum_{w\in W}
\int_{\hat{T}}\,\mathrm{d}^HV_{\tilde{T}}(t)\,\int_{\hat{G}/\hat{T}}
\,\mathrm{d}^HV_{\hat{G}/\hat{T}}(g\,\hat{T})\\
&&\left[(-1)^w\,\overline{E_{w(\boldsymbol{\nu}_{k\,\boldsymbol{\lambda}})}
(t)}
\,\rho_1\left(g\,t\,g^{-1},x,y\right)\,
\Pi^\tau
\left(\tilde{\mu}^\tau_{g\,t^{-1}\,g^{-1}}(x),y   \right)
\,\Delta(t)\right]\nonumber\\
&=&
\frac{d_{k\,\boldsymbol{\lambda}}}{|W|}\,\sum_{w\in W}
\int_{\hat{T}}\,\mathrm{d}^HV_{\hat{T}}(t)\,\int_{\hat{G}/\hat{T}}
\,\mathrm{d}^HV_{\hat{G}/\hat{T}}(g\,\hat{T})\nonumber\\
&&\left[(-1)^w\,\overline{E_{\boldsymbol{\nu}_{k\,\boldsymbol{\lambda}}}
\big(w(t)\big)}
\,\rho_1\left(g\,t\,g^{-1},x,y\right)\,
\Pi^\tau
\left(\hat{\mu}^\tau_{g\,t^{-1}\,g^{-1}}(x),y   \right)
\,\Delta(t)\right]\nonumber
\end{eqnarray}
where $\boldsymbol{\nu}_{k\boldsymbol{\lambda}}:=
k\,\boldsymbol{\lambda}+\boldsymbol{\delta}$.

For any $w\in W$, there is $g_w\in N(\hat{T})$ (the normalizer of
$\hat{T}$ in $\hat{G}$) such that
$w(t):=g_w\,t\,g_w^{-1}$. Furthermore,
$\Delta(g_w\,t\,g_w^{-1})=(-1)^w\,\Delta(t)$ by Remark \ref{rem:alternanza}.
We may then rewrite (\ref{eqn:second reduction Pik}) as follows:
\begin{eqnarray}
\label{eqn:third reduction Pik}
\Pi^\tau_{k\,\boldsymbol{\lambda}}(x,y)&\sim&
\frac{d_{k\,\boldsymbol{\lambda}}}{|W|}\,\sum_{w\in W}
\int_{\hat{T}}\,\mathrm{d}^HV_{\hat{T}}(t)\,\int_{\hat{G}/\hat{T}}
\,\mathrm{d}^HV_{\hat{G}/\hat{T}}(g\,\hat{T})\\
&&\left[\overline{E_{\boldsymbol{\nu}_{k\,\boldsymbol{\lambda}}}
\big(g_w\,t\,g_w^{-1}\big)}
\,\rho_1\left(g\,t\,g^{-1},x,y\right)\,
\Pi^\tau
\left(\hat{\mu}^\tau_{g\,t^{-1}\,g^{-1}}(x),y   \right)
\,\Delta(g_w\,t\,g_w^{-1})\right].\nonumber
\end{eqnarray}
Observe that $w\in W$ also induces a measure preserving diffeomorphism 
$$\alpha_w:g\,\hat{T}\in \hat{G}/\hat{T}\mapsto
g\,g_w\,\hat{T}\in\hat{G}/\hat{T}.$$
Hence (\ref{eqn:third reduction Pik}) may be rewritten
\begin{eqnarray}
\label{eqn:4th reduction Pik}
\lefteqn{\Pi^\tau_{k\,\boldsymbol{\lambda}}(x,y)}\\
&\sim&\nonumber
\frac{d_{k\,\boldsymbol{\lambda}}}{|W|}\,\sum_{w\in W}
\int_{\hat{T}}\,\mathrm{d}^HV_{\hat{T}}(t)\,\int_{\hat{G}/\hat{T}}
\,\mathrm{d}^HV_{\hat{G}/\hat{T}}(g\,\hat{T})\\
&&\left[\overline{E_{\boldsymbol{\nu}_{k\,\boldsymbol{\lambda}}}
\big(g_w\,t\,g_w^{-1}\big)}
\,\rho_1\left(g\,g_w\,t\,g_w^{-1}\,g^{-1},x,y\right)\,
\Pi^\tau
\left(\hat{\mu}^\tau_{g\,g_w\,t^{-1}\,g_w^{-1}\,g^{-1}}(x),y   \right)
\,\Delta(g_w\,t\,g_w^{-1})\right]\nonumber\\
&=&\nonumber
d_{k\,\boldsymbol{\lambda}}
\int_{\hat{T}}\,\mathrm{d}^HV_{\hat{T}}(t)\,\int_{\hat{G}/\hat{T}}
\,\mathrm{d}^H V_{\hat{G}/\hat{T}}(g\,\hat{T})\\
&&\left[\overline{E_{\boldsymbol{\nu}_{k\,\boldsymbol{\lambda}}}
\big(t\big)}
\,\rho_1\left(g\,t\,g^{-1},x,y\right)\,
\Pi^\tau
\left(\tilde{\mu}^\tau_{g\,t^{-1}\,g^{-1}}(x),y   \right)
\,\Delta(t)\right];\nonumber
\end{eqnarray}
to obtain the last equality, in the $w$-th summand we have performed
the measure preserving change of variable 
$t'=g_w\,t\,g_w^{-1}$ in $\hat{T}$.

Let us insert in (\ref{eqn:4th reduction Pik}) the description of $\Pi^\tau$ as a Fourier integral operator recalled in \S \ref{scnt:szego parametrix}.
Furthermore, let us fix an isomorphism 
$\hat{T}\cong (S^1)^{r_G}$ and trigonometric coordinates
$\boldsymbol{\vartheta}=(\vartheta_1,\ldots,\vartheta_{r_G})$
with $\vartheta_j\in (-\pi,\pi)$. We may view
$\mathcal{D}^{\tilde{G}}\subset \mathbb{Z}^{r_G}$.
Thus we have the replacements
$$
t=e^{\imath\,\boldsymbol{\vartheta}},\quad
\mathrm{d}^HV_{\tilde{T}}(t)=\frac{1}{(2\,\pi)^{r_G}}\,
\mathrm{d}\boldsymbol{\vartheta},\quad
E_{\boldsymbol{\nu}_{k\,\boldsymbol{\lambda}}}
\big(t\big)=e^{\imath\,\langle k\,\boldsymbol{\lambda}
+\boldsymbol{\delta},\boldsymbol{\vartheta}\rangle}.
$$
We obtain
\begin{eqnarray}
\label{eqn:5th reduction Pik}
\lefteqn{\Pi^\tau_{k\,\boldsymbol{\lambda}}(x,y)}\\
&\sim&
d_{k\,\boldsymbol{\lambda}}
\int_{\hat{T}}\,\mathrm{d}^HV_{\hat{T}}(t)\,\int_{\hat{G}/\hat{T}}
\,\mathrm{d}^H V_{\hat{G}/\hat{T}}(g\,\hat{T})\,
\int_0^{+\infty}\,\mathrm{d}u\nonumber\\
&&\left[\overline{E_{\boldsymbol{\nu}_{k\,\boldsymbol{\lambda}}}
\big(t\big)}
\,\rho_1\left(g\,t\,g^{-1},x,y\right)\,
e^{\imath\,u\,\psi^\tau\left( \hat{\mu}^\tau_{g\,t^{-1}\,g^{-1}}(x),y   \right) }
s\left(\hat{\mu}^\tau_{g\,t^{-1}\,g^{-1}}(x),y,u   \right)
\,\Delta(t)\right]\nonumber\\
&=&\frac{ d_{k\,\boldsymbol{\lambda}}}{(2\,\pi)^{r_G}}\,
\int_{(-\pi,\pi)^{r_G}}\,\mathrm{d}\boldsymbol{\vartheta}\,\int_{\hat{G}/\hat{T}}
\,\mathrm{d}^H V_{\hat{G}/\hat{T}}(g\,\hat{T})\,
\int_0^{+\infty}\,\mathrm{d}u\nonumber\\
&&\left[e^{-\imath\,\langle k\,\boldsymbol{\lambda}
+\boldsymbol{\delta},\boldsymbol{\vartheta}\rangle}
\,\rho_1\left(g\,e^{\imath\,\boldsymbol{\vartheta}}\,g^{-1},x,y\right)\,
e^{\imath\,u\,\psi^\tau \left(\hat{\mu}^\tau_{g\,e^{-\imath\,\boldsymbol{\vartheta}}\,g^{-1}}(x),y    \right)}
s\left(\hat{\mu}^\tau_{g\,e^{-\imath\,\boldsymbol{\vartheta}}\,g^{-1}}(x),y,u   \right)
\,\Delta(t)\right]\nonumber
\end{eqnarray} 
With the rescaling $u\mapsto k\,u$, we may rewrite
(\ref{eqn:5th reduction Pik}) as an oscillatory integral
\begin{eqnarray}
\label{eqn:6th reduction Pik}
\lefteqn{\Pi^\tau_{k\,\boldsymbol{\lambda}}(x,y)}\\
&\sim&\frac{ k\,d_{k\,\boldsymbol{\lambda}}}{(2\,\pi)^{r_G}}\,
\int_{(-\pi,\pi)^{r_G}}\,\mathrm{d}\boldsymbol{\vartheta}\,\int_{\hat{G}/\hat{T}}
\,\mathrm{d}^H V_{\hat{G}/\hat{T}}(g\,\hat{T})\,
\int_0^{+\infty}\,\mathrm{d}u\nonumber\\
&&\left[e^{-\imath\,\langle k\,\boldsymbol{\lambda}
+\boldsymbol{\delta},\boldsymbol{\vartheta}\rangle}
\,\rho_1\left(g\,e^{\imath\,\boldsymbol{\vartheta}}\,g^{-1},x,y\right)\,
e^{\imath\,k\,u\,\psi^\tau\left( \hat{\mu}^\tau_{g\,e^{-\imath\,\boldsymbol{\vartheta}}\,g^{-1}}(x),y    \right)}
s\left(\hat{\mu}^\tau_{g\,e^{-\imath\,\boldsymbol{\vartheta}}\,g^{-1}}(x),y ,k\,u  \right)
\,\Delta(e^{\imath\,\boldsymbol{\vartheta}})\right]\nonumber\\
&=&\frac{ k\,d_{k\,\boldsymbol{\lambda}}}{(2\,\pi)^{r_G}}\,
\int_{(-\pi,\pi)^{r_G}}\,\mathrm{d}\boldsymbol{\vartheta}\,\int_{\hat{G}/\hat{T}}
\,\mathrm{d}^H V_{\hat{G}/\hat{T}}(g\,\hat{T})\,
\int_0^{+\infty}\,\mathrm{d}u\,
\left[e^{\imath\,k\,\Psi_{x,y}(u,g,\boldsymbol{\vartheta})}\,
\tilde{\mathcal{A}}_{x,y,k}(u,g,\boldsymbol{\vartheta})\right],\nonumber
\end{eqnarray} 
where
\begin{eqnarray}
\label{eqn:PsiAxydefn}
\Psi_{x,y}(u,g\,\hat{T},\boldsymbol{\vartheta})&:=&
u\,\psi^\tau\left( \hat{\mu}^\tau_{g\,e^{-\imath\,\boldsymbol{\vartheta}}\,g^{-1}}(x),y    \right)-
\langle \boldsymbol{\lambda},\boldsymbol{\vartheta}\rangle,\\
\tilde{\mathcal{A}}_{x,y,k}(u,g\,\hat{T},\boldsymbol{\vartheta})&:=&
e^{-\imath\,\langle \boldsymbol{\delta},\boldsymbol{\vartheta}\rangle}
\,\rho_1\left(g\,e^{\imath\,\boldsymbol{\vartheta}}\,g^{-1},x,y\right)\,
s\left(\hat{\mu}^\tau_{g\,e^{-\imath\,\boldsymbol{\vartheta}}\,g^{-1}}(x),y ,k\,u  \right)
\,\Delta(e^{\imath\,\boldsymbol{\vartheta}}).\nonumber
\end{eqnarray}

Let us consider a fixed but general pair 
$(g_0\,\hat{T},t_0)\in (\hat{G}/\hat{T})\times \hat{T}$, where
$t_0=e^{\imath\,\boldsymbol{\vartheta}_0}$, in the
domain of integration (that is, such that 
$\left(g_0\,e^{\imath\,\boldsymbol{\vartheta}_0}\,g_0^{-1},x,y\right)\in 
\mathrm{supp}(\rho')$). 

Near $t_0$, we can write
$t=e^{\imath\,(\boldsymbol{\vartheta}+\boldsymbol{\vartheta}_0)}$
with $\boldsymbol{\vartheta}\sim \mathbf{0}$.
Let us set 
$x_0:=\tilde{\mu}^\tau_{g_0\,e^{-\imath\,\boldsymbol{\vartheta}_0}\,
g_0^{-1}}(x)$. Then $x_0\sim y$ (since we are on the support of
$\rho'$), whence
(recalling (\ref{eqn:Phitauy}))
\begin{equation}
\label{eqn:Phitaux0}
\Phi^\tau(x_0)\sim \Phi^\tau(y)\sim a\,\boldsymbol{\lambda},
\end{equation}
where $\sim$ provisionally stands for \lq is very close to\rq\,
(say at distance $O(\delta)$).
Thus for any $\boldsymbol{\vartheta}$
$$
\hat{\mu}^\tau_{
g_0\,
e^{-\imath\,(\boldsymbol{\vartheta}+\boldsymbol{\vartheta}_0)}\,
g_0^{-1}
}(x)=\hat{\mu}^\tau_{
g_0\,
e^{-\imath\,\boldsymbol{\vartheta}}\,
g_0^{-1}
}(x_0).
$$
Therefore, by the results in \S 2.5 of [GP24], in
NHLC's centered at $x_0$ (\S \ref{scnt:heis coord}) we have
\begin{eqnarray}
\label{eqn:local action}
\tilde{\mu}^\tau_{
g_0\,
e^{-\imath\,(\boldsymbol{\vartheta}+\boldsymbol{\vartheta}_0)}\,
g_0^{-1}
}(x)&=&
x_0+\Big(\big\langle \Phi(x_0),\mathrm{Ad}_{g_0}(\boldsymbol{\vartheta})
\big\rangle+R_3(\boldsymbol{\vartheta}),\mathbf{R}_1(\boldsymbol{\vartheta})\\
&=&x_0+\Big(\big\langle \mathrm{Ad}_{g_0^{-1}}\big(\Phi(x_0)\big),\boldsymbol{\vartheta}
\big\rangle+R_3(\boldsymbol{\vartheta}),\mathbf{R}_1(\boldsymbol{\vartheta})
\Big),\nonumber
\end{eqnarray}
where $R_j$ (respectively, $\mathbf{R}_j$) generically denotes 
a real-valued (respectively, vector valued) 
function defined on a neighbourhood of the origin of
some Euclidean space and vanishing to
$j$-th order at the origin.

In view of 
(\ref{eqn:diagonal diff psi}),
(\ref{eqn:PsiAxydefn}) and (\ref{eqn:Phitaux0}) above, and Proposition 34 of [Pao2024], 
(\ref{eqn:local action}) implies
\begin{equation}
\label{eqn:diff of Psixy}
\partial_{\boldsymbol{\vartheta}}\Psi_{x,y}(u,g\,\tilde{T},\boldsymbol{\vartheta})\sim \left.\left[u\,
\mathrm{Ad}_{g_0^{-1}}\big(\Phi(x_0)\big)-\boldsymbol{\lambda}
\right]\right|_\mathfrak{t}
\sim \left.\left[u\,a\,
\mathrm{Ad}_{g_0^{-1}}\big(\boldsymbol{\lambda}\big)-\boldsymbol{\lambda}\right]\right|_\mathfrak{t}.
\end{equation}
On the other hand, by the assumption on $\mathcal{O}$ 
there exist a constant $c>0$ such that
$
 \left\|\left.\mathrm{Ad}_{g_0^{-1}}\big(\boldsymbol{\lambda}\big)\right|_\mathfrak{t}
\right\|\ge c
$
for all $g_0\in G$. 
Therefore, (\ref{eqn:diff of Psixy}) implies the following.

\begin{lem}
\label{lem:partial theta bdd bl}
Under the previous assumptions, there exist constants $D\gg 0,\,c>0$
such that 
$\|\partial_{\boldsymbol{\vartheta}}\Psi_{x,y}(u,g\,\tilde{T},\boldsymbol{\vartheta})\|\ge c$ whenever $u\not\in (1/D,D)$.
\end{lem} 

We can then \lq integrate by parts\rq\, in $\boldsymbol{\vartheta}$ to reduce
to the case where the integrand in (\ref{eqn:6th reduction Pik})
is compactly supported in $u$. More precisely, by a standard argument
Lemma \ref{lem:partial theta bdd bl} implies the following.

\begin{cor}
\label{cor:compact supp u}
The asymptotics of (\ref{eqn:5th reduction Pik}) are unchanged, if the
integrand is multiplied by a compactly supported 
cut-off function 
$\rho_2\in \mathcal{C}^\infty_c\big( 1/(2\,D),2\,D) \big)$
such that $\rho_2\equiv 1$ on $(1/D,D)$.

\end{cor}
Thus the asymptotics remain unchanged if the
amplitude $\tilde{\mathcal{A}}_{x,y,k}(u,g\,\tilde{T},\boldsymbol{\vartheta})$ in (\ref{eqn:PsiAxydefn}) 
is replaced by
\begin{eqnarray}
\label{eqn:Axyredefn}
\mathcal{A}_{x,y,k}(u,g\,\tilde{T},\boldsymbol{\vartheta}):=
\tilde{\mathcal{A}}_{x,y,k}(u,g\,\tilde{T},\boldsymbol{\vartheta})
\cdot \rho_2(u).\nonumber
\end{eqnarray}

Hence,
\begin{eqnarray}
\label{eqn:7th reduction Pik}
\lefteqn{\Pi^\tau_{k\,\boldsymbol{\lambda}}(x,y)}\\
&\sim&\frac{ k\,d_{k\,\boldsymbol{\lambda}}}{(2\,\pi)^{r_G}}\,
\int_{(-\pi,\pi)^{r_G}}\,\mathrm{d}\boldsymbol{\vartheta}\,\int_{\hat{G}/\hat{T}}
\,\mathrm{d}^H V_{\hat{G}/\hat{T}}(g\,\hat{T})\,
\int_0^{+\infty}\,\mathrm{d}u\,
\left[e^{\imath\,k\,\Psi_{x,y}(u,g,\boldsymbol{\vartheta})}\,
\mathcal{A}_{x,y,k}(u,g,\boldsymbol{\vartheta})\right],\nonumber
\end{eqnarray} 
where the integrand is now compactly supported in $u$.
In particular, it is legitimate to integrate by parts in $u$.

Under hypothesis (\ref{eqn:bd distance orbits}) of the Theorem,
in view of (\ref{eqn:bd dist psitau xy}) for some 
constant $D^\tau>0$ we have
\begin{equation}
\label{eqn:der u Psixy}
\left|\partial_u \Psi_{x,y}(u,g\,\hat{T},\boldsymbol{\vartheta})\right|
=
\left|\psi^\tau (x,y)\right|\ge \Im\left(\psi^\tau(x,y)\right)
\ge D^\tau\,k^{2\,\epsilon-1}.
\end{equation}
Writing 
$$
e^{\imath\,k\,\Psi_{x,y}(u,g\,\hat{T},\boldsymbol{\vartheta})}
=-\frac{\imath}{k\,\psi^\tau (x,y)}\,\partial_u\left(
e^{\imath\,k\,\Psi_{x,y}(u,g\,\hat{T},\boldsymbol{\vartheta})}\right),
$$
we can then iteratively integrate by parts in $u$, introducing
at each step a factor which is bounded by $
D_1^\tau \,k^{-2\,\epsilon}$ for some constant $D_1^\tau>0$.
We conclude that
$\Pi^\tau_{k\,\boldsymbol{\lambda}}(x,y)=O\left(k^{-\infty}\right)$
in the given range.
\end{proof}

\subsection{Proof of Theorem \ref{thm:rapid decay moment map}}

The proof of Theorem \ref{thm:rapid decay moment map}
relies on the Kirillov character formula (see [Kir]), which we briefly recall.

\subsubsection{The Kirillov character formula}
\label{sctn:kircharform}

Let us fix an $\mathrm{Ad}$-invariant Euclidean product
$\kappa^H$ on
$\mathfrak{g}$ whose associated Riemannian density is
the Haar measure $\mathrm{d}^HV_G$, and
let $\mathrm{d}^HV_{\mathfrak{g}}$ be the corresponding
Lebesgue measure on $\mathfrak{g}$.

Let $\exp_G:\mathfrak{g}\rightarrow G$ denote the exponential
map, and let $\mathfrak{g}'\subseteq \mathfrak{g}$
be an open neighbourhood of the origin such that
\begin{equation}
\label{eqn:restriction exp G'}
\exp_G':=
\left.\exp_G\right|_{\mathfrak{g}'}:\mathfrak{g}'\rightarrow G':=\exp_G(\mathfrak{g}')
\end{equation}
is a diffeomorphism. 
Let $\mathcal{P}:\mathfrak{g}'\rightarrow (0,+\infty)$ 
be the smooth function 
defined by
\begin{equation}
\label{eqn:defn P^2}
\exp_G^*\left(\mathrm{d}^HV_G\right)=
\mathcal{P}^2\,\mathrm{d}^HV_{\mathfrak{g}};
\end{equation}
thus $\mathcal{P}(0)=1$. We shall usually write
$\exp_G(\boldsymbol{\xi})=e^{\boldsymbol{\xi}}$
($\boldsymbol{\xi}\in \mathfrak{g}$).

Furthermore, let 
$\mathcal{O}_{\boldsymbol{\nu}_{\boldsymbol{\lambda}}}
\subseteq \mathfrak{g}^\vee$
be the coadjoint orbit through $\boldsymbol{\nu}_{\boldsymbol{\lambda}}:=\boldsymbol{\lambda}+
\boldsymbol{\delta}$, and let $\sigma_{\boldsymbol{\nu}_{\boldsymbol{\lambda}}}$ denote its Konstant-Kirillov 
symplectic structure. The symplectic volume form on
$\mathcal{O}_{\boldsymbol{\nu}_{\boldsymbol{\lambda}}}$
is then 
$$
\mathrm{d}
V_{\mathcal{O}_{\boldsymbol{\nu}_{\boldsymbol{\lambda}}}}
:=
\frac{\sigma_{\boldsymbol{\nu}_{\boldsymbol{\lambda}}}
^{\wedge n_G}}{n_G!},
\quad \text{where}\quad
n_G:=\frac{1}{2}\,\dim(\mathcal{O}_{\boldsymbol{\nu}_{\boldsymbol{\lambda}}})=\frac{1}{2}\,
(d-r_G),
$$
where the latter equality holds because 
$\boldsymbol{\nu}_{\boldsymbol{\lambda}}$ is a regular weight.

The Kirillov character formula expresses the restriction of 
$\chi_{\boldsymbol{\lambda}}$ to $G'$ in 
terms of an integral on
the symplectic manifold $(\mathcal{O}_{\boldsymbol{\nu}_{\boldsymbol{\lambda}}},\sigma_{\boldsymbol{\nu}_{\boldsymbol{\lambda}}})$. Namely, if 
$\boldsymbol{\xi}\in \mathfrak{g}'$ then
\begin{equation}
\label{eqn:kir char form}
\chi_{\boldsymbol{\lambda}}
\left(e^{\boldsymbol{\xi}}\right)
=\frac{1}{(2\,\pi)^{n_G}}\,
\frac{1}{\mathcal{P}(\boldsymbol{\xi})}\,
\int_{\mathcal{O}_{\boldsymbol{\nu}_{\boldsymbol{\lambda}}}}
e^{\imath\,\boldsymbol{\beta}(\boldsymbol{\xi})}\,
\mathrm{d}
V_{\mathcal{O}_{\boldsymbol{\nu}_{\boldsymbol{\lambda}}}}
(\boldsymbol{\beta}).
\end{equation}

\subsubsection{The proof}

We shall first prove the following 
apparently weaker statement.

\begin{thm}
\label{thm:rapid decay moment map diag}
Under the hypothesis of Theorem \ref{thm:rapid decay moment map},
 uniformly for
\begin{equation}
\label{eqn:bound distance momentO1}
x\in K\,:\,\mathrm{dist}_{X^\tau}
\left(x,X^\tau_{\mathcal{O}}\right)
%
\ge \,C\,k^{\epsilon-\frac{1}{2}}
\end{equation}
one has
$$
\Pi^\tau_{k\,\boldsymbol{\lambda}}(x,x)=
O\left(k^{-\infty}\right)\qquad
\text{and}\qquad P^\tau_{k\,\boldsymbol{\lambda}}(x,x)=
O\left(k^{-\infty}\right)
$$
for $k\rightarrow +\infty$.

\end{thm}

\begin{proof}
Let us start with (\ref{eqn:Piklambda}) with 
$x=y$.
We may assume without loss that $x$ belongs to a small tubular
neighbourhood of $X^\tau_{\mathcal{O}}$. 
Perhaps replacing
$x$ by $\mu^\tau_h(x)$ for some $h\in G$, we may further 
assume that $\Phi^\tau(x)$ belongs to a small conic neighourhood of $\mathbb{R}_+\cdot \boldsymbol{\lambda}$,
hence that (\ref{eqn:Phitauy}) holds with $y=x$.

Furthermore, fix $r>0$ sufficiently small
(how small may depend on $K$) and 
let $\rho_2\in \mathcal{C}^\infty(G)$
be supported in the ball
$G_r\subseteq G$
centered at $e$ of radius
$2\,r$ (say, in the Riemannian metric
$\kappa$) and identically equal to $1$ on the ball of
radius $r$. Since $\tilde{\mu}^\tau$ is free on $K$,
the asymptotics of 
$\Pi^\tau_{k\,\boldsymbol{\lambda}}(x,x)$ are unchanged, 
if the integrand in (\ref{eqn:Piklambda}) is multiplied
by $\rho_2(g)$. Thus
\begin{equation}
\label{eqn:Piklambda G}
\Pi^\tau_{k\,\boldsymbol{\lambda}}(x,x)\sim
d_{k\boldsymbol{\lambda}}\,
\int_{G_r} \overline{\chi_{k\,\boldsymbol{\lambda}}(g)}\,
\rho_2(g)\,
\Pi^\tau
\left(\mu^\tau_{g^{-1}}(x),x   \right)\,
\mathrm{d}^HV_G(g).
\end{equation}
We may assume without loss that $G_r\subseteq G'$; hence
integration over $G_r$ may be transferred to 
$\mathfrak{g}_r\subseteq \mathfrak{g}$, the open ball
centered at the origin for $\kappa_e$, by the diffeomorphism
(\ref{eqn:restriction exp G'}). In view of  
(\ref{eqn:defn P^2}), (\ref{eqn:Piklambda G}) may be rewritten
\begin{equation}
\label{eqn:Piklambda g}
\Pi^\tau_{k\,\boldsymbol{\lambda}}(x,x)\sim
d_{k\boldsymbol{\lambda}}\,
\int_{\mathfrak{g}_r} \overline{\chi_{k\,\boldsymbol{\lambda}}(e^{\boldsymbol{\xi}})}\,
\rho_2(e^{\boldsymbol{\xi}})\,
\Pi^\tau
\left(\mu^\tau_{e^{-\boldsymbol{\xi}}}(x),x   \right)\,
\mathcal{P}(\boldsymbol{\xi})^2\,
\mathrm{d}^HV_{\mathfrak{g}}(\boldsymbol{\xi}).
\end{equation}

Next let us make use of (\ref{eqn:kir char form}), with
$k\,\boldsymbol{\lambda}$ in place of $\boldsymbol{\lambda}$.
Namely, for $\boldsymbol{\xi}\in \mathfrak{g}'$ we have
\begin{equation}
\label{eqn:kir char form k}
\chi_{k\,\boldsymbol{\lambda}}
\left(e^{\boldsymbol{\xi}}\right)
=\frac{1}{(2\,\pi)^{n_G}}\,
\frac{1}{\mathcal{P}(\boldsymbol{\xi})}\,
\int_{\mathcal{O}_{\boldsymbol{\nu}_{k\,\boldsymbol{\lambda}}}}
e^{\imath\,\boldsymbol{\beta}(\boldsymbol{\xi})}\,
\mathrm{d}
V_{\mathcal{O}_{\boldsymbol{\nu}_{k\,\boldsymbol{\lambda}}}}
(\boldsymbol{\beta}).
\end{equation}

We have
$$
\mathcal{O}_{\boldsymbol{\nu}_{k\,\boldsymbol{\lambda}}}=
\mathcal{O}_{k\,\boldsymbol{\lambda}+
\boldsymbol{\delta}}
=k\,\mathcal{O}_{\boldsymbol{\lambda}+\frac{1}{k}\,
\boldsymbol{\delta}}\subseteq \mathfrak{g}^\vee.
$$

We can therefore transfer integration on 
$\mathcal{O}_{\boldsymbol{\nu}_{k\,\boldsymbol{\lambda}}}$
to $\mathcal{O}_{\boldsymbol{\lambda}+\frac{1}{k}\,
\boldsymbol{\delta}}$ by the dilation 
$\mathfrak{d} _k:
\boldsymbol{\beta}\in \mathcal{O}_{\boldsymbol{\lambda}+\frac{1}{k}\,
\boldsymbol{\delta}}\mapsto k\,\boldsymbol{\beta}
\in \mathcal{O}_{\boldsymbol{\nu}_{k\,\boldsymbol{\lambda}}}$.
Since 
$$\mathfrak{d} _k^*(\sigma_{\mathcal{O}_{\boldsymbol{\nu}_{k\,\boldsymbol{\lambda}}}})=
k\,\sigma_{\mathcal{O}_{\boldsymbol{\lambda}+\frac{1}{k}\,
\boldsymbol{\delta}}},
$$ 
we can rewrite (\ref{eqn:kir char form}) as 
\begin{equation}
\label{eqn:kir char form rescaled}
\chi_{k\,\boldsymbol{\lambda}}
\left(e^{\boldsymbol{\xi}}\right)
=\left(\frac{k}{2\,\pi}\right)^{n_G}\,
\frac{1}{\mathcal{P}(\boldsymbol{\xi})}\,
\int_{\mathcal{O}_{\boldsymbol{\lambda}+\frac{1}{k}\,
\boldsymbol{\delta}}}
e^{\imath\,k\,\boldsymbol{\beta}(\boldsymbol{\xi})}\,
\mathrm{d}
V_{\mathcal{O}_{\boldsymbol{\lambda}+\frac{1}{k}\,
\boldsymbol{\delta}}}
(\boldsymbol{\beta}).
\end{equation}
Finally, we transfer integration from the
variable orbit $\mathcal{O}_{\boldsymbol{\lambda}+\frac{1}{k}\,
\boldsymbol{\delta}}$ to the fixed orbit 
$\mathcal{O}=\mathcal{O}_{\boldsymbol{\lambda}}$.
Since both $\boldsymbol{\lambda}$ and $\boldsymbol{\delta}$
are regular weights, there is a well-defined equivariant diffeomorphism
$L_k:\mathcal{O}_{\boldsymbol{\lambda}}\rightarrow
\mathcal{O}_{\boldsymbol{\lambda}+\frac{1}{k}\,
\boldsymbol{\delta}}$, defined as follows.
If $\boldsymbol{\beta}=\mathrm{Coad}_g(\boldsymbol{\lambda})$
for some $g\in G$, let us set
\begin{eqnarray}
\label{eqn:delta beta}
\boldsymbol{\delta}_{\boldsymbol{\beta}}&:=&
\mathrm{Coad}_g(\boldsymbol{\delta})\in \mathcal{O}_{\boldsymbol{\delta}}\\
L_k(\boldsymbol{\beta})&:=&\mathrm{Coad}_g\left(\boldsymbol{\lambda}+\frac{1}{k}\,
\boldsymbol{\delta}\right)=\boldsymbol{\beta}+
\frac{1}{k}\,\boldsymbol{\delta}_{\boldsymbol{\beta}}\in \mathcal{O}_{\boldsymbol{\lambda}+\frac{1}{k}\,
\boldsymbol{\delta}}.
\nonumber
\end{eqnarray}
Then
\begin{equation}
\label{eqn:Vk beta exp}
L_k^*(\mathrm{d}
V_{\mathcal{O}_{\boldsymbol{\lambda}+\frac{1}{k}\,
\boldsymbol{\delta}}})=
\mathcal{V}_k\,
\mathrm{d}
V_{\mathcal{O}_{\boldsymbol{\lambda}}},\quad
\text{where}\quad
\mathcal{V}_k(\boldsymbol{\beta})\sim 1+
\sum_{j\ge 1}\,k^{-j}\,V_j(\boldsymbol{\beta}).
\end{equation}
Thus (\ref{eqn:kir char form rescaled}) can be rewritten
\begin{eqnarray}
\label{eqn:kir char form rescaled transl}
\chi_{k\,\boldsymbol{\lambda}}
\left(e^{\boldsymbol{\xi}}\right)
&=&\left(\frac{k}{2\,\pi}\right)^{n_G}\,
\frac{1}{\mathcal{P}(\boldsymbol{\xi})}\,
\int_{\mathcal{O}_{\boldsymbol{\lambda}}}
e^{\imath\,k\,L_k(\boldsymbol{\beta})(\boldsymbol{\xi})
}\,\mathcal{V}_k(\boldsymbol{\beta})\,
\mathrm{d}
V_{\mathcal{O}_{\boldsymbol{\lambda}}}
(\boldsymbol{\beta})\\
&=&\left(\frac{k}{2\,\pi}\right)^{n_G}\,
\frac{1}{\mathcal{P}(\boldsymbol{\xi})}\,
\int_{\mathcal{O}_{\boldsymbol{\lambda}}}
e^{\imath\,[k\,\boldsymbol{\beta}(\boldsymbol{\xi})+
\boldsymbol{\delta}_{\boldsymbol{\beta}}(\boldsymbol{\xi})]
}\,\mathcal{V}_k(\boldsymbol{\beta})\,
\mathrm{d}
V_{\mathcal{O}_{\boldsymbol{\lambda}}}
(\boldsymbol{\beta}).\nonumber
\end{eqnarray}

Let us insert (\ref{eqn:kir char form rescaled transl})
in (\ref{eqn:Piklambda g}), and then make use of the description of 
$\Pi^\tau$ as an FIO (\S \ref{scnt:szego parametrix}).
We obtain
\begin{eqnarray}
\label{eqn:Piklambda g expanded}
\Pi^\tau_{k\,\boldsymbol{\lambda}}(x,x)&\sim&
d_{k\boldsymbol{\lambda}}\,\left(\frac{k}{2\,\pi}\right)^{n_G}\,
\int_{\mathfrak{g}}\,
\mathrm{d}^HV_{\mathfrak{g}}(\boldsymbol{\xi})\,\int_{\mathcal{O}_{\boldsymbol{\lambda}}}\,
\mathrm{d}
V_{\mathcal{O}_{\boldsymbol{\lambda}}}(\boldsymbol{\beta})\\
&&\left[
 e^{-\imath\,[k\,\boldsymbol{\beta}(\boldsymbol{\xi})+
\boldsymbol{\delta}_{\boldsymbol{\beta}}(\boldsymbol{\xi})]
}\,
\rho_2(e^{\boldsymbol{\xi}})\,
\Pi^\tau
\left(\mu^\tau_{e^{-\boldsymbol{\xi}}}(x),y   \right)\,\mathcal{V}_k(\boldsymbol{\beta})\,
\mathcal{P}(\boldsymbol{\xi})\right]\nonumber\\
&=&d_{k\boldsymbol{\lambda}}\,\left(\frac{k}{2\,\pi}\right)^{n_G}\,
\int_{\mathfrak{g}}\,
\mathrm{d}^HV_{\mathfrak{g}}(\boldsymbol{\xi})\,\int_{\mathcal{O}_{\boldsymbol{\lambda}}}\,
\mathrm{d}
V_{\mathcal{O}_{\boldsymbol{\lambda}}}(\boldsymbol{\beta})
\,\int_0^{+\infty}\,\mathrm{d}u
\nonumber\\
&&\left[
 e^{\imath\,\left[u\,\psi^\tau\left(\mu^\tau_{e^{-\boldsymbol{\xi}}}(x),x   \right)-k\,\boldsymbol{\beta}(\boldsymbol{\xi})-
\boldsymbol{\delta}_{\boldsymbol{\beta}}(\boldsymbol{\xi})\right]
}\,
\rho_2(e^{\boldsymbol{\xi}})\,\,
s^\tau\left(\mu^\tau_{e^{-\boldsymbol{\xi}}}(x),x  ,u\right)\,\mathcal{V}_k(\boldsymbol{\beta})\,
\mathcal{P}(\boldsymbol{\xi})\right].\nonumber
\end{eqnarray}
Performing the change of variable $u\mapsto k\,u$, we finally obtain
\begin{eqnarray}
\label{eqn:Piklambda g expanded resc}
\Pi^\tau_{k\,\boldsymbol{\lambda}}(x,x)&\sim&
k\,d_{k\boldsymbol{\lambda}}\,\left(\frac{k}{2\,\pi}\right)^{n_G}\,
\int_{\mathfrak{g}}\,
\mathrm{d}^HV_{\mathfrak{g}}(\boldsymbol{\xi})\,\int_{\mathcal{O}_{\boldsymbol{\lambda}}}\,
\mathrm{d}
V_{\mathcal{O}_{\boldsymbol{\lambda}}}(\boldsymbol{\beta})
\,\int_0^{+\infty}\,\mathrm{d}u
\\
&&\left[
 e^{\imath\,k\,\left[u\,\psi^\tau\left(\mu^\tau_{e^{-\boldsymbol{\xi}}}(x),x   \right)-\boldsymbol{\beta}(\boldsymbol{\xi})\right]
}\,e^{-
\boldsymbol{\delta}_{\boldsymbol{\beta}}(\boldsymbol{\xi})}
\rho_2(e^{\boldsymbol{\xi}})\,\,
s^\tau\left(\mu^\tau_{e^{-\boldsymbol{\xi}}}(x),x  ,k\,u\right)\,\mathcal{V}_k(\boldsymbol{\beta})\,
\mathcal{P}(\boldsymbol{\xi})\right]\nonumber\\
&=&
k\,d_{k\boldsymbol{\lambda}}\,\left(\frac{k}{2\,\pi}\right)^{n_G}\,
\int_{\mathfrak{g}}\,
\mathrm{d}^HV_{\mathfrak{g}}(\boldsymbol{\xi})\,\int_{\mathcal{O}_{\boldsymbol{\lambda}}}\,
\mathrm{d}
V_{\mathcal{O}_{\boldsymbol{\lambda}}}(\boldsymbol{\beta})
\,\int_0^{+\infty}\,\mathrm{d}u
\nonumber\\
&&\left[
 e^{\imath\,k\,\Gamma_x(u,\boldsymbol{\beta},\boldsymbol{\xi})
}\,\tilde{\mathcal{B}}_{x,k}(u,\boldsymbol{\beta},\boldsymbol{\xi})\right],\nonumber
\end{eqnarray}
where
\begin{eqnarray}
\label{eqn:defn of Gamma e B}
\Gamma_x(u,\boldsymbol{\beta},\boldsymbol{\xi})&:=&
u\,\psi^\tau\left(\mu^\tau_{e^{-\boldsymbol{\xi}}}(x),x   \right)-\boldsymbol{\beta}(\boldsymbol{\xi})
\\
\tilde{\mathcal{B}}_{x,k}(u,\boldsymbol{\beta},\boldsymbol{\xi})&:=&e^{-
\boldsymbol{\delta}_{\boldsymbol{\beta}}(\boldsymbol{\xi})}
\tilde{\rho}(\boldsymbol{\xi})\,
s^\tau\left(\mu^\tau_{e^{-\boldsymbol{\xi}}}(x),x  ,k\,u\right)\,\mathcal{V}_k(\boldsymbol{\beta})\,
\mathcal{P}(\boldsymbol{\xi}),\nonumber
\end{eqnarray}
where we have set 
$\tilde{\rho}(\boldsymbol{\xi}):=\rho_2(e^{\boldsymbol{\xi}})$,
and $\boldsymbol{\delta}_{\boldsymbol{\beta}}$ is as in (\ref{eqn:delta beta}).

The right hand side of (\ref{eqn:Piklambda g expanded resc}) is an oscillatory 
integral with phase $\Gamma_x$ and amplitude $\tilde{\mathcal{B}}_{x,k}$.
Our next goal is to prove that we may reduce to a compact domain
of integration without altering the asymptotics.

We have $x=g\,\exp^{\tilde{G}}(\imath\,\boldsymbol{\eta})$,
for some $\boldsymbol{\eta}\in \mathfrak{g}$ of norm $\tau$.
Then by the discussion of \S \ref{sctn:splitting alpha} 
$$
\boldsymbol{\xi}_{\tilde{G}}(x)=\boldsymbol{\xi}_{\tilde{G}}^\sharp(x)
-\tilde{\varphi}^{\boldsymbol{\xi}}(x)\,\mathcal{R}(x)
=\boldsymbol{\xi}_{\tilde{G}}^\sharp(x)
-\kappa_e\big(\mathrm{Ad}_g(\boldsymbol{\eta}),\boldsymbol{\xi}\big)
\,\mathcal{R}(x).
$$
Hence,
$$
\alpha_x\big( \boldsymbol{\xi}_{\tilde{G}}(x)  \big)=
-\kappa_e\big(\mathrm{Ad}_g(\boldsymbol{\eta}),\boldsymbol{\xi}\big)
=-\mathrm{Ad}_g(\boldsymbol{\eta})_\kappa\big(\boldsymbol{\xi}\big).
$$
On the support of $\tilde{\rho}$ we have 
$\mathrm{dist}_{X^\tau}\left(\mu^\tau_{e^{-\boldsymbol{\xi}}}(x),x   \right)
=O(r)$.
Therefore, in view of (\ref{eqn:diagonal diff psi}), we conclude that 
\begin{equation}
\label{eqn:derivative xi psi}
\partial_{\boldsymbol{\xi}}\psi^\tau\left(\mu^\tau_{e^{-\boldsymbol{\xi}}}(x),x   \right)=\mathrm{Ad}_g(\boldsymbol{\eta})_\kappa+O(r)\quad
\text{if}\quad 
x=g\,\exp^{\tilde{G}}(\imath\,\boldsymbol{\eta}).
\end{equation}
In view of (\ref{eqn:defn of Gamma e B}), 
on the domain of integration
of (\ref{eqn:Piklambda g expanded resc}) we have 
\begin{equation}
\label{eqn:der xi Gammax}
\partial_{\boldsymbol{\xi}}\Gamma_x(u,\boldsymbol{\beta},\boldsymbol{\xi})=
u\,\left[\mathrm{Ad}_g(\boldsymbol{\eta})_\kappa+O(r)\right]-
\boldsymbol{\beta}.
\end{equation}
Here $\boldsymbol{\beta}=\mathrm{Coad}_h(\boldsymbol{\lambda})$ for some
$h\in G$, and we may assume that $0<r\ll \|\boldsymbol{\lambda}\|$.
We conclude that for some 
$\epsilon_0>0$ we have
$\|\partial_{\boldsymbol{\xi}}\Gamma_x(u,\boldsymbol{\beta},\boldsymbol{\xi})\|\ge \epsilon_0$ if $0<u\ll 1$
and $\|\partial_{\boldsymbol{\xi}}\Gamma_x(u,\boldsymbol{\beta},\boldsymbol{\xi})\|\ge u\,\epsilon_0$ if $u\gg 1$.
Using a standard  argument based on \lq integration by parts\rq\, 
in the compactly supported variable $\boldsymbol{\xi}$, we conclude the following.

\begin{lem}
\label{lem:compact supp u}
Suppose $D\gg 0$ and let $\rho_3\in \mathcal{C}^\infty_0(1/(2\,D),2\,D)$
be $\equiv 1$ on $(1/D,D)$. Then the asymptotics of
(\ref{eqn:Piklambda g expanded resc}) are unchanged, if 
the amplitude in (\ref{eqn:defn of Gamma e B}) is replaced by
$$\mathcal{B}_{x,k}(u,\boldsymbol{\beta},\boldsymbol{\xi}):=
\tilde{\mathcal{B}}_{x,k}(u,\boldsymbol{\beta},\boldsymbol{\xi})\,\rho_3(u)
.$$
\end{lem}

Let us define, for $k\gg 0$,
\begin{equation}
\label{eqn:cut-off rescaled}
\tilde{\rho}_k(\boldsymbol{\xi}):=
\tilde{\rho}\left(k^{\frac{1}{2}-\epsilon}\,\boldsymbol{\xi}\right).
\end{equation}
Thus $\tilde{\rho}_k$ is supported where 
$\|\boldsymbol{\xi}\|\le 2\,r\,k^{\epsilon-\frac{1}{2}}$,
and identically $\equiv 1$ where $\|\boldsymbol{\xi}\|\le 
r\,k^{\epsilon-\frac{1}{2}}$.
We conclude from (\ref{eqn:Piklambda g expanded resc}) that
$$
\Pi^\tau_{k\,\boldsymbol{\lambda}}(x,x)\sim 
\Pi^\tau_{k\,\boldsymbol{\lambda}}(x,x)_1+
\Pi^\tau_{k\,\boldsymbol{\lambda}}(x,x)_2,
$$
where $\Pi^\tau_{k\,\boldsymbol{\lambda}}(x,x)_1$ (respectively,
$\Pi^\tau_{k\,\boldsymbol{\lambda}}(x,x)_2$) is as in  
(\ref{eqn:Piklambda g expanded resc}), except that $\mathcal{B}_{x,k}$ in 
(\ref{eqn:defn of Gamma e B}) has been multiplied by $\tilde{\rho}_k(\boldsymbol{\xi})$
(respectively, by $1-\tilde{\rho}_k(\boldsymbol{\xi})$).

\begin{lem}
$\Pi^\tau_{k\,\boldsymbol{\lambda}}(x,x)_2=O\left(k^{-\infty}\right)$
as $k\rightarrow+\infty$.
\end{lem}
In the following argument, $C^\tau_j$ will denote suitable positive constants (uniform on $X^\tau_{\mathcal{O}}$ for $\boldsymbol{\xi}$ small).
\begin{proof}
In view of (\ref{eqn:bd dist psitau xy}) and 
 (\ref{eqn:defn of Gamma e B}),
\begin{eqnarray*}
\left|\partial_u\Gamma_x(u,\boldsymbol{\beta},\boldsymbol{\xi})\right|&=&
\left|\psi^\tau\left(\mu^\tau_{e^{-\boldsymbol{\xi}}}(x),x   \right)\right|
\ge \Im\left( \psi^\tau\left(\mu^\tau_{e^{-\boldsymbol{\xi}}}(x),x   \right)  \right)\\
&\ge &C^\tau\,\mathrm{dist}_{X^\tau}\left(\mu^\tau_{e^{-\boldsymbol{\xi}}}(x),x   \right) ^2\ge C_1^\tau\,\|\boldsymbol{\xi}\|^2.
\nonumber
\end{eqnarray*}
Hence, on the support of $1-\tilde{\rho}_k(\boldsymbol{\xi})$ we have
\begin{eqnarray}
\left|\partial_u\Gamma_x(u,\boldsymbol{\beta},\boldsymbol{\xi})\right|
\ge C_2^\tau\,k^{2\,\epsilon-1}.
\nonumber
\end{eqnarray}
Thus, by iteratively integrating by parts in $\mathrm{d}u$ 
(which is now legitimate by Lemma \ref{lem:compact supp u})
we introduce at each step a 
factor which is bounded by $C_3^\tau\,k^{-2\,\epsilon}$.
The claim follows.

\end{proof}

We conclude that 
\begin{equation}
\label{eqn:Pi is Pi1}
\Pi^\tau_{k\,\boldsymbol{\lambda}}(x,x)\sim 
\Pi^\tau_{k\,\boldsymbol{\lambda}}(x,x)_1.
\end{equation}
Let us perform the rescaling $\boldsymbol{\xi}\mapsto k^{-1/2}\,\boldsymbol{\xi}$.
We can rewrite (\ref{eqn:Piklambda g expanded resc}) as follows:
\begin{eqnarray}
\label{eqn:Piklambda g expanded resc resc}
\Pi^\tau_{k\,\boldsymbol{\lambda}}(x,x)&\sim&
k^{1-\frac{d_G}{2}}\,d_{k\boldsymbol{\lambda}}\,\left(\frac{k}{2\,\pi}\right)^{n_G}\,
\int_{\mathfrak{g}}\,
\mathrm{d}^HV_{\mathfrak{g}}(\boldsymbol{\xi})\,\int_{\mathcal{O}_{\boldsymbol{\lambda}}}\,
\mathrm{d}
V_{\mathcal{O}_{\boldsymbol{\lambda}}}(\boldsymbol{\beta})
\,\int_0^{+\infty}\,\mathrm{d}u
\nonumber
\\
&&\left[
 e^{\imath\,k\,\Gamma_x(u,\boldsymbol{\beta},\boldsymbol{\xi}/\sqrt{k})
}\,\mathcal{B}_{x,k}\left(u,\boldsymbol{\beta},\frac{\boldsymbol{\xi}}{\sqrt{k}}\right)\,\tilde{\rho}\left(k^{-\epsilon}\,\boldsymbol{\xi}\right)\right];
\end{eqnarray}
integration in $\boldsymbol{\xi}$ is now over an expanding ball of
radius $O\left(k^{\epsilon}\right)$, and integration in $u$ is
over $(1/D,D)$.

Let us fix a system of NHLC's at $x=g\,\exp^{\tilde{G}}(\imath\,\boldsymbol{\eta})$.
By the discussion in \S 2.5 of \cite{gp24},
we have
\begin{eqnarray}
\label{eqn:action diagon}
\lefteqn{\mu^\tau_{e^{-\boldsymbol{\xi}/\sqrt{k}}}(x)}\\
&=&
x+\left(
\frac{1}{\sqrt{k}}\,
\kappa_e\big( \mathrm{Ad}_g(\boldsymbol{\eta}),\boldsymbol{\xi}\big)
+R_3\left( \frac{\boldsymbol{\xi}}{\sqrt{k}}  \right),
-\boldsymbol{\xi}^\sharp_{\tilde{G}}(x)+\mathbf{R}_2\left( \frac{\boldsymbol{\xi}}{\sqrt{k}}  \right)\right).\nonumber
\end{eqnarray}
In view of Proposition 48 of \cite{p24}, we obtain
\begin{eqnarray}
\label{eqn:psi act diag}
\lefteqn{\imath\,k\,\psi^\tau\left(\mu^\tau_{e^{-\boldsymbol{\xi}/\sqrt{k}}}(x),x   \right)}\\
&=&\imath\,\sqrt{k}\cdot\kappa_e\big( \mathrm{Ad}_g(\boldsymbol{\eta}),\boldsymbol{\xi}\big)
-\frac{1}{4\,\tau^2}\,\kappa_e\big( \mathrm{Ad}_g(\boldsymbol{\eta}),\boldsymbol{\xi}\big)^2-\frac{1}{2}\,\|\boldsymbol{\xi}^\sharp_{\tilde{G}}(x)\|^2+k\,R_3\left( \frac{\boldsymbol{\xi}}{\sqrt{k}} \right). 
\nonumber
\end{eqnarray}
In particular,
\begin{equation}
\label{eqn:real part neg}
\Re\left(\imath\,k\,\psi^\tau\left(\mu^\tau_{e^{-\boldsymbol{\xi}/\sqrt{k}}}(x),x   \right)   \right)\le -{C'}^\tau\,\|\boldsymbol{\xi}\|^2_{\kappa_e}
\end{equation}
for some constant ${C'}^\tau>0$.

We can then rewrite (\ref{eqn:Piklambda g expanded resc resc}) in the 
following form:
\begin{eqnarray}
\label{eqn:Piklambda g expanded resc resc xi}
\Pi^\tau_{k\,\boldsymbol{\lambda}}(x,x)&\sim&
k^{1-\frac{d_G}{2}}\,d_{k\boldsymbol{\lambda}}\,\left(\frac{k}{2\,\pi}\right)^{n_G}\,
\int_{\mathfrak{g}}\,
\mathrm{d}^HV_{\mathfrak{g}}(\boldsymbol{\xi})\,\int_{\mathcal{O}_{\boldsymbol{\lambda}}}\,
\mathrm{d}
V_{\mathcal{O}_{\boldsymbol{\lambda}}}(\boldsymbol{\beta})
\,\int_0^{+\infty}\,\mathrm{d}u
\nonumber
\\
&&\left[
 e^{\imath\,\sqrt{k}\,\Upsilon_x(u,\boldsymbol{\beta},\boldsymbol{\xi})
}\,\mathcal{H}_{x,k}
\left(u,\boldsymbol{\beta},\boldsymbol{\xi}\right)\right],
\end{eqnarray}
where 
\begin{eqnarray}
\label{eqn:defn of Upsilon}
\Upsilon_x(u,\boldsymbol{\beta},\boldsymbol{\xi})&:=&
u\,\kappa_e\big( \mathrm{Ad}_g(\boldsymbol{\eta}),\boldsymbol{\xi}\big)
-\boldsymbol{\beta}(\boldsymbol{\xi})=\left(u\,
\tilde{\Phi}^\tau(x)-\boldsymbol{\beta}\right)(\boldsymbol{\xi}),
\\
\mathcal{H}_{x,k}
\left(u,\boldsymbol{\beta},\boldsymbol{\xi}\right)&:=&
e^{-\frac{u}{4\,\tau^2}\,\kappa_e\big( \mathrm{Ad}_g(\boldsymbol{\eta}),\boldsymbol{\xi}\big)^2-\frac{u}{2}\,\|\boldsymbol{\xi}^\sharp_{\tilde{G}}(x)\|^2+k\,u\,R_3\left( \frac{\boldsymbol{\xi}}{\sqrt{k}} \right)}
\nonumber
\\
&&\cdot\mathcal{B}_{x,k}\left(u,\boldsymbol{\beta},\frac{\boldsymbol{\xi}}{\sqrt{k}}\right)\,\tilde{\rho}\left(k^{-\epsilon}\,\boldsymbol{\xi}\right).
\nonumber
\end{eqnarray}
Under the hypothesis of the Theorem, 
$\mathrm{dist}_{X^\tau}(x,X^\tau_{\mathcal{O}})
\ge C\,k^{\epsilon-\frac{1}{2}}$, and therefore
for some constant $C'>0$
$$
\left\| \tilde{\Phi}^\tau(x)-\boldsymbol{\beta}  \right\|_{\kappa^\vee}=
\|\mathrm{Coad}_g(\boldsymbol{\eta})-\boldsymbol{\beta}\|_{\kappa^\vee}
\ge C'\,k^{\epsilon-\frac{1}{2}}
$$ for any 
$\boldsymbol{\beta}\in \mathcal{O}^\tau$ and $g\in G$.
It follows that (perhaps changing
$C'>0$) we also have 
\begin{equation}
\label{eqn:distance beta eta g u}
\|u\,\mathrm{Coad}_g(\boldsymbol{\eta})-\boldsymbol{\beta}\|
\ge C'\,k^{\epsilon-\frac{1}{2}}
\quad\text{for any}
\quad\boldsymbol{\beta}\in \mathcal{O}_{\boldsymbol{\lambda}},\,u\in (1/D,D),
\,g\in G.
\end{equation}
In other words, by (\ref{eqn:defn of Upsilon}) we have
\begin{equation}
\label{eqn:bound partial xi upsilon}
\left\|\partial_{\boldsymbol{\xi}}\Upsilon_x(u,\boldsymbol{\beta},\boldsymbol{\xi})\right\|\ge 
C'\,k^{\epsilon-\frac{1}{2}}.
\end{equation}
Hence, by a standard argument, iteratively integrating by parts in $\boldsymbol{\xi}$ we introduce at 
each step a factor which bounded by $C''\,k^{-\epsilon+\frac{1}{2}}/
\sqrt{k}=C''\,k^{-\epsilon}$. Then claim follows.

\end{proof}

\begin{proof}
[Proof of Theorem \ref{thm:rapid decay moment map}]
To see why Theorem \ref{thm:rapid decay moment map} follows from
Theorem \ref{thm:rapid decay moment map diag}, recall that
in view of the results in \cite{cr1} and (\ref{eqn:eigenvalue of lambda})
one has an \textit{a priori} uniform
bound $\Pi^\tau_{k\,\boldsymbol{\lambda}}(y,y)\le 
C^\tau_{\boldsymbol{\lambda}}\,k^N $ ($y\in X^\tau$) for some $N$ ($N=2d-2$ will do).
On the other hand,
$$
\left|\Pi^\tau_{k\,\boldsymbol{\lambda}}(x,y)\right|\le 
\Pi^\tau_{k\,\boldsymbol{\lambda}}(x,x)^{\frac{1}{2}}\,\Pi^\tau_{k\,\boldsymbol{\lambda}}(y,y)^{\frac{1}{2}}.
$$
Therefore, since $,\Pi^\tau_{k\,\boldsymbol{\lambda}}(y,y)$ has at most
polynomial growth and $\Pi^\tau_{k\,\boldsymbol{\lambda}}(x,x)
=O\left(k^{-\infty}\right)$, we also have $\Pi^\tau_{k\,\boldsymbol{\lambda}}(x,y)
=O\left(k^{-\infty}\right)$ for $k\rightarrow +\infty$.
The argument for $P^\tau_{k\,\boldsymbol{\lambda}}$ is similar.
\end{proof}

\subsection{Proof of Theorem \ref{thm:rescaled asympt}}

\subsubsection{Preliminaries}
\label{sctn:reductions initial}

If $x\in X^\tau$ and $h\in G$,
composing a system of NHLC's at $x\in X^\tau$
with $\mu^\tau_h$ yields a system of NHLC's 
at $\mu^\tau_h(x)\in X^\tau$. Furthermore, 
since $\alpha^\tau$, $\hat{\kappa}$, and $X^\tau_{\mathcal{O}}$ are $G$-invariant, $\mu_h$ preserves
the relevant local decompositions of 
$T_xX^\tau$ and
$\mathcal{H}_x$ (see
(\ref{eqn:dsdecomp HSN}), (\ref{eqn:direct sum decmp TxX}), (\ref{eqn:decomp Xtau})).

On the other hand, for any $(x,y)\in X^\tau_{\mathcal{O}}$ and
$h\in G$, we have 
$$\Pi^\tau_{k\,\boldsymbol{\lambda}}
\left(\mu_h^\tau(x),\mu^\tau_h(y)  \right)
=\Pi^\tau_{k\,\boldsymbol{\lambda}}(x,y).$$
This remark allows for the following preliminary reduction.
Suppose 
$x=g\cdot\exp^{\tilde{G}}(\imath\,\boldsymbol{\eta}) 
\in X_{\mathcal{O}}^\tau$, so that
$\Phi^\tau(x)=
\mathrm{Ad}_g(\boldsymbol{\eta})_\kappa\in 
\mathcal{O}^\tau$
(\S \ref{sctn:ham str L}).
Perhaps replacing $x$ with $\mu^\tau_h(x)=
h\,g\cdot\exp^{\tilde{G}}(\imath\,\boldsymbol{\eta}) $
for a suitable $h\in G$, we may assume without loss
of generality that $\Phi^\tau(x)=a^\tau\,\boldsymbol{\lambda}$,
where $a^\tau=\tau/\|\boldsymbol{\lambda}\|_{\kappa_e^\vee}$.

On the other hand, 
if $\Phi^\tau(x)=a^\tau\,\boldsymbol{\lambda}$ then
by Lemma \ref{lem:same orthocomplement} for any 
$\boldsymbol{\xi}\in \mathfrak{g}$ we have
\begin{equation*}
\boldsymbol{\xi}=\boldsymbol{\xi}^{\|}+
\boldsymbol{\xi}^\perp,\quad \text{where}\quad
\boldsymbol{\xi}^{\|}\in \mathrm{span}(\boldsymbol{\lambda}),
\quad \boldsymbol{\xi}^{\perp}\in \mathrm{span}(\boldsymbol{\lambda}^\kappa)^{\perp_{\kappa_e}}
=\mathrm{span}
(\boldsymbol{\lambda}^\kappa)^{\perp_{\kappa_x}}.
\end{equation*}
Thus
\begin{equation}
\label{eqn:xi decomposed}
\boldsymbol{\xi}^{\|}=b_{\boldsymbol{\xi}}\,
\boldsymbol{\lambda}^\kappa_u\quad
\text{where}\quad \boldsymbol{\lambda}^\kappa_u:=
\frac{1}{\|\boldsymbol{\lambda}^\kappa\|_{\kappa_e}}
\,\boldsymbol{\lambda}^\kappa,
\quad b_{\boldsymbol{\xi}}:=\kappa_e\left(\boldsymbol{\xi},
\boldsymbol{\lambda}^\kappa_u\right).
\end{equation}
Recalling the discussion in \S \ref{sctn:direct sum dec g}
(especially (\ref{eqn:direct sums T_x})),
\begin{equation}
\label{eqn:xiparxiperpRH}
\boldsymbol{\xi}^{\|}_{\tilde{G}}(x)\in 
\mathrm{span}\left(\mathcal{R}^\tau(x)\right),
\quad
\boldsymbol{\xi}^{\perp}_{\tilde{G}}(x)=
\boldsymbol{\xi}_{\tilde{G}}(x)^{\sharp}\in 
\mathcal{H}^\tau(x).
\end{equation}
We can further decompose $\boldsymbol{\xi}^\perp$
according to the two alternative direct sums in
Definition \ref{defn:sfrakx}. Namely, we have the $\kappa_e$-orthogonal
direct sum
\begin{equation}
\label{eqn:xi perp decomp 1}
\boldsymbol{\xi}^\perp=\boldsymbol{\xi}'+\boldsymbol{\xi}'',
\quad\text{where}\quad\boldsymbol{\xi}'\in 
\mathfrak{t}'_{\boldsymbol{\lambda}},\,
\boldsymbol{\xi}''\in 
\mathfrak{r}_x=\mathfrak{t}_{\boldsymbol{\lambda}}^{\perp_{\kappa_e}}=\mathfrak{t}^{\perp_{\kappa_e}},
\end{equation}
where we use that $\boldsymbol{\lambda}$ is a regular weight,
and the $\sigma_x$-orthogonal one
\begin{equation}
\label{eqn:xitsperp}
\boldsymbol{\xi}^\perp=\boldsymbol{\xi}_{\mathfrak{t}}+\boldsymbol{\xi}_{\mathfrak{s}},
\quad\text{where}\quad\boldsymbol{\xi}_{\mathfrak{t}}\in 
\mathfrak{t}'_{\boldsymbol{\lambda}},\,
\boldsymbol{\xi}_{\mathfrak{s}}\in 
\mathfrak{s}_x.
\end{equation}

\begin{rem}
\label{rem:xii=implisxis=0}
If $\boldsymbol{\xi}''=\mathbf{0}$, then 
$\boldsymbol{\xi}^\perp=\boldsymbol{\xi}'=
\boldsymbol{\xi}_{\mathfrak{t}}\in \mathfrak{t}'_{\boldsymbol{\lambda}}$,
hence $\boldsymbol{\xi}_{\mathfrak{s}}=\mathbf{0}$ (and conversely).
\end{rem}

We aim to study the asymptotics of $\Pi^\tau_{k\,\boldsymbol{\lambda}}(x_{1k},x_{2k})$, where $x_{jk}$ is as in (\ref{eqn:defn di xjk}).
In view of Lemma \ref{lem:normal bundle GO} and Remark
\ref{rem:NSsigma}, there exist unique 
$\boldsymbol{\rho}_j\in \mathfrak{t}'_{\boldsymbol{\lambda}}$,
$\boldsymbol{\eta}_j\in \tilde{\mathfrak{s}}_x$ 
such that 
\begin{equation}
\label{eqn:njsj}
\mathbf{n}_j=J_x\big({\boldsymbol{\rho}_j}_{\tilde{G}}(x)\big),\quad 
\mathbf{s}_j={\boldsymbol{\eta}_j}_{\tilde{G}}(x).
\end{equation}

\subsubsection{Proof of Theorem \ref{thm:rescaled asympt}}

In the following, we assume that $\Phi^\tau(x)=a^\tau\,\boldsymbol{\lambda}$ and fix NHLC's on
$X^\tau$ at $x$; 
In particular, for any $\boldsymbol{\xi}\in \mathfrak{g}$ we have
\begin{equation}
\label{eqn:varphixi in x}
\varphi^{\boldsymbol{\xi}}(x)=
\left\langle\tilde{\Phi}^\tau(x),\boldsymbol{\xi}\right\rangle
=a^\tau\,\left\langle\boldsymbol{\lambda},\boldsymbol{\xi}\right\rangle
=a^\tau\,\kappa_e\left(\boldsymbol{\lambda}^\kappa,\boldsymbol{\xi}\right).
\end{equation}

The definition of $x_{jk}$ is in (\ref{eqn:defn di xjk}).
Thus
$\mathbf{n}_j\in \mathbb{R}^{r_{\mathcal{O}}-1}_{N}\cong N(X^\tau_{\mathcal{O}}/X^\tau)_x$ and, by Lemma 
\ref{lem:normal bundle GO} and (\ref{eqn:varphixi in x}),
$\mathbf{n}_j=J_x\big({\boldsymbol{\rho}_j}_{\tilde{G}}(x)\big)$
for a unique $\boldsymbol{\rho}_j\in \mathfrak{t}'_x=\mathfrak{t}'_{\boldsymbol{\lambda}}$.
Similarly,
$\mathbf{s}_j\in \mathbb{C}^{d-r_{\mathcal{O}}}_{\mathcal{S}}
\cong 
\mathcal{S}_x
$; by Remark \ref{rem:NSsigma}, there exists a unique 
$\boldsymbol{\eta}_j\in \widetilde{\mathfrak{s}}_x=\mathfrak{s}_x\oplus \imath\mathfrak{s}_x$ such that 
$\mathbf{s}_j={\boldsymbol{\eta}_j}_{\tilde{G}}(x)$.

\begin{proof}
By a minor modification of the arguments leading to (\ref{eqn:Piklambda g expanded resc resc}) (perhaps with a
different choice of the radius $r>0$ involved in the
definition of
$\rho_2$ in (\ref{eqn:Piklambda G}), hence of
$\tilde{\rho}$ in (\ref{eqn:defn of Gamma e B}) and
 $\tilde{\rho}_k$ in (\ref{eqn:cut-off rescaled})),
one obtains
\begin{eqnarray}
\label{eqn:Piklambda g expanded resc resc 12}
\Pi^\tau_{k\,\boldsymbol{\lambda}}(x_{1k},x_{2k})&\sim&
k^{1-\frac{d_G}{2}}\,d_{k\boldsymbol{\lambda}}\,\left(\frac{k}{2\,\pi}\right)^{n_G}\,
\int_{\mathfrak{g}}\,
\mathrm{d}^HV_{\mathfrak{g}}(\boldsymbol{\xi})\,\int_{\mathcal{O}_{\boldsymbol{\lambda}}}\,
\mathrm{d}
V_{\mathcal{O}_{\boldsymbol{\lambda}}}(\boldsymbol{\beta})
\,\int_0^{+\infty}\,\mathrm{d}u
\nonumber
\\
&&\left[
 e^{\imath\,k\,\Lambda _{x,k}(u,\boldsymbol{\beta},\boldsymbol{\xi}/\sqrt{k})
}\,\mathcal{D}_{x,k}\left(u,\boldsymbol{\beta},\boldsymbol{\xi}\right)\right],
\end{eqnarray}
where 
\begin{eqnarray}
\label{eqn:LambdacalD}
\Lambda_{x,k}\left(u,\boldsymbol{\beta},
\boldsymbol{\xi}\right)
&:=&
u\,\psi^\tau\left(\mu^\tau_{e^{-\boldsymbol{\xi}/\sqrt{k}}}(x_{1k}),x_{2k}   \right)-
\frac{1}{\sqrt{k}}\,\boldsymbol{\beta}(\boldsymbol{\xi})\\
\mathcal{D}_{x,k}\left(u,\boldsymbol{\beta},\boldsymbol{\xi}\right)&:=&e^{-
\boldsymbol{\delta}_{\boldsymbol{\beta}}(\boldsymbol{\xi})/\sqrt{k}}
\,
s^\tau\left(\mu^\tau_{e^{-\boldsymbol{\xi}/\sqrt{k}}}(x_{1k}),
x_{2k}  ,k\,u\right)\nonumber\\
&&\cdot\mathcal{V}_k(\boldsymbol{\beta})\,
\mathcal{P}\left(\frac{\boldsymbol{\xi}}{\sqrt{k}}\right)
\,\tilde{\rho}\left(k^{-\epsilon}\,\boldsymbol{\xi}\right).
\nonumber
\end{eqnarray}

In order to expand $\Lambda_{x,k}\left(u,\boldsymbol{\beta},
\boldsymbol{\xi}\right)$, consider first
$\mu^\tau_{e^{-\boldsymbol{\xi}/\sqrt{k}}}(x_{1k})$. In view of Lemma 64 of \cite{gp24}, we have the following refinement of (\ref{eqn:action diagon}):
\begin{eqnarray}
\label{eqn:mu action resc}
\lefteqn{
\mu^\tau_{e^{-\boldsymbol{\xi}/\sqrt{k}}}(x_{1k})}\\
&=&x+\left(\frac{1}{\sqrt{k}}\,\left(\theta_1+a^\tau\,\left\langle\boldsymbol{\lambda},\boldsymbol{\xi}\right\rangle  \right)+
\frac{1}{k}\,\omega_x \left(\boldsymbol{\xi}_{\tilde{G}}(x)^\sharp,\mathbf{n}_j+\mathbf{s}_j\right)
+\mathbf{R}_3\left(\frac{\bullet}{\sqrt{k}}\right),
\right.\nonumber\\
&&\left. \frac{1}{\sqrt{k}}\,\big(
\mathbf{n}_1+\mathbf{s}_1
-\boldsymbol{\xi}_{\tilde{G}}(x)^\sharp\big)
+R_2\left(\frac{\bullet}{\sqrt{k}}\right)    \right),\nonumber
\end{eqnarray}
where \lq$\bullet$\rq\, collectively denotes the variables involved.
In view of
the initial discussion in \S \ref{sctn:reductions initial}
and (\ref{eqn:xi decomposed}),
\begin{equation}
\label{eqn:val ataulambda xi}
a^\tau\,\left\langle\boldsymbol{\lambda},\boldsymbol{\xi}\right\rangle =\frac{\tau}{\|\boldsymbol{\lambda}^\kappa\|_{\kappa_e}}
\,\kappa_e\left(\boldsymbol{\lambda}^\kappa,\boldsymbol{\xi}\right)=\tau\,\kappa_e
\left(\boldsymbol{\lambda}^\kappa_u,
\boldsymbol{\xi}\right)=\tau\,b_{\boldsymbol{\xi}}.
\end{equation}
By (\ref{eqn:xiparxiperpRH})
and (\ref{eqn:xitsperp}), we have
\begin{eqnarray}
\label{eqn:omegaxins}
\lefteqn{\omega_x \left(\boldsymbol{\xi}_{\tilde{G}}(x)^\sharp,\mathbf{n}_1+\mathbf{s}_1\right)}\\
&=&
\omega_x \Big(
{\boldsymbol{\xi}_{\mathfrak{t}}}_{\tilde{G}}(x)
+{\boldsymbol{\xi}_{\mathfrak{s}}}_{\tilde{G}}(x),
J_x\big({\boldsymbol{\rho}_1}_{\tilde{G}}(x)\big)
+{\boldsymbol{\eta}_1}_{\tilde{G}}(x)\Big)
\nonumber\\
&=&\omega_x \Big(
{\boldsymbol{\xi}_{\mathfrak{t}}}_{\tilde{G}}(x)
,
J_x\big({\boldsymbol{\rho}_1}_{\tilde{G}}(x)\big)\Big)
+\omega_x \Big(
{\boldsymbol{\xi}_{\mathfrak{s}}}_{\tilde{G}}(x),
{\boldsymbol{\eta}_1}_{\tilde{G}}(x)\Big)
\nonumber\\
&=&\tilde{\kappa}_x \Big(
{\boldsymbol{\xi}_{\mathfrak{t}}}_{\tilde{G}}(x)
,
{\boldsymbol{\rho}_1}_{\tilde{G}}(x)\Big)
+\omega_x \Big(
{\boldsymbol{\xi}_{\mathfrak{s}}}_{\tilde{G}}(x),
{\boldsymbol{\eta}_1}_{\tilde{G}}(x)\Big).
\nonumber
\end{eqnarray}

Given Proposition 48 of [P2024], we obtain
\begin{eqnarray}
\label{eqn:ipsi developped}
\lefteqn{\imath\,\psi^\tau\left(\mu^\tau_{e^{-\boldsymbol{\xi}/\sqrt{k}}}(x_{1k}),x_{2k}   \right)   }\\
&=&   \frac{\imath}{\sqrt{k}}\,\left(\theta_1-\theta_2+\tau\,b_{\boldsymbol{\xi}} \right)                              
+\frac{\imath}{k}\,
\left[ \tilde{\kappa}_x \Big(
{\boldsymbol{\xi}_{\mathfrak{t}}}_{\tilde{G}}(x)
,
{\boldsymbol{\rho}_1}_{\tilde{G}}(x)\Big)
+\omega_x \Big(
{\boldsymbol{\xi}_{\mathfrak{s}}}_{\tilde{G}}(x),
{\boldsymbol{\eta}_1}_{\tilde{G}}(x)\Big)  \right]\nonumber\\
&&
-\frac{1}{4\,\tau^2\,k}\,\left( \theta_1-\theta_2+\tau\,b_{\boldsymbol{\xi}}   \right)^2\nonumber\\
&&
+\frac{1}{k}\,\psi_2^{\omega_x}\Big( 
J_x\big({\boldsymbol{\rho}_1}_{\tilde{G}}(x)\big)
+{\boldsymbol{\eta}_1}_{\tilde{G}}(x)
-\boldsymbol{\xi}_{\tilde{G}}(x)^\sharp ,
J_x\big({\boldsymbol{\rho}_2}_{\tilde{G}}(x)\big)
+{\boldsymbol{\eta}_2}_{\tilde{G}}(x)  \Big)
+R_3\left( \frac{\bullet}{\sqrt{k}} \right).
\nonumber
\end{eqnarray}
We have
\begin{eqnarray}
\label{eqn:psi2 sviluppato}
\lefteqn{
\psi_2^{\omega_x}\Big( 
J_x\big({\boldsymbol{\rho}_1}_{\tilde{G}}(x)\big)
+{\boldsymbol{\eta}_1}_{\tilde{G}}(x)
-\boldsymbol{\xi}_{\tilde{G}}(x)^\sharp ,
J_x\big({\boldsymbol{\rho}_2}_{\tilde{G}}(x)\big)
+{\boldsymbol{\eta}_2}_{\tilde{G}}(x)  \Big)
}\\
&=&-\imath\,\omega_x\Big(
J_x\big({\boldsymbol{\rho}_1}_{\tilde{G}}(x)\big)
+{\boldsymbol{\eta}_1}_{\tilde{G}}(x)
-{\boldsymbol{\xi}_{\mathfrak{s}}}_{\tilde{G}}(x)
-{\boldsymbol{\xi}_{\mathfrak{t}}}_{\tilde{G}}(x),
J_x\big({\boldsymbol{\rho}_2}_{\tilde{G}}(x)\big)
+{\boldsymbol{\eta}_2}_{\tilde{G}}(x)  
\Big)\nonumber\\
&&-\frac{1}{2}\,\Big\|
J_x\big({\boldsymbol{\rho}_1}_{\tilde{G}}(x)-{\boldsymbol{\rho}_2}_{\tilde{G}}(x)\big)
+\big({\boldsymbol{\eta}_1}_{\tilde{G}}(x)
-{\boldsymbol{\eta}_2}_{\tilde{G}}(x)-{\boldsymbol{\xi}_{\mathfrak{s}}}_{\tilde{G}}(x)\big)
-{\boldsymbol{\xi}_{\mathfrak{t}}}_{\tilde{G}}(x)
%
%
\Big\|^2. \nonumber
\end{eqnarray}
Furthermore,
\begin{eqnarray*}
\lefteqn{\omega_x\Big(
J_x\big({\boldsymbol{\rho}_1}_{\tilde{G}}(x)\big)
+{\boldsymbol{\eta}_1}_{\tilde{G}}(x)
-{\boldsymbol{\xi}_{\mathfrak{s}}}_{\tilde{G}}(x)
-{\boldsymbol{\xi}_{\mathfrak{t}}}_{\tilde{G}}(x) ,
J_x\big({\boldsymbol{\rho}_2}_{\tilde{G}}(x)\big)
+{\boldsymbol{\eta}_2}_{\tilde{G}}(x)  
\Big)}\nonumber\\
&=&\omega_x\Big(
{\boldsymbol{\eta}_1}_{\tilde{G}}(x)
,
{\boldsymbol{\eta}_2}_{\tilde{G}}(x)  
\Big)-\omega_x\Big( {\boldsymbol{\xi}_{\mathfrak{s}}}_{\tilde{G}}(x),{\boldsymbol{\eta}_2}_{\tilde{G}}(x)  \Big)
-\tilde{\kappa}_x\Big({\boldsymbol{\xi}_{\mathfrak{t}}}_{\tilde{G}}(x) ,
{\boldsymbol{\rho}_2}_{\tilde{G}}(x)   \Big),
\nonumber
\end{eqnarray*}
\begin{eqnarray*}
\lefteqn{
\Big\|
J_x\big({\boldsymbol{\rho}_1}_{\tilde{G}}(x)-{\boldsymbol{\rho}_2}_{\tilde{G}}(x)\big)
+\big({\boldsymbol{\eta}_1}_{\tilde{G}}(x)
-{\boldsymbol{\eta}_2}_{\tilde{G}}(x)-{\boldsymbol{\xi}_{\mathfrak{s}}}_{\tilde{G}}(x)\big)
-{\boldsymbol{\xi}_{\mathfrak{t}}}_{\tilde{G}}(x)
%
%
\Big\|^2
}\\
&=&\Big\|
{\boldsymbol{\rho}_1}_{\tilde{G}}(x)-{\boldsymbol{\rho}_2}_{\tilde{G}}(x)\Big\|^2
+\Big\|{\boldsymbol{\eta}_1}_{\tilde{G}}(x)
-{\boldsymbol{\eta}_2}_{\tilde{G}}(x)-{\boldsymbol{\xi}_{\mathfrak{s}}}_{\tilde{G}}(x)\Big\|^2
+\Big\|{\boldsymbol{\xi}_{\mathfrak{t}}}_{\tilde{G}}(x)
%
%
\Big\|^2.
\end{eqnarray*}

Inserting this in (\ref{eqn:ipsi developped})
\begin{eqnarray}
\label{eqn:ipsi developped 1}
\lefteqn{\imath\,\psi^\tau\left(\mu^\tau_{e^{-\boldsymbol{\xi}/\sqrt{k}}}(x_{1k}),x_{2k}   \right)   }\\
&=&   \frac{\imath}{\sqrt{k}}\,\left(\theta_1-\theta_2+\tau\,b_{\boldsymbol{\xi}} \right)                              
+\frac{\imath}{k}\,
\left[ \tilde{\kappa}_x \Big(
{\boldsymbol{\xi}_{\mathfrak{t}}}_{\tilde{G}}(x)
,
{\boldsymbol{\rho}_1}_{\tilde{G}}(x)\Big)
+\omega_x \Big(
{\boldsymbol{\xi}_{\mathfrak{s}}}_{\tilde{G}}(x),
{\boldsymbol{\eta}_1}_{\tilde{G}}(x)\Big)  \right]\nonumber\\
&&
-\frac{1}{4\,\tau^2\,k}\,\left( \theta_1-\theta_2+\tau\,b_{\boldsymbol{\xi}}   \right)^2\nonumber\\
&&
+\frac{\imath}{k}\,\Big[
-\omega_x\Big(
{\boldsymbol{\eta}_1}_{\tilde{G}}(x)
,
{\boldsymbol{\eta}_2}_{\tilde{G}}(x)  
\Big)+\omega_x\Big( {\boldsymbol{\xi}_{\mathfrak{s}}}_{\tilde{G}}(x),{\boldsymbol{\eta}_2}_{\tilde{G}}(x)  \Big)
+\tilde{\kappa}_x\Big({\boldsymbol{\xi}_{\mathfrak{t}}}_{\tilde{G}}(x) ,
{\boldsymbol{\rho}_2}_{\tilde{G}}(x)   \Big)
\Big]\nonumber\\
&&-\frac{1}{2\,k}\,\left( \Big\|
{\boldsymbol{\rho}_1}_{\tilde{G}}(x)-{\boldsymbol{\rho}_2}_{\tilde{G}}(x)\Big\|^2
+\Big\|{\boldsymbol{\eta}_1}_{\tilde{G}}(x)
-{\boldsymbol{\eta}_2}_{\tilde{G}}(x)-{\boldsymbol{\xi}_{\mathfrak{s}}}_{\tilde{G}}(x)\Big\|^2
+\Big\|{\boldsymbol{\xi}_{\mathfrak{t}}}_{\tilde{G}}(x)
%
%
\Big\|^2       \right)\nonumber\\
&&
+R_3\left( \frac{\bullet}{\sqrt{k}} \right)
\nonumber\\
&=&   \frac{\imath}{\sqrt{k}}\,\left(\theta_1-\theta_2+\tau\,b_{\boldsymbol{\xi}} \right)  \nonumber\\
&&
+\frac{\imath}{k}\,
\left[ \tilde{\kappa}_x \Big(
{\boldsymbol{\xi}_{\mathfrak{t}}}_{\tilde{G}}(x)
,
{\boldsymbol{\rho}_1}_{\tilde{G}}(x)
+{\boldsymbol{\rho}_2}_{\tilde{G}}(x)\Big)
+\omega_x \Big(
{\boldsymbol{\xi}_{\mathfrak{s}}}_{\tilde{G}}(x),
{\boldsymbol{\eta}_1}_{\tilde{G}}(x)
+{\boldsymbol{\eta}_2}_{\tilde{G}}(x)
\Big)  \right]\nonumber\\
&&
-\frac{1}{4\,\tau^2\,k}\,\left( \theta_1-\theta_2+\tau\,b_{\boldsymbol{\xi}}   \right)^2
-\frac{\imath}{k}\,
\omega_x\Big(
{\boldsymbol{\eta}_1}_{\tilde{G}}(x)
,
{\boldsymbol{\eta}_2}_{\tilde{G}}(x)  
\Big)
\nonumber\\
&&-\frac{1}{2\,k}\,\left( \Big\|
{\boldsymbol{\rho}_1}_{\tilde{G}}(x)-{\boldsymbol{\rho}_2}_{\tilde{G}}(x)\Big\|^2
+\Big\|{\boldsymbol{\eta}_1}_{\tilde{G}}(x)
-{\boldsymbol{\eta}_2}_{\tilde{G}}(x)-{\boldsymbol{\xi}_{\mathfrak{s}}}_{\tilde{G}}(x)\Big\|^2
+\Big\|{\boldsymbol{\xi}_{\mathfrak{t}}}_{\tilde{G}}(x)
%
%
\Big\|^2       \right)\nonumber\\
&&
+R_3\left( \frac{\bullet}{\sqrt{k}} \right).
\nonumber
\end{eqnarray}

In view of (\ref{eqn:LambdacalD}), we obtain
\begin{eqnarray}
\label{eqn:Lambda expanded}
\imath\,k\,\Lambda_{x,k}\left(u,\boldsymbol{\beta},
\boldsymbol{\xi}\right)&=&
\imath\,\sqrt{k}\,
\Psi
(u,\boldsymbol{\xi},\boldsymbol{\beta})+
u\,\Lambda(\boldsymbol{\xi}),
\end{eqnarray}
where 
\begin{eqnarray}
\label{eqn:Psidefnxi}
\Psi
(u,\boldsymbol{\xi},\boldsymbol{\beta})
=\Psi_{\theta_1,\theta_2}
(u,\boldsymbol{\xi},\boldsymbol{\beta})
&:=&u\,\left(\theta_1-\theta_2+\tau\,b_{\boldsymbol{\xi}} \right)
-\boldsymbol{\beta}(\boldsymbol{\xi}),
\end{eqnarray}
\begin{eqnarray}
\label{eqn:Lambdaxi defn 0}
\Lambda(\boldsymbol{\xi})=
\Lambda_2(\boldsymbol{\xi})+k\,R_3\left( \frac{\bullet}{\sqrt{k}} \right),
\end{eqnarray}
with
\begin{eqnarray}
\label{eqn:Lambdaxi defn}
\lefteqn{\Lambda_2(\boldsymbol{\xi})=
\Lambda_2(\theta_j,\boldsymbol{\eta}_j,\boldsymbol{\rho}_j;
\boldsymbol{\xi})}\\
&:=&
\imath\,
\left[ \tilde{\kappa}_x \Big(
{\boldsymbol{\xi}_{\mathfrak{t}}}_{\tilde{G}}(x)
,
{\boldsymbol{\rho}_1}_{\tilde{G}}(x)
+{\boldsymbol{\rho}_2}_{\tilde{G}}(x)\Big)
+\omega_x \Big(
{\boldsymbol{\xi}_{\mathfrak{s}}}_{\tilde{G}}(x),
{\boldsymbol{\eta}_1}_{\tilde{G}}(x)
+{\boldsymbol{\eta}_2}_{\tilde{G}}(x)
\Big)  \right]\nonumber\\
&&
-\frac{1}{4\,\tau^2}\,\left( \theta_1-\theta_2+\tau\,b_{\boldsymbol{\xi}}   \right)^2
-\imath\,
\omega_x\Big(
{\boldsymbol{\eta}_1}_{\tilde{G}}(x)
,
{\boldsymbol{\eta}_2}_{\tilde{G}}(x)  
\Big)
\nonumber\\
&&-\frac{1}{2}\,\left( \Big\|
{\boldsymbol{\rho}_1}_{\tilde{G}}(x)-{\boldsymbol{\rho}_2}_{\tilde{G}}(x)\Big\|^2
+\Big\|{\boldsymbol{\eta}_1}_{\tilde{G}}(x)
-{\boldsymbol{\eta}_2}_{\tilde{G}}(x)-{\boldsymbol{\xi}_{\mathfrak{s}}}_{\tilde{G}}(x)\Big\|^2
+\Big\|{\boldsymbol{\xi}_{\mathfrak{t}}}_{\tilde{G}}(x)
%
%
\Big\|^2       \right);
\nonumber
\end{eqnarray}
dependence of $\Psi$ and $\Lambda$
on $(\theta_j,\boldsymbol{\eta}_j,\boldsymbol{\rho}_j)$ will be left implicit in the following.

In view of (\ref{eqn:xi decomposed}) and
(\ref{eqn:Psidefnxi}), 
\begin{equation}
\label{eqn:der Psi xi}
\partial_{\boldsymbol{\xi}}\Psi
=u\,\tau\,\frac{\boldsymbol{\lambda}}{\|\boldsymbol{\lambda}\|}
-\boldsymbol{\beta}.
\end{equation}  
Let $\rho_2:\mathcal{O}_{\boldsymbol{\lambda}}
\rightarrow \mathbb{R}$ be $\mathcal{C}^\infty$ function 
which is supported in a suitably small neighborhood
of $\boldsymbol{\lambda}$ and identically equal to
$1$ near $\boldsymbol{\lambda}$.
Since integration in $\boldsymbol{\xi}$ is compactly
supported, we conclude from 
(\ref{eqn:Piklambda g expanded resc resc 12}),
(\ref{eqn:Lambda expanded}), and (\ref{eqn:der Psi xi})
that the following holds.
\begin{lem}
\label{lem:reduction in beta}
The asymptotics of (\ref{eqn:Piklambda g expanded resc resc 12}) are unaltered, if the integrand is multiplied by
$\rho_2(\boldsymbol{\beta})$.
\end{lem}

Furthermore, recall from \S \ref{sctn:kircharform} that 
$\mathrm{d}^HV_{\mathfrak{g}}$ is the Lebesgue 
measure on $\mathfrak{g}$ associated to $\kappa^H$.
Its relation to the Lebesgue measure 
$\mathrm{d}^\kappa V_{\mathfrak{g}}$
associated to $\kappa$ 
is given by
$$
\mathrm{d}^HV_{\mathfrak{g}}(\boldsymbol{\xi})=
\frac{1}{\mathrm{vol}^\kappa (G)}\,
\mathrm{d}^\kappa V_{\mathfrak{g}}(\boldsymbol{\xi}).
$$
Furthermore, (\ref{eqn:kir char form rescaled transl}) with $\boldsymbol{\xi}=\mathbf{0}$
yields
\begin{eqnarray}
\label{eqn:kir char form rescaled transl 0}
d_{k\,\boldsymbol{\lambda}}
&=&\left(\frac{k}{2\,\pi}\right)^{n_G}\,
\int_{\mathcal{O}_{\boldsymbol{\lambda}}}
\,\mathcal{V}_k(\boldsymbol{\beta})\,
\mathrm{d}
V_{\mathcal{O}_{\boldsymbol{\lambda}}}
(\boldsymbol{\beta})\\
&=&\left(\frac{k}{2\,\pi}\right)^{n_G}\,
\mathrm{vol}(\mathcal{O}_{\boldsymbol{\lambda}})\,\left[1+O\left(k^{-1}\right)
\right],
\nonumber
\end{eqnarray}
the previous expression being a polynomial in $k$.
Summing up, we can rewrite 
(\ref{eqn:Piklambda g expanded resc resc 12})
in the following form:
\begin{eqnarray}
\label{eqn:Piklambda g expanded resc resc 123}
\Pi^\tau_{k\,\boldsymbol{\lambda}}(x_{1k},x_{2k})&\sim&
k^{1-d_G/2}\,
\left(\frac{k}{2\,\pi}\right)^{2\,n_G}\,\frac{\mathrm{vol}(\mathcal{O}_{\boldsymbol{\lambda}})}{\mathrm{vol}^\kappa (G)} \,
\int_{\mathfrak{g}}\,
\mathrm{d}^\kappa V_{\mathfrak{g}}(\boldsymbol{\xi})\,\int_{\mathcal{O}_{\boldsymbol{\lambda}}}\,
\mathrm{d}
V_{\mathcal{O}_{\boldsymbol{\lambda}}}(\boldsymbol{\beta})
\,\int_0^{+\infty}\,\mathrm{d}u
\nonumber
\\
&&\left[
 e^{\imath\,\sqrt{k}\,
\Psi
(u,\boldsymbol{\xi},\boldsymbol{\beta})
}\,e^{u\,\Lambda (\boldsymbol{\xi})}\,
\rho_2(\boldsymbol{\beta})\,
\mathcal{D}_{x,k}\left(u,\boldsymbol{\beta},\boldsymbol{\xi}\right)\right].
\end{eqnarray}

In (\ref{eqn:Piklambda g expanded resc resc 123}), 
integration in $\boldsymbol{\beta}$ has been restricted to a small
neighbourhood of $\boldsymbol{\lambda}$. To proceed, we need a
parametrization of $\mathcal{O}_{\boldsymbol{\lambda}}$
near $\boldsymbol{\lambda}$. Consider the smooth map
\begin{equation}
\label{eqn:defn di f gamma lambda}
f:\boldsymbol{\gamma}\in \mathfrak{t}^{\perp_{\kappa_e}}\mapsto
e^{\boldsymbol{\gamma}}\cdot \boldsymbol{\lambda}:=
\mathrm{Coad}_{e^{\boldsymbol{\gamma}}}(\boldsymbol{\lambda})\in 
\mathcal{O}_{\boldsymbol{\lambda}}.
\end{equation}

\begin{lem}
Let $\mathrm{d}^{\kappa}\boldsymbol{\gamma}$ denote the Lebesgue measure
on $\mathfrak{t}^{\perp_{\kappa_e}}$ associated 
to the restriction of $\kappa_e$.
Then
$$
f^*(\mathrm{d}V_{\mathcal{O}_{\boldsymbol{\lambda}}})=
\mathrm{vol}(\mathcal{O}_{\boldsymbol{\lambda}})\,\frac{\mathrm{vol}^\kappa(T)}{\mathrm{vol}^\kappa(G)}\,
\mathcal{R} (\boldsymbol{\gamma})\,\mathrm{d}^\kappa\boldsymbol{\gamma},
$$
where $\mathcal{R}
\in \mathcal{C}^\infty(\mathfrak{t}^{\perp_{\kappa_e}})$
and $\mathcal{R}(0)=1$. Furthermore, on a suitably small open neighbourhood $V$ of $0\in 
\mathfrak{t}^{\perp_{\kappa_e}}$, $f$ restricts to a diffeomorphism
onto its image, which is an open neighbourhood of $\boldsymbol{\lambda}$
in $\mathcal{O}_{\boldsymbol{\lambda}}$.
\end{lem}

The proof was given in the discussion following Lemma 7.4 in 
\cite{p22} (with some notational differences), 
but we reproduce it here for the reader's convenience. 

Let $\mathrm{d}^\kappa V_{G/T}$ be the volume density induced by
$\kappa$ on $G/T$. Similarly, let $\mathrm{vol}^H(G/T)$ be the Haar volume density.
Then
$$
\mathrm{vol}^\kappa (G/T)=
\frac{\mathrm{vol}^\kappa (G)}{\mathrm{vol}^\kappa (T)},\quad 
\mathrm{d}^\kappa V_{G/T}=\mathrm{vol}^\kappa (G/T)\,
\mathrm{d}^H V_{G/T}
$$

\begin{proof}
If the first statement holds, then $f$ is a local diffeomorphism at
the origin, and the second statement follows.
To verify the former statement, 
let $f':G/T\rightarrow \mathcal{O}_{\boldsymbol{\lambda}}$ be given by
$$
f'(g\,T):=g\cdot \boldsymbol{\lambda}:=
\mathrm{Coad}_g(\boldsymbol{\lambda}).
$$
Then $f'$ is an equivariant diffeomorphism and 
$$
{f'}^*(\mathrm{d}V_{\mathcal{O}_{\boldsymbol{\lambda}}})=
\mathrm{vol}(\mathcal{O}_{\boldsymbol{\lambda}})\,
\mathrm{d}V^H_{G/H}=
\frac{\mathrm{vol}(\mathcal{O}_{\boldsymbol{\lambda}})}{\mathrm{vol}^\kappa
(G/T)}\,
\mathrm{d}^\kappa V _{G/H}=
\frac{\mathrm{vol}(\mathcal{O}_{\boldsymbol{\lambda}})\,\mathrm{vol}^\kappa
(T)}{\mathrm{vol}^\kappa
(G)}\,
\mathrm{d}V^\kappa _{G/H}.
$$

Furthermore, 
define $f'':\mathfrak{t}^{\perp_\kappa}\rightarrow G/T$ by
$
f''(\boldsymbol{\gamma}):=e^{\boldsymbol{\gamma}}\,T
$.
Then
$$
{f''}^*(\mathrm{d}^\kappa V_{G/T})=\mathcal{R}(\boldsymbol{\gamma})\,
\mathrm{d}^\kappa\boldsymbol{\gamma}, \quad\text{where}\quad 
\mathcal{R}(0)=1.
$$
Furthermore, $f=f'\circ f''$. Thus
\begin{eqnarray*}
\lefteqn{
f^*(\mathrm{d}V_{\mathcal{O}_{\boldsymbol{\lambda}}})={f''}^*\left(
{f'}^*(\mathrm{d}V_{\mathcal{O}_{\boldsymbol{\lambda}}})\right)}\\
&=&
\frac{\mathrm{vol}(\mathcal{O}_{\boldsymbol{\lambda}})\,\mathrm{vol}^\kappa
(T)}{\mathrm{vol}^\kappa
(G)}\cdot 
{f''}^*\left(
\mathrm{d}V^\kappa _{G/H}\right)
=
\frac{\mathrm{vol}(\mathcal{O}_{\boldsymbol{\lambda}})\,\mathrm{vol}^\kappa
(T)}{\mathrm{vol}^\kappa
(G)}\cdot 
\mathcal{R}(\boldsymbol{\gamma})\,
\mathrm{d}^\kappa\boldsymbol{\gamma}
.
\end{eqnarray*}

\end{proof}

In the following, $V\subset \mathfrak{t}^{\perp_\kappa}$ will denote a neighbourhood of the origin such that $f$ induces a diffeomorphism
$V\cong f(V)$. We may assume without loss that 
$\mathrm{supp}(\rho_2)\subset V$. Hence (\ref{eqn:Piklambda g expanded resc resc 123}) may be rewritten as follows:
\begin{eqnarray}
\label{eqn:Piklambda g expanded resc resc 1234}
\Pi^\tau_{k\,\boldsymbol{\lambda}}(x_{1k},x_{2k})
&\sim&
k^{1-d_G/2}\,
\left(\frac{k}{2\,\pi}\right)^{2\,n_G}\,
\left(\frac{\mathrm{vol}(\mathcal{O}_{\boldsymbol{\lambda}})}{\mathrm{vol}^\kappa (G)}\right)^2\cdot 
\mathrm{vol}^\kappa
(T)
\\
&& \cdot 
\int_{\mathfrak{g}}\,
\mathrm{d}^\kappa V_{\mathfrak{g}}(\boldsymbol{\xi})\,\int_{\mathfrak{t}^{\perp_{\kappa_e}}}\,
\mathrm{d}^\kappa\boldsymbol{\gamma}\,\int_0^{+\infty}\,\mathrm{d}u
\nonumber\\
&&
\left[
 e^{\imath\,\sqrt{k}\,
\Psi
(u,\boldsymbol{\xi},e^{\boldsymbol{\gamma}}\cdot \boldsymbol{\lambda})
}\,e^{u\,\Lambda (\boldsymbol{\xi})}\,
\mathcal{R}(\boldsymbol{\gamma})\,
\rho_2(e^{\boldsymbol{\gamma}}\cdot \boldsymbol{\lambda})\,
\mathcal{D}_{x,k}\left(u,e^{\boldsymbol{\gamma}}\cdot \boldsymbol{\lambda},\boldsymbol{\xi}\right)\right].
\nonumber
\end{eqnarray}

Recalling (\ref{eqn:defn di f gamma lambda}),
for $\boldsymbol{\gamma}\sim \mathbf{0}$ we have
$$
\left(e^{\boldsymbol{\gamma}}\cdot 
\boldsymbol{\lambda}\right)^\kappa=
\mathrm{Ad}_{e^{\boldsymbol{\gamma}}}(\boldsymbol{\lambda}^\kappa)=
\boldsymbol{\lambda}^\kappa+
\left[\boldsymbol{\gamma},\boldsymbol{\lambda}^\kappa \right]
+R_2(\boldsymbol{\gamma}).
$$
Hence, in view of (\ref{eqn:Psidefnxi}),
\begin{eqnarray}
\label{eqn:Psidefnxi1}
\Psi
\left(u,\boldsymbol{\xi},e^{\boldsymbol{\gamma}}\cdot \boldsymbol{\lambda}\right)
&=&u\,\left(\theta_1-\theta_2+\tau\,b_{\boldsymbol{\xi}} \right)
-\kappa_e\left(
\left(e^{\boldsymbol{\gamma}}\cdot 
\boldsymbol{\lambda}\right)^\kappa,
\boldsymbol{\xi}\right)\\
&=&u\,\left(\theta_1-\theta_2+\tau\,b_{\boldsymbol{\xi}} \right)
-\kappa_e\left(
\boldsymbol{\lambda}^\kappa+
\left[\boldsymbol{\gamma},\boldsymbol{\lambda}^\kappa \right]
+R_2(\boldsymbol{\gamma}),
\boldsymbol{\xi}\right).\nonumber
\end{eqnarray}
Recalling (\ref{eqn:xi decomposed}),
$\kappa_e(\boldsymbol{\lambda}^\kappa,\boldsymbol{\xi})
=\|\boldsymbol{\lambda}\|\,b_{\boldsymbol{\xi}}$.
Let us decompose $\boldsymbol{\xi}$ according to
(\ref{eqn:xi decomposed}) and
(\ref{eqn:xi perp decomp 1}).
Then
\begin{eqnarray*}
-\kappa_e\left(
\left[\boldsymbol{\gamma},\boldsymbol{\lambda}^\kappa \right],
\boldsymbol{\xi}\right)&=&
\kappa_e\left(
\left[\boldsymbol{\lambda}^\kappa ,\boldsymbol{\gamma}\right],
\boldsymbol{\xi}\right)=-\kappa_e\left(\boldsymbol{\gamma},
\left[\boldsymbol{\lambda}^\kappa ,\boldsymbol{\xi}\right]
\right)=-\kappa_e\left(\boldsymbol{\gamma},
\left[\boldsymbol{\lambda}^\kappa ,\boldsymbol{\xi}''\right]
\right).
\end{eqnarray*}
Therefore, (\ref{eqn:Psidefnxi1}) may be rewritten
\begin{eqnarray}
\label{eqn:Psidefnxi2}
\lefteqn{\Psi
\left(u,\boldsymbol{\xi},e^{\boldsymbol{\gamma}}\cdot \boldsymbol{\lambda}\right)}\\
&=&u\,(\theta_1-\theta_2)+(u\,\tau
-\|\boldsymbol{\lambda}\|)\,b_{\boldsymbol{\xi}}
-\kappa_e\left(\boldsymbol{\gamma},
\left[\boldsymbol{\lambda}^\kappa ,\boldsymbol{\xi}''\right]
\right)
+\kappa_e\left(
R_2(\boldsymbol{\gamma}),
\boldsymbol{\xi}\right).\nonumber
\end{eqnarray}

Let us choose basis of
$\mathfrak{t}'_{\boldsymbol{\lambda}}$
and $\mathfrak{t}^{\perp_{\kappa_e}}$ that
are orthonormal with respect to $\kappa_e$,
so as to unitarily identify 
$\mathfrak{t}'_{\boldsymbol{\lambda}}
\cong \mathbb{R}^{r_G-1}$ and
$\mathfrak{t}^{\perp_{\kappa_e}}\cong 
\mathbb{R}^{2\,n_G}$. 
Together with $\boldsymbol{\lambda}^\kappa_u$, these 
form an orthonormal basis of $(\mathfrak{g},\kappa_e)$.

Accordingly, we shall replace $\boldsymbol{\xi}'\in 
\mathfrak{t}'_{\boldsymbol{\lambda}}$ with
$\mathbf{r}\in \mathbb{R}^{r_G-1}$,
$\boldsymbol{\xi}''\in \mathfrak{t}^{\perp_{\kappa_e}}$
with $\boldsymbol{\rho}\in \mathbb{R}^{2\,n_G}$,
and $\boldsymbol{\xi}^{\|}\in 
\mathrm{span}\left(\boldsymbol{\lambda}^\kappa\right)$
with $b=b_{\boldsymbol{\xi}}$ as in (\ref{eqn:xi decomposed}),
and substitute 
$$
\int_{\mathfrak{g}}\,
\mathrm{d}^\kappa V_{\mathfrak{g}}(\boldsymbol{\xi})\quad
\text{with}\quad 
\int_{\mathbb{R}^{r_G-1}}\,\mathrm{d}\mathbf{r}\,
\int_{\mathbb{R}^{2\,n_G}}\,\mathrm{d}\boldsymbol{\rho}
\,\int_{-\infty}^{+\infty}\,\mathrm{d}b.
$$
We shall simply identify $\boldsymbol{\gamma}
\in \mathfrak{t}^{\perp_{k_e}}$
with its image in $\mathbb{R}^{2\,n_G}$.

Let $Z_{\boldsymbol{\lambda}}$ be the $(2n_G)\times (2 n_G)$
nondegenerate skew-symmetric matrix representing the 
endomorphism $S_{\boldsymbol{\lambda}}:\mathfrak{t}^{\kappa_e}
\rightarrow \mathfrak{t}^{\kappa_e}$ in (\ref{eqn:inf adoint iso}).

Let us denote by $\cdot_{\mathrm{st}}$ the standard 
scalar product (on the appropriate Euclidean space);
in terms of these identifications, (\ref{eqn:Psidefnxi2})
may be rewritten
\begin{eqnarray}
\label{eqn:Psidefnxi23}
\lefteqn{\Psi_{\mathbf{r}}
\left(u,b,\boldsymbol{\rho},
\boldsymbol{\gamma} \right):=\Psi
\left(u,\boldsymbol{\xi},e^{\boldsymbol{\gamma}}\cdot \boldsymbol{\lambda}\right)
}\\
&=&u\,(\theta_1-\theta_2)+\left(u\,\tau
-\|\boldsymbol{\lambda}\|\right)\,b
-\boldsymbol{\gamma}\cdot_{\mathrm{st}}
Z_{\boldsymbol{\lambda}} \,\boldsymbol{\rho}
+
R_2(\boldsymbol{\gamma})\cdot_{\mathrm{st}}
(b,\mathbf{r},\boldsymbol{\rho}).\nonumber
\end{eqnarray}
With these replacements,
(\ref{eqn:Piklambda g expanded resc resc 1234}) becomes
\begin{eqnarray}
\label{eqn:Piklambda g expanded resc resc-12345}
\Pi^\tau_{k\,\boldsymbol{\lambda}}(x_{1k},x_{2k})
&\sim&
k^{1-d_G/2}\,
\left(\frac{k}{2\,\pi}\right)^{2\,n_G}\,
\left(\frac{\mathrm{vol}(\mathcal{O}_{\boldsymbol{\lambda}})}{\mathrm{vol}^\kappa (G)}\right)^2\cdot 
\mathrm{vol}^\kappa
(T)
\nonumber\\
&&\cdot \int_{\mathbb{R}^{r_G-1}}\,
I_{x,k}(\mathbf{r})\,\mathrm{d}\mathbf{r}
,
\end{eqnarray}
where 
\begin{eqnarray}
\label{eqn:defn di Ikr}
I_{x,k}(\mathbf{r})
&:=& \int_0^{+\infty}\,\mathrm{d}u
\,\int_{-\infty}^{+\infty}\,\mathrm{d}b
\,\int_{\mathbb{R}^{2\,n_G}}\,\mathrm{d}\boldsymbol{\rho}
\,\int_{\mathbb{R}^{2\,n_G}}\,
\mathrm{d}\boldsymbol{\gamma}
\nonumber\\
&&
\left[
 e^{\imath\,\sqrt{k}\,
\Psi_{\mathbf{r}}
\left(u,b,\boldsymbol{\rho},
\boldsymbol{\gamma} \right)}\,
\mathcal{K}_{\mathbf{r}}\left(u,b,\boldsymbol{\rho},
\boldsymbol{\gamma} \right)\right].
\end{eqnarray}
where 
$$
\mathcal{K}_{\mathbf{r}}\left(u,b,\boldsymbol{\rho},
\boldsymbol{\gamma} \right)
:=e^{u\,\Lambda (\boldsymbol{\xi})}\,\mathcal{R}(\boldsymbol{\gamma})\,
\rho_2(e^{\boldsymbol{\gamma}}\cdot \boldsymbol{\lambda})\,
\mathcal{D}_{x,k}\left(u,e^{\boldsymbol{\gamma}}\cdot \boldsymbol{\lambda},\boldsymbol{\xi}\right),
$$
where $\boldsymbol{\xi}$ is expressed in terms of 
$(b,\mathbf{r},\boldsymbol{\rho})$, and dependence of
$\mathcal{K}_r$ on $(\theta_j,\boldsymbol{\eta}_j,\boldsymbol{\rho}_j)$
is left implicit.

In view of (\ref{eqn:LambdacalD}) and (\ref{eqn:Lambdaxi defn}),
$\mathcal{K}_{\mathbf{r}}\left(u,b,\boldsymbol{\rho},
\boldsymbol{\gamma} \right)$ admits an asymptotic expansion of the
form
\begin{eqnarray}
\label{eqn:asympt Kr}
\mathcal{K}_{\mathbf{r}}\left(u,b,\boldsymbol{\rho},
\boldsymbol{\gamma} \right)&\sim&(k\,u)^{d-1}\,
e^{u\,\Lambda_2 (\boldsymbol{\xi})}\,\mathcal{R}(\boldsymbol{\gamma})\,
\rho_2(e^{\boldsymbol{\gamma}}\cdot \boldsymbol{\lambda})\,
\,\tilde{\rho}\left(k^{-\epsilon}\,\boldsymbol{\xi}\right)\,
s_0^\tau(x,x)\nonumber\\
&&\cdot \left[1+\sum_{j\ge 1}k^{-j/2}\,P_j(\theta_j,\boldsymbol{\eta}_j,\boldsymbol{\rho}_j;u,
\boldsymbol{\xi},\boldsymbol{\gamma})\right],
\end{eqnarray}
where each $P_j$ is a polynomial in the rescaled variables
$(\theta_j,\boldsymbol{\eta}_j,\boldsymbol{\rho}_j,
\boldsymbol{\xi})$ of degree $\le 3\,j$ and parity $j$,
and coefficients depending smoothly on the remaining variables
(cfr. the similar proof of Theorem 1.3 of \cite{p22},
or the arguments in the proof of Theorem 25 of \cite{gp24}).
Furthermore, 
$s_0^\tau(x,x)$ is given by (\ref{eqn:valore s0 xx}).

Integration is supported on an expanding ball centered
at the origin and of radius $O\left(k^\epsilon\right)$ 
in the variables
$(b,\mathbf{r},\boldsymbol{\rho})\in \mathbb{R}\times 
\mathbb{R}^{r_G-1}\times \mathbb{R}^{2\,n_G}$
(recall the last factor in the amplitude
$\mathcal{D}_{x,k}$ in (\ref{eqn:LambdacalD})), and 
compactly supported in 
$(u,\boldsymbol{\gamma})\in \mathbb{R}\times \mathbb{R}^{2\,n_G}$.
The expansion (\ref{eqn:asympt Kr}) can be integrated term by term.

\begin{lem}
\label{lem:critical point Psi}
$\Psi_{\mathbf{r}}$ has a unique critical point
$P_0$ on the
given domain of integration, given by 
$$
P_0=\left(u_0,b_0,\boldsymbol{\rho}_0,
\boldsymbol{\gamma}_0 \right)
:=\left(\frac{\|\boldsymbol{\lambda}\|}{\tau},
\frac{\theta_2-\theta_1}{\tau},\mathbf{0},\mathbf{0}     \right).
$$
The Hessian matrix
at the critical point
is given by
$$
H_{P_0}(\Psi_{\mathbf{r}})=
\begin{pmatrix}
0   &\tau & \mathbf{0}^t  & \mathbf{0}^t   \\
\tau&0    &  \mathbf{0}^t & \mathbf{0}^t    \\
\mathbf{0}&  \mathbf{0}  &[0] & Z_{\boldsymbol{\lambda}}   \\
\mathbf{0}&  \mathbf{0}^t &-Z_{\boldsymbol{\lambda}} &
\left.
\partial_{\boldsymbol{\lambda},\boldsymbol{\lambda}}^2
\Psi_{\mathbf{r}}\right|_{P_0}
\end{pmatrix},
$$ 
where $[0]$ denotes the $(2\,n_G)\times (2\,n_G)$ 
zero matrix.
Furthermore, 
$H_{P_0}(\Psi_{\mathbf{r}})$ has signature 
uqual to zero, and its determinant is
$$
\det \big(H_{P_0}(\Psi_{\mathbf{r}})\big)=
-\tau^2\,\det\left(Z_{\boldsymbol{\lambda}}\right)^2.
$$
In particular, the critical point is non degenerate. Finally, the critical
value is
$$
\Psi_{\mathbf{r}}(P_0)=
\frac{\|\boldsymbol{\lambda}\|}{\tau}\,(\theta_1-\theta_2).
$$
\end{lem}

\begin{proof}
By (\ref{eqn:Psidefnxi23}), 
\begin{equation}
\label{eqn:xi derv}
\partial_{\boldsymbol{\rho}}
\Psi_{\mathbf{r}}=Z_{\boldsymbol{\lambda}}\,
\boldsymbol{\gamma}+R_2(\boldsymbol{\gamma}).
\end{equation}
Since 
$Z_{\boldsymbol{\lambda}}$ is non-degenerate and
$\boldsymbol{\gamma}$ ranges in 
a small ball centered at the origin in $\mathbb{R}^{2\,n_G}$,
this forces $\boldsymbol{\gamma}=\mathbf{0}$ at any critical
point.
Given this, at any critical point we also need to have
$$
\partial_{\boldsymbol{\gamma}}\Psi_{\mathbf{r}}=-Z_{\boldsymbol{\lambda}}\,
\boldsymbol{\rho}+R_1(\boldsymbol{\gamma})\cdot_{\mathrm{st}}
(b,\mathbf{r},\boldsymbol{\rho})
=-Z_{\boldsymbol{\lambda}}\,
\boldsymbol{\rho} ,
$$
whence $\boldsymbol{\rho}=\mathbf{0}$.

Furthermore, 
$$
\partial_{b}\Psi_{\mathbf{r}}=u\,\tau-\|\boldsymbol{\lambda}\|,
\quad \partial_{u}\Psi_{\mathbf{r}}=\theta_1-\theta_2
+\tau\,b,
$$
which imply that at the critical point $u=\|\boldsymbol{\lambda}\|/\tau$ and
$b=(\theta_2-\theta_1)/\tau$.

The computation of $H_{P_0}(\Psi_{\mathbf{r}})$ and 
$\det \big(H_{P_0}(\Psi_{\mathbf{r}})\big)$ is straightforward.
Regarding the signature, multiplying $\left.
\partial_{\boldsymbol{\lambda},\boldsymbol{\lambda}}^2
\Psi_{\mathbf{r}}\right|_{P_0}$ by a factor $t\in [0,1]$
we obtain a homotopy of non-degenerate symmetric matrices,
which for $t=1$ is $H_{P_0}(\Psi_{\mathbf{r}})$, and
for $t=0$ is a matrix which is easily seen to have 
zero signature, since it is in block-diagonal form 
with each block having zero signature. 
The claim follows.

\end{proof}

Let $\varphi\in 
\mathcal{C}^\infty_c\left(\mathbb{R}^{2\,n_G}\right)$
be compactly supported in an open ball centered at the origin,
and identically equal to $1$ on some smaller neighbourhood
of $\mathbf{0}$. 

\begin{prop}
\label{prop:rho comp supp}
The asymptotics of (\ref{eqn:Piklambda g expanded resc resc-12345}) are unchanged, if the integrand is 
multiplied by $\varphi(\boldsymbol{\rho})$.
\end{prop}

\begin{proof} The Proposition will follow from
the following Lemma.

\begin{lem}
The asymptotics of (\ref{eqn:Piklambda g expanded resc resc-12345}) are unchanged, if the integrand is 
multiplied by $\varphi_k(\boldsymbol{\gamma}):=
\varphi\left(k^{1/3}\,\boldsymbol{\gamma}  \right)$.
\end{lem}

\begin{proof}[Proof of the Lemma]
The proof is by a standard argument based on 
\lq integration by parts\rq\, 
in $\boldsymbol{\rho}$, but we outline it 
for the reader's convenience.
By (\ref{eqn:xi derv}), where $\varphi_k\neq 1$ 
for $k\gg 0$ one has
$\|\partial_{\boldsymbol{\rho}}
\Psi_{\mathbf{r}}\|\ge D\,k^{-1/3}$ for some fixed $D>0$;
furthermore, $\partial_{\boldsymbol{\rho}}
\Psi_{\mathbf{r}}$ does not depend on $\boldsymbol{\rho}$.
Let us consider the differential operator (well defined for
$\boldsymbol{\gamma}\neq \mathbf{0}$)
$$
L:=\frac{1}{\|\partial_{\boldsymbol{\rho}}
\Psi_{\mathbf{r}}\|^2}\,
\sum_j (\partial_{\boldsymbol{\rho}_j}
\Psi_{\mathbf{r}})\,\partial_{\rho_j}.
$$
Then
$$
-\frac{\imath}{\sqrt{k}}\,
L\left(e^{\imath\,\sqrt{k}\,
\Psi_{\mathbf{r}}
\left(u,b,\boldsymbol{\rho},
\boldsymbol{\gamma} \right)}\right)=
 e^{\imath\,\sqrt{k}\,
\Psi_{\mathbf{r}}
\left(u,b,\boldsymbol{\rho},
\boldsymbol{\gamma} \right)}
$$
Thus, looking only at the integral in $\mathrm{d}\boldsymbol{\rho}$,
\begin{eqnarray}
\label{eqn:int by part rho}
\lefteqn{\int_{\mathbb{R}^{2\,n_G}}
\left[
 e^{\imath\,\sqrt{k}\,
\Psi_{\mathbf{r}}
\left(u,b,\boldsymbol{\rho},
\boldsymbol{\gamma} \right)}\,
\mathcal{K}_{\mathbf{r}}\left(u,b,\boldsymbol{\rho},
\boldsymbol{\gamma} \right)\right]\,\mathrm{d}
\boldsymbol{\rho}  }\\
&=&-\frac{\imath}{\sqrt{k}}\,
\frac{1}{\|\partial_{\boldsymbol{\rho}}
\Psi_{\mathbf{r}}\|^2}\,
\sum_j (\partial_{\boldsymbol{\rho}_j}
\Psi_{\mathbf{r}})\,
\int_{\mathbb{R}^{2\,n_G}}
\left[\partial_{\rho_j}\left(
 e^{\imath\,\sqrt{k}\,
\Psi_{\mathbf{r}}
\left(u,b,\boldsymbol{\rho},
\boldsymbol{\gamma} \right)}\right)\,
\mathcal{K}_{\mathbf{r}}\left(u,b,\boldsymbol{\rho},
\boldsymbol{\gamma} \right)\right]\,\mathrm{d}
\boldsymbol{\rho}\nonumber\\
&=&\frac{\imath}{\sqrt{k}}\,
\frac{1}{\|\partial_{\boldsymbol{\rho}}
\Psi_{\mathbf{r}}\|^2}\,
\sum_j (\partial_{\boldsymbol{\rho}_j}
\Psi_{\mathbf{r}})\,
\int_{\mathbb{R}^{2\,n_G}}
\left[
 e^{\imath\,\sqrt{k}\,
\Psi_{\mathbf{r}}
\left(u,b,\boldsymbol{\rho},
\boldsymbol{\gamma} \right)}\,
\partial_{\rho_j}\left(\mathcal{K}_{\mathbf{r}}\left(u,b,\boldsymbol{\rho},
\boldsymbol{\gamma} \right)\right)\right]\,\mathrm{d}
\boldsymbol{\rho}\nonumber
\end{eqnarray}
We have $\left|(\partial_{\boldsymbol{\rho}_j}
\Psi_{\mathbf{r}})/\|\partial_{\boldsymbol{\rho}}
\Psi_{\mathbf{r}}\|^2\right|\le C'\,k^{1/3}$ for some
$C'>0$.

On the other hand, in view of the exponential factors in (\ref{eqn:Lambdaxi defn}), the integrand in 
(\ref{eqn:Piklambda g expanded resc resc-12345})
may expanded as a linear combination of terms, each of which
is bounded by a product of the form
$k^N\,|\boldsymbol{\eta}^I|\,|P(\theta,u,\mathbf{r})|\,
\left|\tilde{\boldsymbol{\rho}}^J\right|\,
e^{-\frac{1}{2}\,\|\boldsymbol{\rho}\|^2}$, where $N$ is
uniformly bounded from above over all the summands,
$I$ and $J$ are multi-indexes, and 
$\tilde{\boldsymbol{\rho}}
:=
\Re(\boldsymbol{\eta}_1-\boldsymbol{\eta}_2)-\boldsymbol{\rho}$;
recall that $\boldsymbol{\rho}\in \mathbb{R}^{2\,n_G}$
is the coordinate expression for $\boldsymbol{\xi}_s$
and $\mathbf{r}\in \mathbb{R}^{r_G-1}$ is the one for
$\boldsymbol{\xi}_t$ in (\ref{eqn:xitsperp}).
Iteratively integrating by parts as in (\ref{eqn:int by part rho}) $r$ times, each such term is trasnformed in 
a linear combination of terms, each of which
is bounded by an expression of the form
$k^{N-r/6}\,|\boldsymbol{\eta}^I|\,|P(\theta,u,\mathbf{r})|\,
\left|\tilde{\boldsymbol{\rho}}^{J'}\right|\,
e^{-\frac{1}{2}\,\|\boldsymbol{\rho}\|^2}$,
where $|J'|=|J|+r$. 
The claim follows.

\end{proof}
Let us conclude the proof of the Proposition.
By the Lemma,
we may assume without loss that the integrand has been
multiplied by $\varphi_k(\boldsymbol{\gamma})$.
Where $1-\varphi(\boldsymbol{\rho})\neq 0$,
we have $\|\boldsymbol{\rho}\|\ge C'$ for some
$C'>0$. Then on the domain of integration
for some constants $C'',\,C'''>0$
we have
$$
\|\partial_{\boldsymbol{\gamma}}\Psi_{\mathbf{r}}\|=
\|Z_{\boldsymbol{\lambda}}\,
\boldsymbol{\rho}+R_1(\boldsymbol{\gamma})\cdot_{\mathrm{st}}
(b,\mathbf{r},\boldsymbol{\rho})\|
\ge C''\,\|\boldsymbol{\rho}\|
+O\left(k^{-\frac{1}{3}+\epsilon}  \right)
\ge C'''.
$$
The statement follows by integration by parts in 
$\mathrm{d}\boldsymbol{\gamma}$, by a modification of
the previous argument.
\end{proof}

By Proposition \ref{prop:rho comp supp}, in the 
asymptotic evaluation of (\ref{eqn:defn di Ikr}) we may
assume without loss that all variables are compactly supported,
hence we are in a position to apply the Stationary Phase Lemma.
On the critical locus, $b=b_0=(\theta_2-\theta_1)/\tau$ 
and 
$\boldsymbol{\rho}=\mathbf{0}$. 
Hence $\boldsymbol{\xi}''=\boldsymbol{\xi}_{\mathfrak{s}}=\mathbf{0}$
(Remark \ref{rem:xii=implisxis=0}),
thus 
$\boldsymbol{\xi}^\perp=\boldsymbol{\xi}_{\mathfrak{t}}
=\boldsymbol{\xi}'$.
Using coordinates with respect to the given orthonormal
basis with respect to $\kappa_e$, we shall identify
$\boldsymbol{\xi}_{\mathfrak{t}}$ with $\mathbf{r}\in \mathbb{R}^{r_G-1}$.
Let us set
\begin{eqnarray}
\label{eqn:Lambdaxi defn r}
L_x(\mathbf{r})
&:=&
\imath\,
\tilde{\kappa}_x \Big(
\mathbf{r}_{\tilde{G}}(x)
,
{\boldsymbol{\rho}_1}_{\tilde{G}}(x)
+{\boldsymbol{\rho}_2}_{\tilde{G}}(x)\Big)
%
%
-\frac{1}{2}\,\big\|\mathbf{r}_{\tilde{G}}(x)
%
%
\big\|^2      .
\nonumber
\end{eqnarray}
Recalling (\ref{eqn:Lambdaxi defn}), we have
\begin{eqnarray}
\label{eqn:Ikr stationary}
\lefteqn{
I_{x,k}(\mathbf{r})=I_{x,k}(\theta_1,\theta_2,\boldsymbol{\eta}_1,
\boldsymbol{\eta}_2,\boldsymbol{\rho}_1,
\boldsymbol{\rho}_2;\mathbf{r})}\nonumber\\
&\sim&
e^{\imath\,\sqrt{k}\,\frac{\|\boldsymbol{\lambda}\|}{\tau}\,(\theta_1-\theta_2)}\,
\left( \frac{2\,\pi}{\sqrt{k}}  \right)^{1+2\,n_G}\,
\frac{1}{\tau\,|\det\left(Z_{\boldsymbol{\lambda}}\right)   |}
\cdot \left(\frac{k\,\|\boldsymbol{\lambda}\|}{\tau}\right)^{d-1}
\nonumber\\
&&\cdot e^{\frac{\|\boldsymbol{\lambda}\|}{\tau}\,\left[
\psi_2\big( {\boldsymbol{\eta}_1}_{\tilde{G}}(x)
,
{\boldsymbol{\eta}_2}_{\tilde{G}}(x)  \big)
-\frac{1}{2}\,\big\|
{\boldsymbol{\rho}_1}_{\tilde{G}}(x)-{\boldsymbol{\rho}_2}_{\tilde{G}}(x)\big\|^2
\right]}\cdot 
e^{\frac{\|\boldsymbol{\lambda}\|}{\tau}\,L_x (\mathbf{r})}\nonumber\\
&&\cdot\tilde{\rho}\left(k^{-\epsilon}\,\mathbf{r}\right)\,
\frac{\tau}{(2\,\pi)^d}
\cdot \left[1+\sum_{j\ge 1}k^{-j/2}\,F_j(\theta_1,\theta_2,\boldsymbol{\eta}_1,
\boldsymbol{\eta}_2,\boldsymbol{\rho}_1,
\boldsymbol{\rho}_2;\mathbf{r})    \right]\nonumber\\
&=&(2\,\pi)^{1+2\,n_G-d}\,
\frac{1}{|\det\left(Z_{\boldsymbol{\lambda}}\right)   |}
\, e^{\imath\,\sqrt{k}\,\frac{\|\boldsymbol{\lambda}\|}{\tau}\,(\theta_1-\theta_2)}\,k^{d-n_G-\frac{3}{2}}\,
\left(\frac{\|\boldsymbol{\lambda}\|}{\tau}\right)^{d-1}
\cdot\tilde{\rho}\left(k^{-\epsilon}\,\mathbf{r}\right)\nonumber\\
&&
\cdot e^{\frac{\|\boldsymbol{\lambda}\|}{\tau}\,\left[
\psi_2\big( {\boldsymbol{\eta}_1}_{\tilde{G}}(x)
,
{\boldsymbol{\eta}_2}_{\tilde{G}}(x)  \big)
-\frac{1}{2}\,\big\|
{\boldsymbol{\rho}_1}_{\tilde{G}}(x)-{\boldsymbol{\rho}_2}_{\tilde{G}}(x)\big\|^2
\right]}\cdot 
e^{\frac{\|\boldsymbol{\lambda}\|}{\tau}\,L_x (\mathbf{r})}
\nonumber\\
&&\cdot \left[1+\sum_{j\ge 1}k^{-j/2}\,F_j(\theta_1,\theta_2,\boldsymbol{\eta}_1,
\boldsymbol{\eta}_2,\boldsymbol{\rho}_1,
\boldsymbol{\rho}_2;\mathbf{r})    \right],\nonumber 
\end{eqnarray}
where $L_x(\mathbf{r})=\Lambda_2(\boldsymbol{\xi})$ with
$\boldsymbol{\xi}$ corresponding to $(b_0,\mathbf{r},\mathbf{0})
\in \mathbb{R}\times \mathbb{R}^{r_G-1}\times \mathbb{R}^{2\,n_G}$,
and $F_j$ is a polynomial of degree $\le 3\,j$ and parity $j$.
In view of (\ref{eqn:Piklambda g expanded resc resc-12345}),
integrating the previous asymptotic expansion
term by term yields
an asymptotic expansion for 
$\Pi^\tau_{k\,\boldsymbol{\lambda}}(x_{1k},x_{2k})$.

Let us consider the leading order term.

Integration in (\ref{eqn:Piklambda g expanded resc resc-12345}) is with respect to the standard measure associated to the Euclidean structure of
$\mathfrak{t}'_{\boldsymbol{\lambda}}\cong \mathbb{R}^{r_G-1}$ induced
by $\kappa_e$; on the other hand, the norms and products in 
(\ref{eqn:Lambdaxi defn r}) are with respect 
to the Euclidean product 
on ${\mathfrak{t}'_{\boldsymbol{\lambda}}}_{\tilde{G}}(x)$ induced by
$\tilde{\kappa}_x$. 
Let
$D_x$ be the matrix representing the pull-back of $\tilde{\kappa}_x$
(under the evaluation $\mathfrak{t}'_{\boldsymbol{\lambda}}
\rightarrow {\mathfrak{t}'_{\boldsymbol{\lambda}}}_{\tilde{G}}(x)$)
with respect to an orthonormal basis of 
$\mathfrak{t}'_{\boldsymbol{\lambda}}$.
If we set 
$\mathbf{a}:=\mathbf{r}_{\tilde{G}}(x)$, then 
$
\mathrm{d}\mathbf{r}=\det(D_x)^{-1/2}\,
\mathrm{d}\mathbf{a}=\mathfrak{D}^\kappa (x)^{-1}\,\mathrm{d}\mathbf{a}$.
Hence
\begin{eqnarray}
\lefteqn{
\int_{\mathbb{R}^{r_G-1}}\,
e^{\frac{\|\boldsymbol{\lambda}\|}{\tau}\,L_x (\mathbf{r})}\,\mathrm{d}\mathbf{r} }  \\
&=&\frac{1}{\sqrt{\det(D_x)}}\,
\int_{\mathbb{R}^{r_G-1}}\,
e^{\frac{\|\boldsymbol{\lambda}\|}{\tau}\,
\left[ \imath\,\langle\mathbf{a}, {\boldsymbol{\rho}_1}_{\tilde{G}}(x)
+{\boldsymbol{\rho}_2}_{\tilde{G}}(x)\rangle
-\frac{1}{2}\,\|\mathbf{a}\|^2\right]
}\,\mathrm{d}\mathbf{a}  \nonumber\\
&=&\frac{1}{\mathfrak{D}^\kappa (x)}\,
\left(\frac{2\,\pi\,\tau}{\| \boldsymbol{\lambda} \|} \right)^{\frac{r_G-1}{2}}
\,e^{-\frac{1}{2}\,\frac{\|\boldsymbol{\lambda}\|}{\tau}\,
\big\|
{\boldsymbol{\rho}_1}_{\tilde{G}}(x)+{\boldsymbol{\rho}_2}_{\tilde{G}}(x)\big\|^2
}.\nonumber
\end{eqnarray}
Inserting this in (\ref{eqn:Piklambda g expanded resc resc-12345}),
we obtain an asymptotic expansion of the form
\begin{eqnarray}
\label{eqn:Piklambda g expanded resc resc-012345int}
\Pi^\tau_{k\,\boldsymbol{\lambda}}(x_{1k},x_{2k})
&\sim&
\left( \frac{k\,\| \boldsymbol{\lambda} \|}{2\pi\,\tau}  \right)
^{d-1+\frac{1-r_G}{2}}\,
\left(\frac{\mathrm{vol}(\mathcal{O}_{\boldsymbol{\lambda}})}{\mathrm{vol}^\kappa (G)}\right)^2\cdot 
\frac{\mathrm{vol}^\kappa
(T)}{\mathfrak{D}^\kappa (x)\cdot|\det\left(Z_{\boldsymbol{\lambda}}\right)   |}\nonumber
\\
&&\cdot e^{\frac{\|\boldsymbol{\lambda}\|}{\tau}\,\left[
\psi_2\big( {\boldsymbol{\eta}_1}_{\tilde{G}}(x)
,
{\boldsymbol{\eta}_2}_{\tilde{G}}(x)  \big)
-\big\|
{\boldsymbol{\rho}_1}_{\tilde{G}}(x)\big\|^2
-\big\|{\boldsymbol{\rho}_2}_{\tilde{G}}(x)\big\|^2
\right]}\nonumber\\
&&\cdot \left[1+\sum_{j\ge 1}k^{-j/2}\,R_j(\theta_1,\theta_2,\boldsymbol{\eta}_1,
\boldsymbol{\eta}_2,\boldsymbol{\rho}_1,
\boldsymbol{\rho}_2)    \right],\nonumber
\end{eqnarray}
where the $R_j$'s are as stated.

The argument for
 $P^\tau_{k\,\boldsymbol{\lambda}}(x_{1k},x_{2k})$ is essentially the same,
 in view of the considerations in \S \ref{scnt:szego parametrix}.
However, by (\ref{eqn:asy exp Ptau}) the leading order term in the amplitude
of the FOI representation of $P^\tau$ is the the one for $\Pi^\tau$
multiplied by $(\pi/u)^{(d-1)/2}$.
Since the previous arguments involve the rescaling $u\mapsto k\,u$,
to leading order
there is an additional factor $(\pi/k\,u)^{(d-1)/2}$ in the asymptotic expansion corresponding to (\ref{eqn:asympt Kr}). 
Evaluating at the critical point
of Lemma \ref{lem:critical point Psi}, we obtain
an asymptotic expansion formally similar to the one for $\Pi^\tau_{k\,\boldsymbol{\lambda}}(x_{1k},x_{2k})$, but multiplied by 
$(\tau\,\pi/k\,\|\boldsymbol{\lambda}\|)^{(d-1)/2}=
(k\,\|\boldsymbol{\lambda}\|/\tau\,\pi)^{-(d-1)/2}$. 
\end{proof}

\section{Proofs of the applications}

\subsection{Theorem \ref{thm:Husimi equiv}}

The proof of Theorem \ref{thm:Husimi equiv} rests on the previous
asymptotic expansions and on an $L^2$-norm
asymptotic estimate for restrictions of complexified eigenfunctions, 
which is the specialization in the present setting of a basic result of Zelditch 
(Lemma 0.2 in \cite{z20}).

\begin{lem} (Zelditch)
\label{lem:L2 norm est zel}
There exists a universal constant $C(d,\tau)>0$ such that the following holds.
Let $\varphi\in L^2(G)_{k\,\boldsymbol{\lambda}}$ 
have unit $L^2$-norm. Let $\tilde{\varphi}$ be its complexification and
$\tilde{\varphi}^\tau:=\left.\tilde{\varphi}\right|_{X^\tau}$.
Then
$$
\left\|\tilde{\varphi}^\tau\right\|^2_{L^2(X^\tau)}
=C(d,\tau)\,e^{2\,\tau\,c_{k\,\boldsymbol{\lambda}}}\,(c_{k\,\boldsymbol{\lambda}})^{-\frac{d-1}{2}}\cdot 
\left(1+O\left(\frac{1}{k\,\|\boldsymbol{\lambda}\|}\right)\right).
$$
\end{lem}

\begin{proof}
[Proof of Theorem \ref{thm:Husimi equiv}]
Let $\varphi\in L^2(G)_{k\,\boldsymbol{\lambda}}$  
have unit $L^2$-norm. Suppose $x\in X^\tau_{\mathcal{O}}$,
choose a system of NHLC's on $X^\tau$ centered at $x$,
and let $\mathbf{n}\in N(X^\tau_{\mathcal{O}}/X^\tau)_x$ be
of norm $C\,k^{\epsilon}$. 
Since $\varphi$ can be 
extended to an $L^2$-orthonormal basis of 
$L^2(G)_{k\,\boldsymbol{\lambda}}$, by 
(\ref{eqn:Poisson-wave equiv}) 
$$
e^{-2\,\tau\,c_{k\,\boldsymbol{\lambda}}}\,\left|
\tilde{\varphi}^\tau\left( y\right)\right|^2
\le P^\tau_{k\,\boldsymbol{\lambda}} 
\left( y,y
\right)\quad \forall\,y\in X^\tau.
$$

Choose $C,\epsilon$ as in the statement of Theorem 
\ref{thm:rapid decay moment map}.

In view of Theorem \ref{thm:rapid decay moment map},
we conclude that
\begin{equation}
\label{eqn:rapid decrease single eig}
e^{-2\,\tau\,c_{k\,\boldsymbol{\lambda}}}\,\left|
\tilde{\varphi}^\tau\left( y\right)\right|^2=O\left( k^{-\infty} \right),
\end{equation}
uniformly for $\mathrm{dist}_{X^\tau},
\left(y,X^\tau_{\mathcal{O}}\right)
\ge \,C\,k^{\epsilon-\frac{1}{2}}$.

On the other hand, any point $y$ in a tubular neighbourhood or radius
$O\left( k^{\epsilon-\frac{1}{2}} \right)$ of $X^\tau_{\mathcal{O}}$
can be written in the form 
\begin{equation}
\label{eqn:y locl param}
y=x+\frac{\mathbf{n}}{\sqrt{k}}
\end{equation}
with $x\in X^\tau_{\mathcal{O}}$ and 
$\mathbf{n}\in N(X^\tau_{\mathcal{O}}/X^\tau)_x$ of norm 
$O\left(k^\epsilon\right)$, for some choice of NHLC's centered at $x$.
Since we may locally smoothly vary systems of NHLC's centered at 
moving points in $X^\tau_{\mathcal{O}}$, this is indeed 
a local parametrization
of a shrinking neighbourhood of $X^\tau_{\mathcal{O}}$.

In view of Theorem \ref{thm:rescaled asympt},
we obtain that for certain constants 
$C_{d,\tau}',\,C_{d,\tau}''>0$
\begin{eqnarray}
\label{eqn:estimate Linfty norm}
e^{-2\,\tau\,c_{k\,\boldsymbol{\lambda}}}\,\left|
\tilde{\varphi}^\tau\left( x+\frac{\mathbf{n} }{\sqrt{k}}\right)\right|^2
&\le&P^\tau_{k\,\boldsymbol{\lambda}} 
\left( x+\frac{\mathbf{n} }{\sqrt{k}},
x+\frac{\mathbf{n} }{\sqrt{k}}
\right)\nonumber\\
&\le& C_{d,\tau}'\,\left( k\,\| \boldsymbol{\lambda} \| \right)
^{\frac{d-r_G}{2}}\,e^{-2\,\frac{\|\boldsymbol{\lambda}\|}{\tau}\,
\|\mathbf{n}  \|^2}\le 
C_{d,\tau}'\,\left( k\,\| \boldsymbol{\lambda} \| \right)
^{\frac{d-r_G}{2}}\nonumber\\
&\le&
C_{d,\tau}''\,
(c_{k\,\boldsymbol{\lambda}})
^{\frac{d-r_G}{2}},
\end{eqnarray}
since 
$c_{k\,\boldsymbol{\lambda}}\sim k\,\boldsymbol{\lambda}$
for $k\rightarrow +\infty$
by (\ref{eqn:eigenvalue of lambda}).
Pairing (\ref{eqn:rapid decrease single eig}) and
(\ref{eqn:estimate Linfty norm}), we conclude
that some constant $C_{d,\tau}'''>0$
$$
\left|
\tilde{\varphi}^\tau\left( x\right)\right|^2\le 
e^{2\,\tau\,c_{k\,\boldsymbol{\lambda}}}\,
C_{d,\tau}'''\,
{c_{k\,\boldsymbol{\lambda}}}
^{\frac{d-r_G}{2}},\quad \forall\,x\in X^\tau.
$$
This proves the first statement of Theorem \ref{thm:Husimi equiv}.

The second statement follows from the first and 
Lemma \ref{lem:L2 norm est zel}.

\end{proof}

\subsection{Theorem \ref{thm:operator norm norm estimate}}

Following arguments in \cite{sz03} and \cite{cr2},
we shall make recourse to the Shur-Young inequality (\cite{so}): 
given a Riemannian manifold
$(M,\beta)$ with Riemannian density $\mathrm{d}V_M$ and $q\ge p\ge 1$,
there is a constant $C_p>0$ such that
for any integral self-adjoint operator kernel $K$ on $M$ we have
$$
\|K\|_{L^p(M)\rightarrow L^q(M)}\le 
C_p\,\left[
\sup_{y\in M} 
\int\big|K(y,y')\big|^R\,
\mathrm{d}V_M(y')  \right ]^{\frac{1}{R}},\quad 
\frac{1}{R}:=1-\frac{1}{p}+\frac{1}{q}.
$$

\begin{proof}
[Proof of Theorem \ref{thm:operator norm norm estimate}]
We need to estimate 
\begin{equation}
\label{eqn:integral shur-young}
\sup_{y\in M} \left[
\int\big|\Pi^\tau_{k\,\boldsymbol{\lambda}}(y,y')\big|^R\,
\mathrm{d}V_M(y') \right] .
\end{equation}
Let us fix $C,\,\epsilon$ as in Theorem \ref{thm:rescaled asympt}.
By Theorem \ref{thm:rapid decay moment map}, 
$$
\Pi^\tau_{k\,\boldsymbol{\lambda}}(y,y')=O\left(
k^{-\infty}\right)
$$
uniformly for $\mathrm{dist}_{X^\tau}(y,X^\tau_{\mathcal{O}})
\ge C\,k^{\epsilon-\frac{1}{2}}$. We may thus reduce to the case
where $y$ is given by (\ref{eqn:y locl param}).
Given this, by Theorems \ref{thm:rapid decay moment map} and 
\ref{thm:rapid decay orbit}, a non-negligible contribution to
the integral in (\ref{eqn:integral shur-young}) only comes from the
locus where both the conditions
$$
\mathrm{dist}_{X^\tau}(y',X^\tau_{\mathcal{O}}),\,
\mathrm{dist}_{X^\tau}(y',G\cdot y)
\le C\,k^{\epsilon-\frac{1}{2}}
$$
are met. 
Thus, perhaps replacing $C$ with a bigger constant $C'$, we also have
$$\mathrm{dist}_{X^\tau}(y',G\cdot x)
\le C'\,k^{\epsilon-\frac{1}{2}}.
$$
Any such $y'$ has the form
\begin{equation}
\label{eqn:y'n's'}
y'=\mu_g(x)+\frac{1}{\sqrt{k}}\,(\mathbf{n}'+\mathbf{s}'),
\end{equation}
for suitable $\mathbf{n}'\in  N(X^\tau_{\mathcal{O}}/X^\tau)_{\mu_g(x)}$, 
$\mathbf{s}'\in \mathcal{S}_{\mu_g(x)}$ of norm 
$O\left(k^\epsilon\right)$.
The system of NHLC's at $\mu_g(x)$ may be taken to be the $\mu_g$-translate of the one at $x$.

However, in view of Remark
\ref{rem:NSsigma}, (\ref{rem:NSsigma}) is not a parametrization,
since the real summand $\mathfrak{s}_x(\mu_g(x))$ in
$\mathcal{S}_{\mu_g(x)}$ is contained in the
tangent space to the $G$-orbit. To obtain an effective parametrization, we
restrict $\mathbf{s}'$ to be an imaginary vector,
i.e. assume
$\mathbf{s}'\in J_{\mu_g(x)}\big( \mathfrak{s}_x(\mu_g(x)) \big)$.

Thus, up to a negligible contribution, for a suitable $D>0$ we may restrict integration to the
image in $X^\tau$ of the immersion 
\begin{eqnarray}
\label{eqn:Lambda_x}
\Lambda_{x,k}:&
(g,\mathbf{n}',\mathbf{s}')&\in 
G\times B_{r_G-1}\left(\mathbf{0},D\,k^\epsilon\right)
\times B_{d-r_G}\left(\mathbf{0},D\,k^\epsilon\right)\\
&\mapsto &
y'=\mu_g(x)+\frac{1}{\sqrt{k}}\,\big(\mathbf{n}'+J_{\mu_g(x)}(\mathbf{s}') \big)
\in 
X^\tau.\nonumber
\end{eqnarray}
Then 
\begin{equation}
\label{eqn:Lambda x volume form pb}
\Lambda_{x,k}^*(\mathrm{d}V_M)=
\frac{1}{k^{\frac{d-1}{2}}}\,\mathcal{V}_k(g,\mathbf{n}',
\mathbf{s}')\,\mathrm{d}^HV_G(g)
\,\mathrm{d}\mathbf{n}'
\,\mathrm{d}\mathbf{s}',
\end{equation}
where 
$$
\mathcal{V}_k(g,\mathbf{n}',
\mathbf{s}')\sim V_0(g)
+\sum_{j\ge 1}k^{-j/2}\,V_j(g,\mathbf{n}',
\mathbf{s}'),
$$
with $V_0(g)>0$ and $V_j$
homogeneous of degree $j$ in $(\mathbf{n}',\mathbf{s}')$.

On the other hand, by Theorem \ref{thm:rescaled asympt}
and the Cauchy-Schwartz inequality, 
if $y$ and $y'$ are given by (\ref{eqn:y locl param})
and (\ref{eqn:y'n's'}) 
we have 
\begin{eqnarray}
\label{eqn:bound on szego}
\big|\Pi^\tau_{k\,\boldsymbol{\lambda}}(y,y')\big|
&\le& \sqrt{\Pi^\tau_{k\,\boldsymbol{\lambda}}(y,y)\,
\Pi^\tau_{k\,\boldsymbol{\lambda}}(y',y')}\\
&\le& C_{\boldsymbol{\lambda},\tau}\,k^{d-1+\frac{1-r_G}{2}}\,
e^{-\frac{\|\boldsymbol{\lambda}\|}{2\,\tau}\,
\left(\|\mathbf{n}\|^2+\|\mathbf{n}'\|^2+\frac{1}{2}\,
\|\mathbf{s}'\|^2  \right)}\nonumber\\
&\le& 
C_{\boldsymbol{\lambda},\tau}\,k^{d-1+\frac{1-r_G}{2}}\,
e^{-\frac{\|\boldsymbol{\lambda}\|}{2\,\tau}\,
\|\mathbf{n}'\|^2}.
\nonumber
\end{eqnarray}
Hence, allowing the constant to vary from line to line,
\begin{equation}
\label{eqn:Pitau R bd}
\big|\Pi^\tau_{k\,\boldsymbol{\lambda}}(y,y')\big|^R\le 
C_{\boldsymbol{\lambda},\tau}\,
k^{R\,\left(d-1+\frac{1-r_G}{2}\right)}\,
e^{-R\,\frac{\|\boldsymbol{\lambda}\|}{2\,\tau}\,
\|\mathbf{n}'\|^2}.
\end{equation}
Using this and (\ref{eqn:Lambda x volume form pb}),
we conclude that
\begin{eqnarray}
\label{eqn:bound on integral Pitau}
\lefteqn{   
\int\big|\Pi^\tau_{k\,\boldsymbol{\lambda}}(y,y')\big|^R\,
\mathrm{d}V_M(y')}\\
&\le &
C_{\boldsymbol{\lambda},\tau}\,
k^{R\,\left(d-1+\frac{1-r_G}{2}\right)}\,
k^{-\frac{d-1}{2}}
\,k^{\epsilon\cdot (d-r_G)}.\nonumber
\end{eqnarray}
Thus we conclude that for some constant $C>0$ (depending
on $\tau$, $\boldsymbol{\lambda}$, $p$ and $q$), if 
$\epsilon'>0$ then for $k\gg 0$
we have
$$
\|\Pi^{\tau}_{k\,\boldsymbol{\lambda}}\|_{L^p(M)\rightarrow L^q(M)}\le C_{\boldsymbol{\lambda},\tau,}\,
k^{\frac{d-1}{2}\,\left( 1-\frac{1}{R}  \right)
+\frac{d-r_G}{2}+\epsilon'}.
$$

\end{proof}

\end{document}